\documentclass[12pt]{amsart} 
\RequirePackage[l2tabu,orthodox]{nag}
\usepackage[usenames]{color}
\usepackage{amssymb, amscd, amsmath, float, graphicx, longtable,
   enumerate, tabularx, 
  stmaryrd, fullpage,
  mathtools,todonotes, empheq }
\usepackage{fouriernc}
\usepackage[all]{xy}

\usepackage[pagebackref]{hyperref}
    \renewcommand*{\backref}[1]{}
    \renewcommand*{\backrefalt}[4]{%
    \ifcase #1 %
        (Not cited).
    \or
        (Cited on page~#2).%
    \else
        (Cited on pages~#2).
    \fi}

\DeclareMathOperator{\add}{{\mathsf{add}}}
\DeclareMathOperator{\ann}{Ann}

\DeclareMathOperator{\Aut}{Aut}
\DeclareMathOperator{\Bl}{\mathcal{B}}

\DeclareMathOperator{\coh}{{\mathsf{coh}}}

\DeclareMathOperator{\cok}{cok}

\DeclareMathOperator{\D}{{\mathsf{D}}}
\DeclareMathOperator{\depth}{depth}

\DeclareMathOperator{\End}{End}
\DeclareMathOperator{\Ext}{Ext}

\DeclareMathOperator{\GL}{GL}
\DeclareMathOperator{\GKdim}{GKdim}
\DeclareMathOperator{\gldim}{gldim}

\DeclareMathOperator{\GrMod}{\mathsf{GrMod}}
\DeclareMathOperator{\grmod}{\mathsf{grmod}}
\DeclareMathOperator{\gr}{gr}

\DeclareMathOperator{\Hilb}{Hilb}
\DeclareMathOperator{\Hom}{Hom}

\DeclareMathOperator{\injdim}{injdim}
\DeclareMathOperator{\image}{image}

\DeclareMathOperator{\lcm}{lcm}
\DeclareMathOperator{\Mat}{Mat}

\DeclareMathOperator{\ncProj}{\mathsf{Proj}}
\DeclareMathOperator{\nc}{nc}
\newcommand{\op}{{\operatorname{op}}}

\DeclareMathOperator{\pd}{pd}

\newcommand{\perf}{{\mathsf{perf}{}}}
\DeclareMathOperator{\proj}{{\mathsf{proj}}}
\DeclareMathOperator{\Proj}{{Proj}}
\DeclareMathOperator{\Qcoh}{{\mathsf{Qch}}}

\DeclareMathOperator{\rad}{rad}
\DeclareMathOperator{\rank}{rank}
\DeclareMathOperator{\rep}{{\mathsf{rep}}}

\DeclareMathOperator{\repdim}{{{repdim}}}
\DeclareMathOperator{\SL}{SL}
\DeclareMathOperator{\Spec}{\mathrm{Spec}}
\DeclareMathOperator{\Supp}{Supp}
\DeclareMathOperator{\Sym}{Sym}

\DeclareMathOperator{\tails}{\mathsf{tails}}
\DeclareMathOperator{\Tails}{\mathsf{Tails}}
\DeclareMathOperator{\Tor}{Tor}
\DeclareMathOperator{\Tors}{\mathsf{Tors}}
\DeclareMathOperator{\tors}{\mathsf{tors}}
\DeclareMathOperator{\tr}{tr}

\newcommand{\Lmod}[1]{#1\text{-}{\mathsf{Mod}}}

\newcommand{\lmod}[1]{#1\text{-}{\mathsf{mod}}}
\newcommand{\rmod}[1]{\mathsf{mod}\text{-}#1}

\newcommand{\fl}[1]{#1\text{-}{\mathsf{fl}}}

\renewcommand{\phi}{\varphi}

\renewcommand{\bar}{\overline}
\renewcommand{\tilde}{\widetilde}

\renewcommand{\leq}{\leqslant}

\renewcommand{\geq}{\geqslant}

\renewcommand{\to}{\longrightarrow}
\newcommand\into{{\hookrightarrow}}

\newcommand{\xto}{\xrightarrow}

\newcommand{\Wedge}{{\textstyle\bigwedge}}
\newcommand{\w}{\Wedge}
\renewcommand{\L}{{\mathbf{L}}}

\newcommand{\svee}{{\scriptscriptstyle \vee}}

\newcommand{\Gam}{\Gamma}
\newcommand{\Lam}{\Lambda}
\newcommand{\SG}{S\#G}

\newcommand{\p}{{\mathfrak{p}}}

\renewcommand{\r}{{\mathfrak{r}}}

\newcommand{\m}{{\mathfrak{m}}}

\newcommand{\K}{\mathsf{K}}

\newcommand{\CC}{{\mathbb C}}

\newcommand{\PP}{{\mathbb P}}
\newcommand{\QQ}{{\mathbb Q}}

\newcommand{\ZZ}{{\mathbb Z}}

\newcommand{\cata}{{\mathsf A}}
\newcommand{\catb}{{\mathsf B}}
\newcommand{\catc}{{\mathsf C}}

\newcommand{\catf}{{\mathsf F}}
\newcommand{\catg}{{\mathsf G}}

\newcommand{\catk}{{\mathsf K}}

\newcommand{\cala}{{\mathcal A}}

\newcommand{\cale}{{\mathcal E}}
\newcommand{\calf}{{\mathcal F}}
\newcommand{\calg}{{\mathcal G}}

\newcommand{\cali}{{\mathcal I}}

\newcommand{\call}{{\mathcal L}}
\newcommand{\calm}{{\mathcal M}}
\newcommand{\caln}{{\mathcal N}}
\newcommand{\calo}{{\mathcal O}}
\newcommand{\calp}{{\mathcal P}}

\newcommand{\calox}{{\calo_X}}
\newcommand{\caloy}{{\calo_Y}}
\newcommand{\caloz}{{\calo_Z}}
\newcommand{\calop}{{\calo_\PP}}
\newcommand{\F}{\mathcal{F{}}}

\def\cEnd{{\mathcal E}nd\,}

\def\cHom{{\mathcal H}\!om}
\newcommand{\rHom}[1][{}]{{\mathbf R}^{#1}\!\Hom}
\newcommand{\R}{{\mathbf R}}
\newcommand{\lotimes}{\overset{\mathbf{L}}{\otimes}}

\theoremstyle{plain}
\newtheorem{theorem}{Theorem}[section]

\newtheorem{prop}[theorem]{Proposition}
\newtheorem{proposition}[theorem]{Proposition}

\newtheorem{lemma}[theorem]{Lemma}

\newtheorem{conj}[theorem]{Conjecture}
\newtheorem{corollary}[theorem]{Corollary}
\newtheorem{question}[theorem]{Question}
\newtheorem{provdef}[theorem]{Provisional Definition}
\newtheorem*{theorem*}{Theorem}
\newtheorem*{prop*}{Proposition}

\theoremstyle{definition}

\newtheorem{definition}[theorem]{Definition}

\newtheorem*{conventions}{Conventions}

\newtheorem{example}[theorem]{Example}

\numberwithin{equation}{section}

\begin{document}

\title[]{%
Non-commutative crepant resolutions:\\ scenes from categorical geometry
}

\author[G.J. Leuschke]{Graham J. Leuschke}
\address{Mathematics Department, Syracuse University,
Syracuse NY 13244, USA}
\email{\href{mailto:gjleusch@math.syr.edu}{gjleusch@math.syr.edu}}
\urladdr{\href{http://www.leuschke.org}{www.leuschke.org}}


\thanks{The author was partly supported by NSF grant DMS-0902119.}

\date{\today}

\keywords{non-commutative algebraic geometry, categorical geometry,
  NCCR, non-commutative crepant resolution, McKay correspondence,
  minimal model program, tilting}

\subjclass[2010]{Primary: 
  13-02, 
  14A22, 
  14E15, 
  16S38; 
  Secondary:
  13C14, 
  13D09. 
}

\begin{abstract}
  Non-commutative crepant resolutions are algebraic objects defined by
  Van den Bergh to realize an equivalence of derived categories in
  birational geometry.  They are motivated by tilting theory, the
  McKay correspondence, and the minimal model program, and have
  applications to string theory and representation theory.  In this
  expository article I situate Van den Bergh's definition within these
  contexts and describe some of the current research in the area.
\end{abstract}
\maketitle

\begin{center}
\begin{minipage}{0.5\textwidth}
\tiny
\tableofcontents
\normalsize 
\end{minipage}
\end{center}

\section*{Introduction}\label{sect:intro}

A resolution of singularities replaces a singular algebraic variety by
a non-singular one that is isomorphic on a dense open set.  As such,
it is a great boon to the algebraic geometer, allowing the reduction
of many calculations and constructions to the case of a smooth
variety.  To the pure commutative algebraist, however, this process
can seem like the end of a story rather than the beginning: it
replaces a well-understood thing, the spectrum of a ring, with a much
more mysterious thing glued together out of other spaces.  Put simply,
a resolution of singularities of an affine scheme $\Spec R$ is almost
never another affine scheme (but see \S\ref{sect:normalization}).  One
cannot in general resolve singularities and stay within the categories
familiar to commutative algebraists.

The usual solution, of course, is to expand one's landscape on the
geometric side to include more complicated schemes.  There are plenty
of good reasons to do this other than resolving singularities, and it
has worked well for a century.  Locally, the more convoluted objects
are built out of affine schemes/commutative algebra, so one has not
strayed too far.  

Here is another alternative: expand the landscape on the algebraic
side instead, to include non-commutative rings as well as commutative
ones.  This suggestion goes by the name ``non-commutative algebraic
geometry'' or, my preference, ``categorical geometry''.  For some
thoughts on the terminology, see \S\ref{sect:non-existence}.  Whatever
the name, the idea is to treat algebraic objects, usually derived
categories, as coming from geometric objects even when no such
geometric things exist.  Since one trend in algebraic geometry in the
last forty years has been to study algebraic varieties indirectly, by
studying their (derived) categories of quasicoherent sheaves, one can
try to get along without the variety at all.  Given a category of
interest $\catc$, one can postulate a ``non-commutative'' space $X$
such that the (derived) category of quasicoherent sheaves on $X$ is
$\catc$, and write $\catc = \D^b(\Qcoh X)$.  In this game, the derived
category \emph{is} the geometry and the symbol $X$ simply stands in as
a grammatical placeholder; the mathematical object in play is $\catc$.

Of course, such linguistic acrobatics can only take you so far.  The
bonds between algebra and geometry cannot be completely severed: the
``non-commutative spaces'' must be close enough to the familiar
commutative ones to allow information to pass back and forth.  This
article is about a particular attempt to make this program work.

The idea of non-commutative resolutions of singularities appeared
around the same time in physics \cite{Berenstein-Leigh:2001,
  Berenstein-Douglas:2002, Berenstein:2002,
  Douglas-Greene-Morrison:1997} and in pure mathematics,
notably~\cite{Bridgeland-King-Reid:2001,Bondal-Orlov:2002}.  In 2002,
inspired by Bridgeland's proof~\cite{Bridgeland:2002} of a conjecture
due to Bondal and Orlov, Van den Bergh~\cite{VandenBergh:flops}
proposed a definition for a \emph{non-commutative crepant resolution}
of a ring $R$.  This is an $R$-algebra $\Lam$ which is (a) finitely
generated as an $R$-module, (b) generically Morita equivalent to $R$,
and (c) has finite global dimension.  These three attributes are
supposed to stand in for the components of the definition of a
resolution of singularities: it is (a) proper, (b) birational, and (c)
non-singular.  The additional ``crepancy'' condition is a certain
symmetry hypothesis on $\Lam$ which is intended to stand in for the
condition that the resolution of singularities not affect the
canonical sheaf.  See~\S\ref{sect:NCCRs} for details.

My main goal for this article is to motivate the definition of a
non-commutative crepant resolution (Definition~\ref{def:NCCR}).  In
order to do that effectively, I will attempt to describe the contexts
out of which it arose.  These are several, including Morita theory and
tilting, the McKay correspondence, the minimal model program of Mori
and Reid, and especially work of Bondal and Orlov on derived
categories of coherent sheaves.  Of course, the best motivation for a
new definition is the proof of a new result, and I will indicate where
the new concepts have been applied to problems in ``commutative''
geometry.  Finally, the article contains a healthy number of examples,
both of existence and of non-existence of non-commutative crepant
resolutions.  Since it is not at all clear yet that the definitions
given below are the last word, we can hope that reasoning by example 
will point the way forward.

As Miles Reid writes in \cite{Reid:oldperson}, \begin{quote}It is
  widely appreciated that mathematicians usually treat history in a
  curiously dishonest way, rewriting the history of the subject as it
  should have been discovered [\dots.]  The essential difficulty seems
  to be that the story in strictly chronological order will not make
  sense to anyone; the writer wants to give an explanation based on
  the logical layout of the subject, whatever violence it does to
  historical truth.\end{quote} This article will be guilty of the
dishonesty Reid suggests, intentionally in some places and, I fear,
unintentionally in others.  I do intend to build a certain logical
layout around the ideas below, and I sincerely apologize to any who
feel that violence has been done to their ideas.

Here is a thumbnail sketch of the contents.  The first few sections
consider, on both the algebraic side and the geometric, the
reconstruction of the underlying ring or space from certain associated
categories. The obstructions to this reconstruction---and even to
reconstruction of the commutative property---are explained by Morita
equivalence (\S\ref{sect:Morita}) and tilting theory
(\S\S\ref{sect:der-mods}--\ref{sect:der-sheaves}).
\S\ref{sect:beilinson} contains a central example: Be\u\i linson's
``tilting description'' of the derived category of coherent sheaves on
projective space.

This is not intended to be a comprehensive introduction to
non-commutative algebraic geometry; for one thing, I am nowise
competent to write such a thing.  What I cannot avoid saying is
in~\S\ref{sect:non-existence}. 

The next two sections give synopses of what I need from the geometric
theory of resolutions of singularities and the minimal model program,
followed in~\S\ref{sect:cat-desing} by some remarks on purely
category-theoretic replacements for resolutions of singularities.
Another key example, the McKay correspondence, appears
in~\S\ref{sect:McKay}.

At last in~\S\ref{sect:NCCRs} I define non-commutative crepant
resolutions. The definition I give is slightly different from Van den
Bergh's original, but agrees with his in the main case of interest.
The next few sections
\S\S\ref{sect:normalization}--\ref{sect:ratsings} focus on particular
aspects of the definition, recapping some related research and
focusing on obstructions to existence.  In particular I give
several more examples of existence and non-existence of
non-commutative crepant resolutions.  Two more families of examples
take up~\S\ref{sect:FCMT} and~\S\ref{sect:detX}: rings of finite
representation type and the generic determinantal hypersurface.  Here
tilting returns, now as a source of non-commutative crepant
resolutions.  I investigate a potential theory of ``non-commutative
blowups'' in~\S\ref{sect:blowups}, and give very quick indications of
some other examples in~\S\ref{sect:omissions}.  That section also
lists a few open questions and gestures at some topics that were
omitted for lack of space, energy, or expertise.

Some results are simplified from their published versions for
expository reasons. In particular I focus mostly on local rings,
allowing some cleaner statements at the cost of generality, even
though such generality is in some cases necessary for the proofs.  In
any case I give very few proofs, and sketchy ones at that.  The only
novel contribution is a relatively simple proof, in \S\ref{sect:detX},
of the $m=n$ case of the main theorem
of~\cite{Buchweitz-Leuschke-VandenBergh:2010}.  

The reader I have in mind has a good background in commutative
algebra, but perhaps less in non-commutative algebra, algebraic
geometry, and category theory.  Thus I spend more time on trivialities
in these latter areas than in the first.  I have tried to make the
references section comprehensive, though it surely is out of date
already.

I am grateful to Jesse Burke, Hailong Dao, Kos Diveris, and Michael
Wemyss for insightful comments on earlier drafts, and to Ragnar-Olaf
Buchweitz and Michel Van den Bergh for encouragement in this project,
as well as for years of enjoyable mathematical interaction.

\begin{conventions}
  All modules will be left modules, so for a ring $\Lam$ I will denote by
  $\lmod \Lam$ the category of finitely generated left $\Lam$-modules, and
  by $\rmod \Lam = \lmod {\Lam^{\op}}$ the category of finitely generated
  left $\Lam^\op$-modules. Other categories of modules will be defined on
  the fly.  Capitalized versions of names, namely $\Lmod \Lam$, etc.,
  will denote the same categories without any hypothesis of finite
  generation.

  Throughout $R$ and $S$ will be commutative noetherian rings, usually
  local, while $\Lam$ and its Greek-alphabet kin will not necessarily
  be commutative.
\end{conventions}

\section{Morita equivalence}\label{sect:Morita}

To take a representation-theoretic view is to replace the study of a
ring by the study of its (abelian category of) modules.  Among other
advantages, this allows to exploit the tools of homological algebra.
A basic question is: How much information do we lose by becoming
representation theorists?  In other words, when are two rings
indistinguishable by their module categories, so that $\lmod \Lam$ and
$\lmod \Gam$ are the same abelian category for different rings $\Lam$
and $\Gam$?

To fix terminology, recall that a functor $\catf \colon \cata \to
\catb$ between abelian categories is \emph{fully faithful} if it
induces an isomorphism on $\Hom$-sets, and \emph{dense} if it is
surjective on objects up to isomorphism.  If $\catf$ is both fully
faithful and dense, then it is an
\emph{equivalence}~\cite[IV.4]{MacLane:categories}, that is, there is
a functor $\catg \colon \catb \to \cata$ such that both compositions
are isomorphic to the respective identities.  In this case  write
$\cata \simeq \catb$.  Equivalences preserve and reflect essentially
all ``categorical'' properties and attributes: mono- and epimorphisms,
projectives, injectives, etc.  Thus the question above asks when two
module categories are equivalent.

Morita's theorem on equivalences of module categories~\cite[Section
3]{Morita} completely characterizes the contexts in which $\lmod \Lam
\simeq \lmod \Gam$ for rings $\Lam$ and $\Gam$.  First I define some
of the necessary terms.

\begin{definition}
  \label{def:morita}
  Let $\Lam$ be a ring and $M \in \Lmod \Lam$.
  \begin{enumerate}
  \item \label{item:add} Denote by $\add M$ the full subcategory of
    $\Lmod \Lam$ containing all direct summands of finite direct sums
    of copies of $M$.
  \item \label{item:generator} Say $M$ is a \emph{generator} (for
    $\lmod \Lam$) if every finitely generated left $\Lam$-module is a
    homomorphic image of a finite direct sum of copies of $M$.
    Equivalently, $\Lam \in \add M$.
  \item \label{item:progenerator} Say $M$ is a \emph{progenerator} if
    $M$ is a finitely generated projective module and a generator.
    Equivalently, $\add \Lam = \add M$.
  \end{enumerate}
\end{definition}

\begin{theorem}[Morita equivalence, see e.g.\ {\cite[Chap. V]{Gabriel:1962}}]
  \label{thm:morita}
  The following  are equivalent for rings $\Lam$ and $\Gam$.
  \begin{enumerate}
  \item There is an equivalence of abelian categories $\lmod \Lam
    \simeq \lmod \Gam$.
  \item There exists a progenerator $P \in \lmod \Lam$ such that
    $\Gam \cong \End_\Lam(P)^\op$.
  \item There exists a $(\Lam\text{-}\Gam)$-bimodule ${}_\Lam P_\Gam$ such
    that the functor $\Hom_\Lam(P,-)\colon \lmod \Lam \to \lmod \Gam$
    is an equivalence.
  \end{enumerate}
  In this case,  say that $\Lam$ and $\Gam$ are \emph{Morita
    equivalent.}   \qed
\end{theorem}
Interesting bits of the history of Morita's theorem, as well as his
other work, can be found in~\cite{Morita:obit}.

An immediate corollary of Morita's theorem will be useful later.
\begin{corollary}
  \label{cor:intertwine} Let $\Lam$ be a ring and $M$, $N$ two
  $\Lam$-modules such that $\add M = \add N$, equivalently $M$ is a
  direct summand of $N^s$ for some $s$ and $N$ is a direct summand of
  $M^t$ for some $t$. Then $\End_\Lam(M)$ and $\End_\Lam(N)$ are
  Morita equivalent via the functors
  $\Hom_\Lam(M,N)\otimes_{\End_\Lam(M)}-$ and
  $\Hom_\Lam(N,M)\otimes_{\End_\Lam(N)}-$.
\end{corollary}

Among the consequences of Theorem~\ref{thm:morita}, the most
immediately relevant to our purposes are those related to
commutativity. One can show~\cite[18.42]{Lam:secondbook} that if
$\Lam$ and $\Gam$ are Morita equivalent, then the centers $Z(\Lam)$ and
$Z(\Gam)$ are isomorphic.  It follows that two \emph{commutative}
rings $R$ and $S$ are Morita equivalent if and only if they are
isomorphic.  On the other hand, in general Morita equivalence is blind
to the commutative property.  Indeed, the free module $\Lam^n$ is a
progenerator for any $n \geq 1$, so that $\Lam$ and the matrix ring
$\End_\Lam(\Lam^n) \cong \Mat_n(\Lam)$ are Morita equivalent.  Even if
$\Lam$ is commutative, $\Mat_n(\Lam)$ will not be for $n\geq 2$.

The fact that commutativity is invisible to the module category is a
key motivation for categorical geometry.  It is interesting to observe
that this idea, and even the connection with endomorphism rings, is
already present in the Freyd--Mitchell Theorem~\cite[Theorem
7.34]{Freyd:AbelianCategories} classifying abelian categories as
categories of modules.  In detail, the Freyd--Mitchell Theorem says
that if $\catc$ is a category whose objects form a set (as opposed to
a proper class) which is closed under all set-indexed direct sums, and
$\catc$ has a progenerator $P$ such that $\Hom_\catc(P,-)$ commutes
with all set-indexed direct sums, then $\catc \simeq \Lmod \Lam$ for
$\Lam = \End_\catc(P)$.  Different choices of $P$ obviously give
potentially non-commutative rings $\Lam$, even if $\catc = \lmod R$
for some commutative ring $R$.

The property of $P$ referred to above will recur later:  say that
$P$ is \emph{compact} if $\Hom_\catc(P,-)$ commutes with all
(set-indexed) direct sums.

The following cousin of Morita equivalence will be essential later
on.  

\begin{prop}[Projectivization~{\cite[II.2.1]{AuslanderReitenSmalo}}]
  \label{prop:projectivization}
  Let $\Lam$ be a ring and $M$ a finitely generated $\Lam$-module
  which is a generator. Set $\Gam = \End_\Lam(M)^\op$.  Then the
  functor
  \[
  \Hom_\Lam(M,-) \colon \lmod \Lam \to \lmod \Gam
  \]
  is fully faithful, and restricts to an equivalence
  \[
  \Hom_\Lam(M,-) \colon \add M \to \add \Gam\,.
  \]
  In particular, the indecomposable projective $\Gam$-modules are
  precisely the modules of the form $\Hom_\Lam(M,N)$ for $N$ an
  indecomposable module in $\add M$. \qed
\end{prop}

\section{(Quasi)coherent sheaves}\label{sect:sheaves}

On the geometric side, it has also long been standard operating
procedure to study a variety or scheme $X$ in a
representation-theoretic mode by investigating the sheaves on $X$,
particularly those of algebraic origin, the quasicoherent sheaves.

Let $X$ be a noetherian scheme.  Recall that an $\calox$-module is
\emph{quasicoherent} if it locally can be represented as the cokernel
of a homomorphism between direct sums of copies of $\calox$.  A
quasicoherent sheaf is \emph{coherent} if those direct sums can be
chosen to be finite.  Write $\Qcoh X$ for the category of
quasicoherent sheaves and $\coh X$ for that of coherent sheaves.
Since $X$ is assumed noetherian, these are both abelian categories
(for the quasicoherent sheaves, quasi-compact and quasi-separated is
enough~\cite{Bondal-VdB:2003}).

The category $\Qcoh X$ is a natural environment for homological
algebra over schemes; for example, computations of cohomology
naturally take place in $\Qcoh X$.  Now one may ask the same question
as in the previous section: What information, if any, is lost in
passage from the geometric object $X$ to the category $\coh X$ or
$\Qcoh X$?  In this case, the kernel is even smaller: we lose
essentially nothing.  Indeed, it is not hard to show that for
arbitrary complex varieties $X$ and $Y$, the categories $\coh X$ and
$\coh Y$ are equivalent if and only if $X$ and $Y$ are isomorphic.
The key idea is to associate to a coherent sheaf the closed subset of
$X$ on which it is supported; for example, the points of $X$
correspond to the simple objects in $\coh X$.  See for
example~\cite[Section 8]{Buan-Krause-Solberg}.  More generally,
Gabriel~\cite{Gabriel:1962} taught us how to associate to any abelian
category $\cata$ a \emph{geometric realization}: a topological space
$\Spec \cata$, together with a sheaf of rings $\calo_\cata$.  (In
fact, the sheaf $\calo_\cata$ is the endomorphism sheaf of the
identity functor on $\cata$, reminiscent of the Freyd--Mitchell
theorem mentioned in the previous section.  The space $\Spec \cata$ is
nothing but the set of isomorphism classes of indecomposable injective
objects of $\cata$, with a base for the topology given by $\Supp M =
\left\{\ [I] \ \middle|\ \text{there is a nonzero arrow } M \to I\
\right\}$ for noetherian objects $M$.)  In the case $\cata = \Qcoh X$
for a noetherian scheme $X$, the pair $(\Spec \cata,\ \calo_\cata)$ is
naturally isomorphic to $(X,\ \calox)$.  This construction has been
generalized to arbitrary schemes by Rosenberg~\cite{Rosenberg:1998,
  Rosenberg:recon}, giving the following theorem.

\begin{theorem}
  [Gabriel--Rosenberg Reconstruction]
  \label{thm:gab-ros}
  A scheme $X$ can be reconstructed up to isomorphism from the abelian
  category $\Qcoh X$. \qed
\end{theorem}
%

Theorem~\ref{thm:gab-ros} implies that there is no interesting Morita-type
theory for  (quasi)coherent sheaves.  This is not all
that surprising, given that Morita-equivalent commutative rings are
necessarily isomorphic.  The well-known equivalence between modules
over a ring $R$ and quasicoherent sheaves over the affine scheme
$\Spec R$ strongly suggests the same sort of uniqueness on the
geometric side as on the algebraic.

For \emph{projective} schemes, Serre's fundamental
construction~\cite{Serre:FAC} describes the quasicoherent sheaves on
$X$ in terms of the graded modules over the homogeneous coordinate
ring.  Explicitly, let $A$ be a finitely generated graded algebra over
a field, and set $X = \Proj A$, the associated projective scheme.
Let $\GrMod A$, resp.\ $\grmod A$, denote the category of graded,
resp.\ finitely generated graded, $A$-modules.  The graded modules
annihilated by $A_{\geq n}$ for $n \gg 0$ form a subcategory $\Tors A$, resp.\
$\tors A$, and 
\[
\Tails A = \GrMod A/\Tors A
\qquad\text{ and }\qquad
\tails A = \grmod A/\tors A
\]
are defined to be the quotient categories.  This means that two graded
modules $M$ and $N$ are isomorphic in $\Tails A$ if and only if
$M_{\geq n} \cong N_{\geq n}$ as graded modules for large enough $n$.

\begin{theorem}
  [Serre]\label{thm:serre-grmod}
  Let $A$ be a commutative graded algebra generated in degree one over
  $A_0 = k$, a field, and set $X = \Proj A$.   Then the
  functor $\Gam_* \colon \coh X \to \tails A$, defined by
  sending a coherent sheaf $\calf$ to the image in $\tails A$ of
  $\bigoplus_{n=-\infty}^\infty H^0(X,\calf(n))$, defines an
  equivalence of categories $\coh X \simeq \tails A$.  \qed
\end{theorem}

Serre's theorem is the starting point for ``non-commutative projective
geometry,'' as we shall see in \S\ref{sect:non-existence} below.  From
the point of view of categorical geometry, it is the first instance of
a purely algebraic description of the (quasi)coherent sheaves on a
space, and thus opens the possibility of ``doing geometry'' with only
a category in hand. 

\section{Derived categories of modules}\label{sect:der-mods}

Originally introduced as technical tools for organizing homological
(or ``hyperhomological''~\cite{Verdier:thesis}) information, derived
categories have in the last 30 years been increasingly viewed as a
basic invariant of a ring or variety.  Passing from an abelian
category to an associated derived category not only tidies the
workspace by incorporating the non-exactness of various natural
functors directly into the notation, but in some cases it allows a
``truer description''~\cite{Bridgeland:2006} of the underlying algebra
or geometry than the abelian category does.  For example, there are
varieties with non-trivial derived auto-equivalences $\D^b(X) \simeq
\D^b(X)$ that do not arise from automorphisms; one might think of
these as additional symmetries that were invisible from the geometric
point of view.  Another example is Kontsevich's Homological Mirror
Symmetry conjecture~\cite{Kontsevich:1994}, which proposes an
equivalence of certain derived categories related to ``mirror pairs''
of Calabi-Yau manifolds.

Let us fix some notation.  Let $\cata$ be an abelian category.  The
\emph{homotopy category} $\K(\cata)$ has for objects the complexes
over $\cata$, and for morphisms homotopy-equivalence classes of chain
maps.  The \emph{derived category} $\D(\cata)$ is obtained by formally
inverting those morphisms in $\K(\cata)$ which induce isomorphisms on
cohomology, i.e.\ the quasi-isomorphisms.

We decorate $\K(\cata)$ and $\D(\cata)$ in various ways to denote full
subcategories.  For the moment I need only $\K^b(\cata)$, the full
subcategory composed of complexes $C$ having only finitely many
non-zero components, and $\D^b(\cata)$, the corresponding
\emph{bounded derived category} of $\cata$.

The homotopy category $\K(\cata)$, the derived category $\D(\cata)$
and their kin are no longer abelian categories, but they have a
\emph{triangulated structure}, consisting of a \emph{shift functor}
$(-)[1]$ shifting a complex one step against its differential and
changing the sign of that differential, and a collection of
\emph{distinguished triangles} taking the place occupied by the short
exact sequences in abelian categories. A functor between triangulated
categories is said to be a \emph{triangulated functor} if it preserves
distinguished triangles and intertwines the shift operators.

Homomorphisms $\phi\colon M \to N$ in $\D(A)$ are diagrams $M
\xleftarrow{\ f\ } P \xrightarrow{\ g\ } N$ of homotopy classes of
chain maps, where $f \colon P \to M$ is a quasi-isomorphism and we
think of $\phi$ as $f^{-1}g$.  Much more usefully,
\[
\Hom_{\D(A)}(M,N[i]) = \Ext_A^i(M,N)
\]
for all $i \in \ZZ$ and all $M$, $N$ in $\D(A)$.

Let us say that two rings $\Lam$ and $\Gam$ are \emph{derived
  equivalent} if there is an equivalence of triangulated categories
$\D^b(\Lmod \Lam) \to \D^b(\Lmod \Gam)$.

It is nearly obvious that a Morita equivalence between rings $\Lam$
and $\Gam$ gives rise to a derived equivalence $\D^b(\Lmod \Lam)
\simeq \D^b(\Lmod \Gam)$. (Any equivalence between abelian categories
preserves short exact sequences.)  In general, derived equivalence is
a much weaker notion.  It does, however, preserve some essential
structural information.  For example, if $\Lam$ and $\Gam$ are derived
equivalent, then their Grothendieck groups $K_0(\Lam)$ and $K_0(\Gam)$
are isomorphic~\cite[Prop. 9.3]{Rickard:1989}, as are the Hochschild
homology and cohomology groups~\cite{Rickard:1991} and the
cyclic cohomologies~\cite{Happel:1989}.  If $\Lam$ and $\Gam$ are
derived-equivalent finite-dimensional algebras over a field $k$, then
they have the same number of simple modules, and simultaneously have
finite global dimension~\cite{Happel:1987, Happel:1988}.  

Most importantly for this article, derived-equivalent rings have
isomorphic centers~\cite[Prop. 9.2]{Rickard:1989}.  In
particular, if $R$ and $S$ are commutative rings, then $\D^b(\Lmod R)
\simeq \D^b(\Lmod S)$ if and only if $R \cong S$.  Thus there is at
most one commutative ring in any derived-equivalence class, another
sign that one should look at non-commutative rings for non-trivial
derived equivalences.  

All these facts follow from Rickard's Morita theory of derived
equivalences, in which the progenerator of Theorem~\ref{thm:morita} is
replaced by a \emph{tilting object.}  Here is the main result of
Rickard's theory.

\begin{theorem}
  [Rickard~{\cite{Rickard:1989}}]
  \label{thm:rickard}
  Let $\Lam$ and $\Gam$ be rings. The following conditions are
  equivalent.
  \begin{enumerate}
  \item\label{item:rickard1} $\D^b(\Lmod \Lam)$ and $\D^b(\Lmod \Gam)$
    are equivalent as triangulated categories.
  \item\label{item:rickard2} $\K^b(\add \Lam)$ and $\K^b(\add \Gam)$
    are equivalent as triangulated categories.
  \item \label{item:rickard3}There is an object $T \in \K^b(\add
    \Lam)$ satisfying
    \begin{enumerate}[\quad(a)]
    \item $\Ext^i_{\Lam}(T,T) = 0$ for all $i > 0$, and
    \item $\add T$ generates $\K^b(\add \Lam)$ as a triangulated category,
    \end{enumerate}
    such that $\Gam \cong \End_\Lam(T)$. 
  \end{enumerate}
  If $\Lam$ and $\Gam$ are finite dimensional algebras over a field,
  then these are all equivalent to $\D^b(\lmod \Lam) \simeq \D^b(\lmod
  \Gam)$ \qed
\end{theorem}
A complex $T$ as in condition~(\ref{item:rickard3}) is called a
\emph{tilting complex} for $\Lam$, and $\Gam$ is \emph{tilted from}
$\Lam$.  Tilting complexes appeared first in the form of \emph{tilting
  modules,} as part of Brenner and Butler's~\cite{Brenner-Butler:1980}
study of the reflection functors of Bern{\v{s}}te{\u\i}n,
Gel{$'$}fand, and Ponomarev~\cite{Bernstein-Gelfand-Ponomarev}.  (The
word was chosen to illustrate their effect on the vectors in a root
system, namely a change of basis that tilts the axes relative to the
positive roots.)  Their properties were generalized, formalized, and
investigated subsequently by Happel and
Ringel~\cite{Happel-Ringel:1982}, Bongartz~\cite{Bongartz:1981},
Cline--Parshall--Scott~\cite{Cline-Parshall-Scott:1986},
Miyashita~\cite{Miyashita:1986}, and others. Happel seems to have been
the first to realize~\cite{Happel:1988} that if $T$ is a $\Lam$-module
of finite projective dimension, having no higher self-extensions
$\Ext_\Lam^{>0} (T,T)=0$, and $\Lam$ has a finite co-resolution $0 \to
\Lam \to T_1 \to \cdots \to T_r \to 0$ with each $T_i \in \add T$,
then the functor $\Hom_\Lam(T,-) \colon \D^b(\Lmod \Lam) \to
\D^b(\Lmod {\End_\Lam(T)})$ is an equivalence.  In their earliest
incarnation, tilting modules were defined to have projective dimension
one, but this more general version has become standard.

Morita equivalence is a special case of tilting.  Indeed, any
progenerator is a tilting module.  However, see \S\ref{sect:beilinson}
below for a pair of derived-equivalent algebras which are not Morita
equivalent.

\section{Derived categories of sheaves}\label{sect:der-sheaves}

The first real triumphs of the derived category came in the geometric
arena: Grothendieck and coauthors' construction of a global
intersection theory and the theorem of Riemann--Roch~\cite{SGA6} are
the standard examples~\cite{Caldararu:2005}.  The idea of the
(bounded) derived category of a scheme as a geometric invariant first
emerged around 1980 in the work of Be\u\i linson, Mukai, and others.
I will describe some of Be\u\i linson's observations in the next \S{}.
Mukai found the first example of non-isomorphic varieties which are
derived equivalent~\cite{Mukai:1981}; he showed that an abelian
variety $X$ and its dual $X^\svee$ always have equivalent derived
categories of quasicoherent sheaves.  His construction is modeled on a
Fourier transform and is now called a Fourier-Mukai
transform~\cite{Huybrechts, Hille-VdB:2007}.  It would draw us too far
afield from our subject to discuss Fourier-Mukai transforms in any
depth here.  Several examples will appear later in the text: see the
end of \S\ref{sect:beilinson} and
Theorems~\ref{thm:bridgeland},~\ref{thm:kap-vass}, and~\ref{thm:BKR}.
It is an important result of Orlov~\cite{Orlov:1997} that \emph{any}
equivalence $\D^b(\coh X) \to \D^b(\coh Y)$, for $X$ and $Y$ connected
smooth projective varieties, is given by a Fourier-Mukai transform.

The existence of non-trivial derived equivalences for categories of
sheaves means that one cannot hope for a general reconstruction
theorem, even for smooth varieties.  However, under an assumption on
the \emph{canonical sheaf} $\omega_X$, the variety $X$ can be
reconstructed from its derived category.  Recall that for $X$ a smooth
complex variety over $\CC$, $\omega_X = \bigwedge^{\dim X}
\Omega_{X/\CC}$ is the sheaf of top differential forms on $X$ where
$\Omega_{X/\CC}$ is the cotangent bundle, a.k.a.\ the sheaf of
$1$-forms on $X$.  It is an invertible sheaf.  Recall
further~\cite[II.7]{Hartshorne} that an invertible sheaf $\call$ is
\emph{ample} if for every coherent sheaf $\calf$, $\calf \otimes
\call^n$ is generated by global sections for $n \gg 0$.

\begin{theorem}
  [Bondal--Orlov~{\cite{Bondal-Orlov:2001}}]
  \label{thm:BOrecon}
  Let $X$ and $Y$ be smooth connected projective varieties over $\CC$.
  Assume that either the canonical sheaf $\omega_X$ or the
  anticanonical sheaf $\omega_X^{-1}$ is ample.  If $\D^b(\coh X)
  \simeq \D^b(\coh Y)$, then $X$ is isomorphic to $Y$. \qed
\end{theorem}
Note that the result is definitely false for abelian varieties by the
result of Mukai mentioned above; in this case $\omega_X\cong \calox$
is trivial, so not ample.  Calabi-Yau varieties are another
example where $\omega_X\cong \calox$ is not ample, and the conclusion
does not hold.

One consequence of this theorem is that, under the same hypotheses,
the group of \emph{auto-equivalences} $\D^b(\coh X)
\xto{\ \simeq\ } \D^b(\coh X)$ of $X$ is generated by the obvious
suspects: $\Aut(X)$, the shift $(-)[1]$, and the tensor products
$-\otimes_\calox \call$ with fixed line bundles $\call$.

Triangulated categories arising in nature like $\D^b(\coh X)$
generally have a lot of additional structure: there is a tensor
(symmetric monoidal) structure induced from the derived tensor
product, among other things.  Taking this into account gives stronger
results.  To give an example, recall that a \emph{perfect complex} on
a scheme $X$ is one which locally is isomorphic in the derived
category to a bounded complex of locally free sheaves of finite rank.
Perfect complexes form a subcategory $\D^\perf(\Qcoh X)$.  As long as
$X$ is quasi-compact and separated (noetherian is enough),
$\D^\perf(\Qcoh X)$ contains precisely the \emph{compact} objects of
$\D(\Qcoh X)$, that is, the complexes $C$ such that $\Hom_\calox(C,-)$
commutes with set-indexed direct
sums. See~\cite[3.11]{Bondal-VdB:2003}.  Balmer~\cite{Balmer:2002,
  Balmer:2005} shows that a noetherian scheme $X$ can be reconstructed
up to isomorphism from $\D^\perf(\Qcoh X)$, as long as the natural
tensor structure is taken into account, and that two reduced
noetherian schemes $X$ and $X'$ are isomorphic if and only if
$\D^\perf(\Qcoh X)$ and $\D^\perf(\Qcoh X')$ are equivalent as tensor
triangulated categories.


The theory of tilting sketched in the previous \S{} has a geometric
incarnation as well, which signals the first appearance of
non-commutative rings on the geometric side of our story.

\begin{definition}
  \label{def:tilt-obj}
  Let $X$ be a noetherian scheme and $T$ an object of $\D(\Qcoh X)$.
  Say that $T$ is a \emph{tilting object} if it is compact, is a
  classical generator for $\D^\perf(\Qcoh X)$, and has no non-trivial
  self-extensions.  Explicitly, this is to say:
  \begin{enumerate}
  \item \label{item:compact} $T$ is a perfect complex;
  \item \label{item:gen} The smallest triangulated subcategory of
    $\D(\Qcoh X)$ containing $T$ and closed under direct summands is
    $\D^\perf(\Qcoh X)$; and
  \item \label{item:ext} $\Ext^i_\calox(T,T)=0$ for $i> 0$.
  \end{enumerate}
  If $T$ is quasi-isomorphic to a complex consisting of a locally free
  sheaf in a single degree, it is sometimes called a \emph{tilting
    bundle}.
\end{definition}

The generating condition~(\ref{item:gen}) is sometimes replaced by the
requirement that $T$ \emph{generates} $\D(\Qcoh X)$, i.e.\ that if an
object $N$ in $\D(\Qcoh X)$ satisfies $\Ext_\calox^i(T,N)=0$ for all
$i \in \ZZ$, then $N=0$.  If an object $T$ classically generates
$\D^\perf(\Qcoh X)$ as in the definition, then it 
generates $\D(\Qcoh X)$; the converse holds in the presence of the
assumption~(\ref{item:compact}) that $T$ is compact.  This is a
theorem due to Ravenel and Neeman~\cite[Theorem
2.1.2]{Bondal-VdB:2003}.

A class of schemes particularly well-suited for geometric tilting
theory consists of those which are projective over a scheme $Z$, which
in turn is affine of finite type over an algebraically closed field
$k$.  This generality allows a wide range of interesting examples, but
also ensures, by~\cite[Th\'eor\`eme 2.4.1(i)]{EGAIII.1}, that if $T$
is a tilting object on $X$ then the endomorphism ring $\Lam
= \End_\calox(T)$ is a \emph{finitely generated} algebra over the
field $k$.  In particular, $\Lam$ is finitely generated as a module
over its center.  

The next result is fundamental for everything that follows.  It has
origins in the work of Be\u\i linson presented in the next \S{}, with
further refinements in~\cite{Bondal:1989, Baer:1988, Bondal-VdB:2003}. 

\begin{theorem}
  [Geometric Tilting Theory~{\cite[7.6]{Hille-VdB:2007}}]
  \label{thm:geomtilt}
  Let $X$ be a scheme, projective over a finite-type affine scheme
  over an algebraically closed field $k$.  Let $T$ be a tilting object
  in $\D(\Qcoh X)$, and set $\Lam = \End_\calox(T)$.  Then
  \begin{enumerate}
  \item $\rHom_\calox(T,-)$ induces an equivalence of triangulated
    categories between $\D(\Qcoh X)$ and $\D(\Lmod \Lam)$, with
    inverse $-\lotimes_\Lam T$.
  \item If $T$ is in $\D^b(\coh X)$, then this equivalence restricts
    to give an equivalence between $\D^b(\coh X)$ and $\D^b(\lmod
    \Lam)$.
  \item If $X$ is smooth, then $\Lam$ has finite global dimension.\qed
  \end{enumerate} 
\end{theorem}

It is not at all clear from this result when tilting objects exist,
though it does impose some necessary conditions on $X$.  For example,
assume that in addition $X$ is projective over $k$ and $T$ is a
tilting object in $\D(\Qcoh X)$.  Then $\Lam = \End_\calox(T)$ is
a finite-dimensional algebra over $k$.  The Grothendieck group
$K_0(\Lam)$ is thus a free abelian group of finite rank, equal to the
number of simple $\Lam$-modules.  This implies that $K_0(X)$ is free
abelian as well.  Thus any torsion in $K_0$ rules out the existence of
a tilting object.


\section{Example: Tilting on projective space}\label{sect:beilinson}

In this \S{} I illustrate Theorem~\ref{thm:geomtilt} via Be\u\i
linson's tilting description of the derived category of projective
space.  The techniques have been refined and are now standard; they
have been used, most notably by Kapranov, to construct explicit
descriptions of the derived category of coherent sheaves on several
classes of varieties.  For example, there are tilting bundles on
smooth projective quadrics~\cite{Kapranov:1986}, on
Grassmannians~\cite{Kapranov:1984,
  Buchweitz-Leuschke-VandenBergh:tiltGrass}, on flag
manifolds~\cite{Kapranov:1988}, on various toric
varieties~\cite{King:tilting, Hille-Perling:2006, Hille-Perling:2008},
and on weighted projective spaces~\cite{Geigle-Lenzing:1987,Baer:1988}.
Here we stick to projective space.

Let $k$ be a field, $V$ a $k$-vector space of dimension $n \geq 2$,
and $\PP = \PP^{n-1} = \PP(V)$ the projective space on $V$.  We
consider two families of $n$ locally free sheaves on $\PP$.  First let 
\[
\cale_1 = \left\{\calo, \calo(-1), \dots, \calo(-n+1)\right\}
\]
where $\calo = \calop$ is the structure sheaf.  Also let $\Omega =
\Omega_\PP$ be the cotangent sheaf, so that $\Omega^i = \Wedge^i
\Omega$ is the $\calo$-module of differential $i$-forms on $\PP$, and
set
\[
\cale_2 = \left\{\Omega^0(1) = \calo(1), \Omega^1(2), \dots,
  \Omega^{n-1}(n)\right\}\,.
\]
Let $T_1$ and $T_2$ be the corresponding direct sums, 
\[
T_1 = \bigoplus_{a=0}^{n-1} \calo(-a)
\qquad\text{and}\qquad
T_2 = \bigoplus_{a=1}^{n} \Omega^{a-1}(a)\,.
\]
The constituent sheaves of $\cale_1$ and $\cale_2$ are related by the
tautological Koszul complex on $\PP$.  Indeed, the Euler derivation $e
\colon V \otimes_k \calo(-1) \to \calo$, which corresponds to the
identity on $V$ under $\Hom_\calo(V\otimes_k \calo(-1),\calo) \cong
\Hom_\calo(V\otimes_k \calo,\calo(1)) \cong \Hom_k(V,V)$, gives rise
to a complex
%
\begin{equation}\label{eq:Omegares}
0 \to \Wedge^n V \otimes_k \calo(-n) \to \cdots \to 
\Wedge^1 V \otimes_k \calo(-1) \to \calo \to 0
\end{equation}
on $\PP$.  In fact it is acyclic~\cite[Ex. 17.20]{Eisenbud:book}, and
the cokernels are exactly the sheaves $\Omega^i$, which decompose the
Koszul complex into short exact sequences
\[
0 \to \Omega^a \to \Wedge^a V \otimes_k \calo(-a) \to \Omega^{a-1} \to
0\,.
\]
Together with the well-known calculation of the cohomologies of the
sheaves $\calo(-a)$~\cite[III.5.1]{Hartshorne}, the identification
$\cHom_\calo(\calo(-a),\calo(-b))=\calo(a-b)$, and the fact that
$\Ext_\calo^i(-,-) = H^i(\cHom_\calo(-,-))$ on vector bundles, this
produces the following data.
(See~\cite{Buchweitz-Leuschke-VandenBergh:2010} for a jazzed-up
version which holds over any base ring $k$.)

\begin{lemma}
  \label{lem:beilinson}
  Keep the notation established so far in this section.
  \begin{enumerate}
  \item We have $\Ext_\calo^i(\calo(-a),\calo(-b)) = 0$ for all $i>0$,
    and 
    \[
    \Hom_\calo(\calo(-a),\calo(-b)) \cong \Sym_{a-b}(V)
    \]
    for $0 \leq a, b \leq n-1$. 

  \item We have $\Ext_\calo^i(\Omega^{a-1}(a),\Omega^{b-1}(b)) = 0$
    for all $i>0$, and
    \[
    \Hom_\calo(\Omega^{a-1}(a),\Omega^{b-1}(b)) \cong
    \Wedge^{a-b}(V^*)
    \]
    for $1 \leq a, b \leq n$, where $V^*$ is the dual of
    $V$.  
    \qed
  \end{enumerate}
\end{lemma}

The lemma in particular implies that the endomorphism rings of $T_1$
and $T_2$, 
\[
\Lam_1 = \End_\calo(\bigoplus_{a=0}^{n-1} \calo(-a)) \cong
\bigoplus_{a,b=0}^{n-1}\Sym_{a-b}(V)
\]
and
\[
\Lam_2 = \End_\calo(\bigoplus_{a=1}^{n} \Omega^{a-1}(a)) \cong
\bigoplus_{a,b=0}^{n-1} \Wedge^{a-b}(V^*)
\]
are ``spread out'' versions of the truncated symmetric and exterior
algebras, respectively.  This can be made more precise by viewing
$\Lam_1$ and $\Lam_2$ as \emph{quiver algebras.}  Consider a quiver on
$n$ vertices labeled, say, $0, 1, \dots, n-1$, and having $n$ arrows
from each vertex to its successor, corresponding to a basis of $V$,
resp.\ of $V^*$.  Introduce quadratic relations $v_iv_j =v_j v_i$
corresponding to the kernel of the natural map $V \otimes_k V \to
\Sym_2(V)$, respectively $v_iv_j =- v_j v_i$ corresponding to the
kernel of $V^* \otimes_k V^* \to \Wedge^2(V^*)$.  The resulting path
algebras with relations are isomorphic to $\Lam_1$ and $\Lam_2$,
respectively. In~\cite{Buchweitz-Leuschke-VandenBergh:2010} we call
these ``quiverized'' symmetric and exterior algebras.


I have not yet proven that $\Lam_1$ and $\Lam_2$ are derived
equivalent to $\PP$.  For this, it remains to show that the
collections $\cale_1$ and $\cale_2$ generate the derived category
$\D^b(\coh \PP)$.  This is accomplished via Be\u\i linson's
``resolution of the diagonal'' argument.  Let $\Delta \subset \PP
\times \PP$ denote the diagonal, and $p_1, p_2 \colon \PP \times \PP
\to \PP$ the projections onto the factors.  For sheaves $\calf$ and
$\calg$ on $\PP$,  set
\[
\calf \boxtimes \calg = p_1^*\calf \otimes_{\PP \times \PP} p_2^*
\calg\,,
\]
a sheaf on $\PP \times \PP$. 
One can show that the structure sheaf of the diagonal $\calo_{\Delta}$
is resolved over $\calo_{\PP \times \PP}$ by a Koszul-type resolution
\[
0 \to \calo(-n) \boxtimes \Omega^n(n) \to \cdots \to \calo(-1)
\boxtimes \Omega^1(1) \to \calo_{\PP \times \PP} \to \calo_\Delta \to0\,.
\]
In particular, $\calo_\Delta$ is in the triangulated subcategory of
$\D^b(\coh(\PP \times \PP))$ generated by sheaves of the form
$\calo(-i) \boxtimes Y$ for $Y$ in $\D^b(\coh \PP)$.  The same goes
for any object of the form $\calo_\Delta \boxtimes \mathbf{L} p_1^* X$
with $X$ in $\D^b(\coh \PP)$ as well. Push down now by $p_2$ and use
the projection formula to see that $X = \R {p_2}_* (\calo_\Delta
\lotimes \mathbf{L} p_1^* X)$ belongs to the triangulated subcategory
of $\D^b(\coh \PP)$ generated by $\calo(-i) \otimes \R {p_2}_* p_1^*
Y$.  The factor $\R {p_2}_* p_1^* Y$ is represented by the complex of
$k$-vector spaces with zero differential $\R\Gam(Y)$, and hence
$\cale_1 = \left\lbrace \calo, \calo(-1), \dots,
  \calo(-n+1)\right\rbrace$ generates $\D^b(\coh \PP)$.  On the other
hand, reversing the roles of $p_1$ and $p_2$ gives the result for
$\cale_2 = \left\lbrace \calo(1), \Omega^1(2), \dots,
  \Omega^{n-1}(n)\right\}$ as well.

This discussion proves the following theorem.

\begin{theorem}
  [Be\u\i linson]
  \label{thm:beilinson}
  Let $k$ be a field, $V$ a vector space of dimension $n \geq 2$ over
  $k$, and $\PP = \PP(V)$.  The vector bundles 
  \[
  T_1 = \bigoplus_{a=0}^{n-1} \calop(-a) 
  \qquad\text{and}\qquad
  T_2 = \bigoplus_{a=1}^{n} \Omega_{\PP}^{a-1}(a)
  \]
  are tilting bundles on $\PP$.  In particular, there are  equivalences
  of triangulated categories 
  \[
  \D^b(\lmod {\Lam_1}) \simeq \D^b(\coh \PP) \simeq \D^b(\lmod {\Lam_2})
  \]
  defined by $\rHom_\calop(T_i,-)$ for $i=1,2$, where
  $\Lam_i = \End_\calop(T_i)$.  \qed
\end{theorem}




By the way, the construction $\R {p_2}_* (\calo_\Delta 
\lotimes_\calop \mathbf{L}p_1^* (-))$, which accepts sheaves on $\PP$
and returns sheaves on $\PP$, is an example of a Fourier--Mukai
transform, the definition of which was gracefully avoided in
\S\ref{sect:der-sheaves}.  Replacing $\calo_\Delta$ by any other fixed
complex in $\D^b(\coh (\PP \times \PP))$ would give another.

\section{The non-existence of non-commutative
  spaces}\label{sect:non-existence} 

As mentioned in the Introduction, I personally am reluctant to use the
phrase ``non-commutative algebraic geometry'' to describe results like
Be\u\i linson's in \S\ref{sect:beilinson}.  While the phrase is
certainly apposite on a word-by-word basis, given that the ideas are a
natural blend of algebraic geometry and non-commutative algebra, I
find that using it in public leads immediately to being asked awkward
questions like, ``What on earth is non-commutative geometry?''  While
many people have offered thoughtful, informed answers to this
question---\cite{SmithPaul:book, Keeler:2003, Kaledin:Seoul,
  Kaledin:Tokyo, Mori:2008, Verschoren-Willaert, Laudal:2003,
  JorgensenP:notes, Stafford-VdB:2001, Ginzburg:NCG} are some of my
personal favorites---I find the whole conversation distracting from
the more concrete problems at
hand.  
I propose instead that results like Be\u\i linson's and those to
follow in later sections should be considered as part of ``categorical
geometry''.  The name seems unclaimed, apart from an online book from
1998.

In this \S{} I say a few words about a couple of approaches to
building a field called non-commutative algebraic geometry.  I have
chosen a deliberately provocative title for the section, so that there
can be no question that these are opinionated comments.  The reader
who is intrigued by the ideas mentioned here would do well to seek out
a less idiosyncratic, more comprehensive introduction such as those
cited in the previous paragraph.

One potential pitfall for the prospective student of non-commutative
geometry is that there are several disparate approaches.  For one
thing, the approach of Connes and his
collaborators~\cite{Jones-Moscovici:1997}, which some hope will
explain aspects of the Standard Model of particle physics or even
prove the Riemann Hypothesis, is based on differential geometry and
$C^*$-algebras, and is, as far as I can tell, completely separate from
most of the considerations in this article.  More subtly, even within
non-commutative \emph{algebraic} geometry, there are a few different
points of view.  I do not consider myself competent even to give
references, for fear of giving offense by omission.

So what is the problem here?  Why can't one simply do algebraic
geometry, say at the level of~\cite{Hartshorne}, over non-commutative
rings~\cite{MathOverflow:NCAG}?  There have been several sustained
attempts to do exactly this, starting in the 1970s.  There are a
couple of immediate obstacles to a na\"\i ve approach.

The first problem is to mimic the fact that a ring $R$ can be
recovered from the Zariski topology on the prime spectrum $\Spec R$
and the structure sheaf $\calo_{\Spec R}$.  (One finds, of course,
$H^0(\Spec R, \calo_{\Spec R}) = R$.)  Both of these sets of
information depend essentially on \emph{localization.}  For
non-commutative rings,  the prime spectrum is rather
impoverished; for example, the Weyl algebra $\CC\langle
x,y\rangle/(y x-xy-1)$ has trivial two-sided prime spectrum.  Even
ignoring this difficulty, localization for non-commutative
rings~\cite{NCLocalization,Jategaonkar:1986} only functions well for
\O re sets, and the complement of a prime ideal need not be an \O re
set.

One possible resolution of the problem  would be to focus on the
quotient modules $\Lam/\p$ instead of the prime ideals $\p$.  The
points of a commutative affine variety $X$ (over $\CC$, say) are in
one-one correspondence with the simple modules over the coordinate
ring $\CC[X]$.  Furthermore, a point $x \in X$ is a non-singular point
if and only if the corresponding simple module $\CC[X]/\m_x$ has
finite projective dimension.

Unfortunately, here there is a second problem: finite projective
dimension, even finite \emph{global} dimension, is a very weak
property for non-commutative rings.  For example, there is no
Auslander--Buchsbaum Theorem giving a uniform upper bound on finite
projective dimensions over a given ring; the existence of such a bound
over an Artin algebra is called the \emph{finitistic dimension
  conjecture,} and has been open since at least 1960~\cite{Bass:1960,
  Huisgen-Zimmermann:1995}. There are a host of additional technical
problems to be overcome.  It's unknown, for instance, whether finite
global dimension implies primeness (as regularity implies domain for a
commutative local ring); the Jacobson radical might fail to satisfy
the Artin--Rees property~\cite{Brown-Hajarnavis-MacEacharn:1982},
derailing the standard proof.  Pathologies abound: for example, there
is a local noetherian domain $\Lam$ of global dimension $3$ such that
every quotient ring other than $\Lam$ itself, $0$, and
$\Lam/\rad(\Lam)$ has infinite global dimension~\cite[Example
7.3]{Brown-Hajarnavis-MacEacharn:1982}.

Restricting to a smaller class of rings solves some of these problems.
For example, the class of rings $\Lam$ which are finitely generated
modules over their center $Z(\Lam)$ are much better-behaved than the
norm.  For example, $\Lam$ is left and right noetherian if $Z(\Lam)$
is, so that $\gldim \Lam = \gldim \Lam^\op$.  The ``lying over'',
``incomparability'', and ``going up'' properties hold for prime ideals
along the extension $Z(\Lam) \into \Lam$~\cite[Theorem
16.9]{Passman:1989}.  Furthermore, the following reassuring results
hold~\cite[Section 2]{Iyama-Reiten:2008}.
\begin{proposition}
  \label{prop:center-finite}
  Let $(R, \m)$ be a local ring and $\Lam$ a module-finite
  $R$-algebra.  Let $M$ be a finitely generated $\Lam$-module.
  \begin{enumerate}
  \item The \emph{dimension} of $M$, defined by $\dim M =
    \dim(R/\ann_R(M))$, is independent of the choice of central
    subring $R$ over which $\Lam$ is a finitely generated module.
  \item The \emph{depth} of $M$, defined by $\depth M = \inf\left\{\ i
      \ \middle|\ \Ext_R^i(R/\m, M) \neq0\right \}$, is also
    independent of the choice of $R$.
  \item{}(Ramras~\cite{Ramras:1969})\label{Ramras} We have 
    \[
    \depth M \leq \dim M \leq \injdim_\Lam M\,.
    \]
    In particular, if $\Lam$ is a torsion-free $R$-module and
    $\gldim\Lam < \infty$, then $\injdim_\Lam \Lam =
    \gldim\Lam$~\cite[Lemma 1.3]{Auslander:isolsing}, so that
    \[
    \depth_R \Lam \leq \dim R \leq \gldim \Lam\,.
    \]
  \item{}(\cite{Rainwater:1987} or~\cite{Goodearl:1989}) The global
    dimension of $\Lam$ is the supremum of $\pd_\Lam L$ over all
    $\Lam$-modules $L$ of finite length.\qed
  \end{enumerate}
\end{proposition}

Restricting still further, one arrives at a very satisfactory
class of rings.  Recall that for $(R, \m)$ a local ring, a finitely
generated $R$-module $M$ is \emph{maximal Cohen-Macaulay (MCM)} if
$\depth M = \dim R$.  Equivalently, there is a system of parameters
$x_1, \dots, x_d$, with $d=\dim R$, which is an $M$-regular sequence.
In the special case where $R$ is Gorenstein, this condition is
equivalent to $\Ext_R^i(M,R)=0$ for all $i >0$.

\begin{definition}
  \label{def:order}
  Let $(R, \m)$ be a local ring and $\Lam$ a module-finite
  $R$-algebra.  Say that $\Lam$ is an \emph{$R$-order} if $\Lam$ is
  maximal Cohen--Macaulay as an $R$-module.
\end{definition}
 
The terminology is imperfect: there are several other definitions of
the word ``order'' in the literature, going back decades.  Here we
follow~\cite{Auslander:1978}.  See \S\ref{sect:NCCRs} for a
connection to the classical theory of hereditary and classical orders
over Dedekind domains.

Localization is still problematic, even for orders.  In order to get a
workable theory, a condition stronger than finite global dimension is
sometimes needed.

\begin{definition}
  \label{def:non-sing} Let $R$ be a commutative ring and
  let $\Lam$ be a module-finite $R$-algebra.  Say that $\Lam$ is
  \emph{non-singular} if $\gldim \Lam_\p = \dim R_\p$ for every prime
  ideal $\p$ of $R$.
\end{definition}

Non-singular orders have a very satisfactory homological theory,
especially over Gorenstein local rings.  A non-singular order over a
local ring satisfies a version of the Auslander--Buchsbaum
Theorem~\cite[Proposition 2.3]{Iyama-Reiten:2008}: If $\Lam$ is an
$R$-order with $\gldim \Lam = d < \infty$, then for any $\Lam$-module
$M$ the equality $\pd_\Lam M + \depth M = d$ holds.
Furthermore, the following characterization of non-singularity holds
for orders~\cite[Proposition 2.13]{Iyama-Wemyss:ARduality}.

\begin{proposition}
  \label{prop:nonsingorder}
  Let $R$ be a CM ring with a canonical module $\omega$, and let
  $\Lam$ be an $R$-order.  Then the following are equivalent.
  \begin{enumerate}
  \item $\Lam$ is non-singular.
  \item $\gldim \Lam_\m = \dim R_\m$ for all maximal ideals $\m$ of $R$.
  \item The finitely generated $\Lam$-modules which are MCM as
    $R$-modules are precisely the finitely generated projective
    $\Lam$-modules.
  \item $\Hom_R(\Lam,\omega)$ is a projective $\Lam$-module and
    $\gldim \Lam < \infty$.\qed
  \end{enumerate}
\end{proposition}

\bigskip

The definitions above represent an attempt to force classical
algebraic geometry, or equivalently commutative algebra, to work over
a class of non-commutative rings.  Here is a different approach, more
consonant with the idea of ``categorical geometry.''  Rather than
focusing attention on the rings, concentrate on an abelian or
triangulated category $\catc$, which we choose to think of as $\Qcoh
X$ or $\D^b(\Qcoh X)$ for some space $X$ \emph{about which we say
  nothing further.}  In this approach, the space $X$ is nothing but a
notational placeholder, and the geometric object is the category
$\catc$.

This idea has had particular success in taking Serre's
Theorem~\ref{thm:serre-grmod} as a template and writing $\D^b(\Qcoh
X)$ for a quotient category of the form $\tails \Lam = \grmod \Lam
/\tors \Lam$.  One thus obtains what is called \emph{non-commutative
  projective geometry.}  To describe these successes, let us make the
following definition, based on the work of
Geigle--Lenzing~\cite{Geigle-Lenzing:1987},
Ver{\"e}vkin~\cite{Verevkin-a,Verevkin-b},
Artin--Zhang~\cite{Artin-Zhang:1994} and Van den
Bergh~\cite{VdB:memoir}.

\begin{definition}
  \label{def:quasi-scheme}
  A \emph{quasi-scheme} (over a field $k$) is a pair $X = (\lmod X,
  \calox)$ where $\lmod X$ is a ($k$-linear) abelian category and
  $\calox \in \lmod X$ is an object.  Two quasi-schemes $X$ and $Y$ are
  isomorphic (over $k$) if there exists a ($k$-linear) equivalence
  $\catf \colon \lmod X \to \lmod Y$ such that $\catf(\calox) \cong
  \caloy$.
\end{definition}

The obvious first example is that a (usual, commutative) scheme $X$ is
a quasi-scheme $(\coh X, \calox)$.  For any ring $\Lam$, commutative
or not, one can define the affine quasi-scheme associated to $\Lam$ to
be $\Spec \Lam := (\lmod \Lam, \Lam)$.  One checks that if $R$ is
commutative and $X = \Spec R$ is the usual prime spectrum, then the
global section functor $\Gamma (X,-) \colon \coh X \to \lmod R$
induces an isomorphism of quasi-schemes $(\coh X,\calox) \to (\lmod R,
R)$.

The basic example of a quasi-scheme in non-commutative projective
geometry is the following, which mimics the definition of $\tails$
from \S\ref{sect:sheaves} precisely. Let $\Lam$ be a noetherian graded
algebra over a field $k$.  For simplicity, assume that $A_0=k$.  Let
$\GrMod \Lam$ and $\grmod \Lam$ be the categories of graded
$\Lam$-modules, resp.\ finitely generated graded $\Lam$-modules.
Let $\Tors \Lam$, resp.\ $\tors \Lam$, be the subcategory of graded
modules annihilated by $\Lam_{\geq n}$ for $n \gg 0$.  Then define the
quotient categories
\[
\Tails \Lam = \GrMod \Lam/\Tors \Lam
\qquad\text{ and }\qquad
\tails \Lam = \grmod \Lam/\tors \Lam\,,
\]
 and set
\[
\ncProj\Lam = (\Tails \Lam, \calo)
\qquad\text{ and }\qquad
\proj \Lam = (\tails \Lam, \calo)
\]
where $\calo$ is the image of $\Lam$ in $\tails \Lam$.  Call $\ncProj
\Lam$ and $\proj \Lam$ the \emph{(noetherian) projective quasi-scheme}
determined by $\Lam$.  The \emph{dimension} of the projective
quasi-scheme is $\GKdim \Lam-1$, where $\GKdim \Lam$ is the
Gelfand-Kirillov dimension; this means $\dim \proj \Lam$ is the
polynomial rate of growth of $\{\dim _k \Lam_n\}_{n \geq 0}$.

One can define sheaf cohomology $H^j(\Tails \Lam,-)$ in $\Tails \Lam$
to directly generalize the commutative definition.  In particular the
global sections functor is $\Gam(-) = \Hom_{\Tails \Lam}(\calo,-)$.
For $M$ in $\Tails \Lam$, then, one would like versions of two basic
results in algebraic geometry: Serre-finiteness ($H^j(\Tails
\Lam,M)=0$ for $j \gg0$) and Serre-vanishing ($H^j(\Tails \Lam,
M(i))=0$ for all $j \geq 1$ and $i \gg0$).  These results turn out only
to be true under a technical condition called $\chi$
(see~\cite{Artin-Zhang:1994}), which is automatic in the commutative
case. There is also an analogue of Serre's
Theorem~\ref{thm:serre-grmod} due to Artin and Van den
Bergh~\cite{Artin-VdB:1990}, which gives the same sort of purely
algebraic description of $\Qcoh X$ as $\Tails \Lam$, where $\Lam$ is
defined to be a \emph{twisted homogeneous coordinate ring.}  For
details, see~\cite{Stafford-VdB:2001}. 

The classification of projective quasi-schemes of small dimension,
i.e.\ categories of the form $\tails \Lam$ where $\Lam$ is a graded
algebra with small rate of growth, is an ongoing program.  The case of
non-commutative curves (where $\dim_k A_n$ grows linearly) was
completed by Artin and Stafford~\cite{Artin-Stafford:1995}.  There is
a conjectural classification of non-commutative surfaces due to Artin,
but it is still open. The important special case of
\emph{non-commutative projective planes,} that is, $\tails \Lam$ where
$\Lam$ is a so-called Artin-Schelter regular algebra of
Gelfand-Kirillov dimension $3$ with Hilbert series $(1-t)^{-3}$, has
been completely understood~\cite{Artin-Schelter:1987,Artin-Tate-VdB:1990,
  Artin-Tate-VdB:1991, VandenBergh:1987, Bondal-Polishchuk:1993}.

\section{Resolutions of singularities}\label{sect:resns}

So far I have considered only ``absolute'' situations, that is,
constructions applied to individual rings or categories in isolation.
In the sections to come, I will want to understand certain relative
situations, particularly analogues of resolutions of singularities.
In this \S{} I collect a few definitions and facts about resolutions
of singularities, for easy reference later.  Begin with the
definition.

\begin{definition}
  \label{def:resnofsing}
  Let $X$ be an algebraic variety over a field $k$.  A
  \emph{resolution of singularities} of $X$ is a proper, birational
  morphism $\pi \colon \tilde X \to X$ with $\tilde X$ a non-singular
  algebraic variety.
\end{definition}

Resolutions of singularities are also sometimes called ``smooth
models,'' indicating that the non-singular variety $\tilde X$ is not
too different from $X$: the map is an isomorphism on a dense open set
and is proper, hence surjective.  For curves, construction of
resolution of singularities is easy, as every irreducible curve is
birational to a unique smooth projective curve, namely the
normalization (see \S\ref{sect:normalization}).  For surfaces,
resolutions of singularities still exist in any characteristic, but
now an irreducible surface is birational to infinitely many smooth
surfaces.  This observation is the beginning of the minimal model
program, cf.~\S\ref{sect:MMP}.

Of course existence of resolutions of singularities in any dimension
is a theorem due to Hironaka for $k$ an algebraically closed field of
characteristic zero; in this case the morphism $\pi \colon \tilde X
\to X$ can be taken to be an isomorphism over the smooth locus of $X$,
and even to be obtained as a sequence of blowups of non-singular
subvarieties of the singular locus followed by normalizations.  We
will not need this.

As an aside, I mention here that a proper map between affine schemes
is necessarily finite~\cite[Ex. II.4.6]{Hartshorne}.  It follows
immediately that a resolution of singularities of a singular normal affine
scheme is \emph{never} an affine scheme.

Our other definitions require the canonical sheaf of a singular
variety.  The canonical sheaf $\omega_Y$ of a smooth variety $Y$ has
already appeared, as the sheaf of top differential forms $\w^{\dim
  Y}\Omega_{Y}$ (see the discussion before Theorem~\ref{thm:BOrecon}).
If $Y$ is merely normal, then define $\omega_Y$ to be $j_*
\omega_{Y_\text{reg}}$, where $j$ is the open immersion $Y_\text{reg}
\into Y$ of the smooth locus.  When $Y$ is Cohen-Macaulay, $\omega_Y$
is also a dualizing sheaf~\cite[III.7]{Hartshorne}; in other words, if
the local rings of $Y$ are CM, then the stalks of $\omega_Y$ are
canonical modules in the sense of~\cite{Herzog-Kunz}.  Similarly,
$\omega_Y$ is an invertible sheaf (line bundle) if and only if $Y$ is
Gorenstein.  The Weil divisor $K_Y$ such that $\omega_Y = \caloy(K_Y)$
is called the \emph{canonical divisor}.

The behavior of the canonical sheaf/divisor under certain morphisms is
of central interest.  For example, the Grauert--Riemenschneider
Vanishing theorem describes the higher direct images of $\omega$.

\begin{theorem}
  [GR Vanishing~{\cite{Grauert-Riemenschneider}}]
  \label{thm:GRvan}
  Let $\pi \colon \tilde X \to X$ be a resolution of singularities of
  a variety $X$ over $\CC$.  Then 
  \(
  \R^i \pi_* \omega_{\tilde X} = 0
  \)
  for all $i>0$. \qed
\end{theorem}

Now I come to a pair of words which will be central for the rest of
the article.

\begin{definition}
  \label{def:ratl-crep}
  Let $\pi \colon \tilde X \to X$ be a resolution of singularities of
  a normal variety $X$.
  \begin{enumerate}
  \item\label{item:ratl} Say that $\pi$ is a \emph{rational}
    resolution if $\R^i \pi_* \calo_{\tilde X} = 0$ for $i>0$.
    Equivalently, since $X$ is normal, $\R\pi_* \calo_{\tilde X} =
    \calox$.  In this case $X$ is said to have \emph{rational
      singularities.}
  \item\label{item:crep} Say that $\pi$ is a
    \emph{crepant}\footnote{Obligatory comment on the terminology: the
      word ``crepant'' is due to Miles Reid.  He describes
      it~\cite[p. 330]{Reid:oldperson} as a pun meaning
      ``non-discrepant'', in that the \emph{discrepancy divisor}
      $K_{\tilde X} - \pi^* K_X$ vanishes.} resolution if $\pi^*
    \omega_X = \omega_{\tilde X}$.
  \end{enumerate}
\end{definition}

Crepancy is a condition relating the two ways of getting a sheaf on
$\tilde X$ from one on $X$, namely via $\Hom$ and via $\otimes$.  To
get an idea what this condition is, consider a homomorphism of CM
local rings $R \to S$ such that $S$ is a finitely generated
$R$-module.  Let $\omega_R$ be a canonical module for $R$.  Then one
knows that the ``co-induced'' module $\Ext_R^t(S,\omega_R)$, where $t
= \dim R-\dim S$, is a canonical module for $S$~\cite[3.3.7]{BH}.  The
``induced'' module $S \otimes_R \omega_R$ is not necessarily a
canonical module.  Back in the geometric world, $\pi^* \omega_{\Spec
  R}$ corresponds to $S \otimes_R \omega_R$, so the assumption that
this is equal to $\omega_S$ is locally a condition of the form
$\Ext_R^t(S,\omega_R) \cong S \otimes_R \omega_R$.  When $X$ is
Gorenstein, i.e.\ $\omega_X \cong \calox$, a crepant resolution
$\tilde X$ is also Gorenstein.

One of the main motivations for considering crepant resolutions of
singularities comes from the study of Calabi-Yau varieties, which in
particular have trivial canonical sheaves.  In this case, if one wants a
resolution $\pi \colon \tilde X \to X$ in which $\tilde X$ is also
Calabi-Yau, then  $\pi$ needs to be crepant.

A \emph{small} resolution, that is, one for which the exceptional
locus has codimension at least two, is automatically crepant.  This is
a very useful sufficient condition.

The next proposition follows from GR
vanishing~\cite[p. 50]{Kempf-Knudsen-Mumford}.

\begin{proposition}
  \label{prop:crep-ratl}
  Let $X$ be a complex algebraic variety and let $\pi \colon \tilde X
  \to X$ be a resolution of singularities.
  \begin{enumerate}
  \item $X$ has rational singularities if and only if $X$ is CM and
    $\pi_*\omega_{\tilde X} = \omega_{X}$.
  \item \label{item:crep-ratl} If $X$ is Gorenstein and has a crepant
    resolution of singularities, then $X$ has rational
    singularities. \qed
  \end{enumerate}
\end{proposition}

Not every rational singularity has a crepant resolution.  Here are two
examples.

\begin{example}
  \label{eg:Lin}
  Let $R$ be the diagonal hypersurface ring
  $\CC[x,y,z,t]/(x^3+y^3+z^3+t^2)$.  Then $R$ is quasi-homogeneous
  with the variables given weights $2$, $2$, $2$, and $3$.  The
  $a$-invariant of $R$ is thus $6-(2+2+2+3) = -3 <0$, and $R$ has
  rational singularities by Fedder's criterion~\cite[Example
  3.9]{Huneke:CBMS}.  However, Lin~\cite{Lin:2002} shows that a
  diagonal hypersurface defined by $x_0^{r}+x_1^{d}+\cdots +x_d^d$ has
  a crepant resolution of singularities if and only if $r$ is
  congruent to $0$ or $1$ mod $d$.
\end{example}

\begin{example}
  \label{eg:ADE}
  Quotient singularities $X = Y/G$, where $Y$ is smooth and $G$ is a
  finite group of automorphisms, have rational
  singularities~\cite{Viehweg:1977}.  Consider quotient singularities
  $\CC^n/G$, where $G \subset \SL(n,\CC)$ is finite.  These are
  by~\cite{Watanabe:1974} the Gorenstein quotient singularities. 

  If $n=2$, the results are the \emph{rational double points}, also
  known as Kleinian singularities or Du Val singularities, which are
  the quotient singularities $X = \CC^2/G = \Spec (\CC[u,v]^G)$, where
  $G \subset \SL(2,\CC)$ is a finite subgroup.  These are also
  described as ADE hypersurface rings $\CC[x,y,z]/(f(x,y,z))$ with
  explicit equations as follows.
  \begin{equation}\label{eq:ADE}
    \begin{split}
      (A_n):& \qquad x^2 + y^{n+1} + z^2\,, \qquad n \geq 1\\
      (D_n):& \qquad x^2y + y^{n-1} + z^2\,, \qquad n \geq 4\\
      (E_6):& \qquad x^3 + y^4 + z^2\\
      (E_7):& \qquad x^3 + xy^3 + z^2\\
      (E_8):& \qquad x^3 + y^5 + z^2
    \end{split}
  \end{equation}
  For these singularities, a crepant resolution always exists and is
  unique.  In fact, a normal affine surface singularity $R$ over $\CC$
  admits a crepant resolution if and only if every local ring of $R$
  is (at worst) a rational double point.  I will return to the
  rational double points in \S\ref{sect:McKay} below.

  If $n=3$, $\CC^3/G$ always has a crepant resolution as well, though
  they are no longer unique, thanks to the existence of \emph{flops}
  (see the next \S{}).  There is a classification of the finite
  subgroups of $\SL(3,\CC)$ up to conjugacy, and existence of crepant
  resolutions was verified on a case-by-case basis by
  Markushevich~\cite{Markushevich:1997}, Roan~\cite{Roan:1994,
    Roan:1996}, Ito~\cite{Ito:1995a,Ito:1995b}, and
  Ito--Reid~\cite{Ito-Reid:1996}.  See Theorem~\ref{thm:BKR} below for
  a unified statement.
  
  For $n\geq 4$, quotient singularities need not have crepant resolutions
  of singularities.  For example, the quotient of $\CC^4$ by the
  involution $(x,y,z,w) \mapsto (-x,-y,-z,-w)$ admits no crepant
  resolution~\cite[Example 5.4]{Reid:LaCorrMcKay}.
\end{example}

\section{The minimal model program}\label{sect:MMP}

A key motivation for categorical desingularizations in general, and
non-commutative crepant resolutions in particular, is the
\emph{minimal model program} of Mori and Reid.  This is an attempt to
find a unique ``best'' representative for the birational equivalence
class of any algebraic variety.  For curves, this is obvious, since
there is in each equivalence class a unique smooth projective
representative.

It is also the case that every surface is birationally equivalent to a
smooth projective surface, but now matters are complicated by the fact
that the blowup of a smooth surface at a point is again
smooth. However, every birational morphism of surfaces factors as a
sequence of blowups, so must have a {$(-1)$-curve}, that is, a
rational curve $C \cong \PP^1$ with self-intersection $-1$, lying over
a smooth point.  One can compute that if $C$ is a $(-1)$-curve on a
surface $X$, then $K_X \cdot C = -1$, where $K_X$ is the canonical
divisor.

By Castelnuovo's criterion, a $(-1)$-curve can always be blown down,
essentially undoing the blowup.  The algorithm for obtaining a minimal
model is thus to contract all the $(-1)$-curves, and one obtains the
classification of minimal models for
surfaces~\cite[V.5]{Hartshorne}: the result of the algorithm is a
smooth projective surface $S$ which is either $\PP^2$, a ruled surface
over a curve (the ``Fano'' case), or such that $K_S \cdot C \geq 0$
for every curve $C$ in $S$.  In this last case say that $K_S$ is
\emph{nef.}

The minimal model program is a framework for extending this
simple-minded algorithm to one that will work for threefolds and
higher-dimensional varieties.  The theory turns out to be much richer,
in part because it turns out that one must allow minimal models to be
a little bit singular.  Here ``a little bit'' means in codimension
$\geq 2$. Precisely, a projective variety $X$ is a \emph{minimal
  model} if every birational map $Y \dashrightarrow X$ is either a
contraction of a divisor to a set of codimension at least two, or is
an isomorphism outside sets of codimension at least
two~\cite{Kollar:WhatIsMinMod}.  There are compelling reasons to allow
singular minimal models; for example, there exists a
three-dimensional smooth variety which is not birational to any smooth
variety with nef canonical divisor~\cite{MathOverflow:flips}.  Mori
and Reid realized that this meant minimal models need not be smooth;
they can be taken to be \emph{terminal} instead.

I won't worry about the technical definitions of terminal and
canonical singularities here, but only illustrate with a class of
examples.  A diagonal hypersurface singularity defined by $x_0 ^{a_0} +
x_1^{a_1} + \cdots + x_d^{a_d}$ is
\begin{enumerate}
\item \emph{canonical} if and only if $a_1 + \cdots + a_d > 1$, and
\item \emph{terminal} if and only if $a_1 + \cdots + a_d > 1 + \frac{1}{\lcm(a_i)}$.
\end{enumerate}
For Gorenstein singularities, canonical singularities are the same as
rational singularities, so
Proposition~\ref{prop:crep-ratl}(\ref{item:crep-ratl}) says that the
existence of a crepant resolution implies canonical singularities.

In this language, a projective variety $X$ is a minimal model if and
only if it is $\QQ$-factorial (i.e.\ the divisor class group of every
local ring is torsion), has nef canonical divisor, and has terminal
singularities.

In dimension two, minimal models are unique up to isomorphism by
definition.  Terminal surface singularities are smooth, and the canonical
surface singularities are the rational double points of
Example~\ref{eg:ADE}~\cite[(2.6.2)]{Kollar:Bourbaki}.  

In dimension three, terminal singularities are well-understood, cf.\
\cite{Reid:1983} or~\cite[2.7]{Kollar:Bourbaki}. The Gorenstein ones
are precisely the isolated compound Du Val (cDV) singularities.
(Recall that a cDV singularity is a hypersurface defined by $f(x,y,z)
+ t g(x,y,z,t)$, where $f$ is a simple singularity as
in~\eqref{eq:ADE} and $g$ is arbitrary.)  However, minimal models of
threefolds are no longer unique~\cite{Corti:WhatIsFlip}.  Here is the
simplest example. 

\begin{example}
  [The ``classic flop''] \label{eg:Atiyah} Let $X$ be the
  three-dimensional ($A_1$) singularity over $\CC$, so $X = \Spec
  \CC[u,v,x,y]/(uv-xy)$.  Consider the blowup $f \colon Y \to X$ of
  the plane $u=x=0$.  It's easy to check that $Y$ is smooth, and
  that $f \colon Y \to X$ is a birational map which contracts a line
  $L \cong \PP^1$ to the origin. Thus $f$ is a small resolution,
  whence crepant.  Furthermore $Y$ is a minimal model. 

  One could also have considered the plane $u=y=0$ and its blowup $f'
  \colon Y' \to X$.  Symmetrically, $Y'$ is smooth, $f'$ contracts a
  line $L' \cong \PP^1$ and is crepant, and $Y'$ is a minimal model.

  The resolutions $Y$ and $Y'$ are almost indistinguishable, but they
  are not isomorphic over $X$.  One can check that the birational
  transforms of the plane $u=x=0$ to $Y$ and $Y'$ have intersection
  number $+1$ with $L$ and $-1$ with $L'$.

  On the other hand, the induced birational map $\phi \colon Y
  \dashrightarrow Y'$ is an isomorphism once one removes $L$ from $Y$
  and $L'$ from $Y"$. This $\phi$ is called a (or ``the classic'')
  \emph{flop.}  It is also sometimes called the ``Atiyah flop''
  after~\cite{Atiyah:1958}, though Reid traces it back through work of
  Zariski in the 1930s, and assigns it a birthdate of around 1870.

  Let $Z$ be the blowup of the origin of $X$.  Then $Z$ is in fact the
  closed graph of $\phi$ and there is a diagram
  \[
  \xymatrix{
    & Z \ar[dl] \ar[dr] \\
    Y \ar[dr]_f \ar@{-->}[rr]^\phi& & Y' \ar[dl]^{f'}\\
    & X }
  \] 
  The exceptional surface of $Z \to X$ is the quadric $Q = \PP^1
  \times \PP^1$, which  is cut out by two families of lines.
  The lines $L$ and $L'$ are the contractions of $Q$ along these two
  rulings, and conversely $Q$ is the blowup of $L \subset Y$, resp.\
  $L' \subset Y'$.
\end{example}

The next definition is a special case of the usual
definition of a flop~\cite[6.10]{Kollar-Mori:1998} (in general, one
need not assume $Y$ and $Y'$ are smooth, nor that $X$ is Gorenstein).

\begin{definition}
  \label{def:flipflop}
  Let $Y$ and $Y'$ be smooth projective varieties.  A birational map
  $\phi \colon Y \dashrightarrow  Y'$ is a \emph{flop} if
  there is a diagram
  \[
  \xymatrix{
    Y \ar@{-->}[rr]^\phi \ar[dr]_f && Y' \ar[dl]^{f'}\\
    &X
  }
  \]
  where $X$ is a normal projective Gorenstein variety, $f$ and $f'$
  are small resolutions of singularities, and there is a divisor $D$
  on $Y$ such that, if $D'$ is the strict transform of $D$ on $Y'$,
  then $-D'$ is ample.

  Say $\phi \colon  Y \dashrightarrow Y'$ is a \emph{generalized
    flop} if for some (equivalently, for every) diagram
  \[
  \xymatrix{
    & Z \ar[dl]_\pi \ar[dr]^{\pi'}\\
    Y  \ar@{-->}[rr]^\phi& & Y' 
  }\] 
  with $Z$ smooth, there is  an equality $\pi^*K_Y ={\pi'}^*K_{Y'}$.
\end{definition}

It is known that the existence of a crepant resolution forces
canonical singularities, so that in particular if $X$ participates in
a flop as above, it has canonical singularities.  On the other hand,
if $X$ is $\QQ$-factorial and has terminal singularities, then it can
have no crepant resolution of singularities~\cite[Corollary
4.11]{Kollar:1989} (this is one explanation of the name ``terminal'').

Bondal and Orlov~\cite{Bondal-Orlov:2002} observed that one ingredient
of the minimal model program, namely the blowup $\tilde X$ of a smooth
variety $X$ at a smooth center, induces a fully faithful functor on
derived categories $\D^b(\coh X) \to \D^b(\coh \tilde X)$.  They
propose that each of the operations of the program should induce such
fully faithful embeddings.  In particular, they make the following
conjecture.

\begin{conj}
  [Bondal--Orlov]
  \label{conj:BO}
  For any generalized flop $\phi \colon Y \dashrightarrow Y'$ between
  smooth varieties, there is an equivalence of triangulated categories
  $\catf \colon \D^b(\coh Y') \to \D^b(\coh Y)$.
\end{conj}

Notice that even though there always exists the natural Fourier-Mukai
type functor $\R\pi_* \L{\pi'}^*(-) \colon \D^b(\coh Y') \to \D^b(\coh
Y)$, this is known not to be fully faithful in general, so some new
idea is needed.

Bondal and Orlov proved Conjecture~\ref{conj:BO} in some special cases
in dimension three, and Bridgeland~\cite{Bridgeland:2002} gave a
complete proof for threefolds.  Here is Bridgeland's result.

\begin{theorem}
  [Bridgeland]
  \label{thm:bridgeland}
  Let $X$ be a projective complex threefold with terminal
  singularities.  Let $f \colon Y \to X$ and $f' \colon Y' \to X$ be
  crepant resolutions of $X$.  Then $\D^b(\coh Y) \simeq \D^b(\coh
  Y')$. \qed 
\end{theorem}
The equivalence in this theorem is a Fourier-Mukai type functor of the
form $\R f_* (\calp \lotimes {f'}^* (-))$, where $\calp$ is a
well-chosen object of $\D^b(\coh (Y \times_X Y'))$.  In fact the
construction of $\calp$ is very difficult and is the heart of the
proof.

\section{Categorical desingularizations}\label{sect:cat-desing}

Now let us combine the philosophical ramblings of
\S\ref{sect:non-existence} with the concrete problems of
\S\S\ref{sect:resns} and~\ref{sect:MMP}.  Treating commutative and
non-commutative varieties---in the form of their derived
categories---on equal footing, one can entertain the notion of a
\emph{resolution of a commutative algebraic variety by a
  non-commutative one.}  Bondal and Orlov~\cite{Bondal-Orlov:2002}
seem to have been the first to articulate such a possibility in pure
mathematics.  Other authors have considered modified or specialized
versions, e.g.~\cite{Bezrukavnikov:2006, Kuznetsov:2008, Lunts:2010}.

To begin, let us consider resolutions of singularities from a
categorical point of view.  Let $X$ be a normal algebraic variety, and
let $\pi \colon \tilde X \to X$ be a resolution of singularities.
There are two natural functors between derived categories, namely the
derived pushforward $\R \pi_* \colon \D^b(\coh \tilde X) \to \D^b(\coh
X)$ and the derived pullback $\L\pi^* \colon \D(\coh X) \to \D(\coh
\tilde X)$.  The derived pullback may not take bounded complexes to
right-bounded ones, so does not generally give a functor on $\D^b$.
One could restrict $\L \pi^*$ to the {perfect complexes} over $X$
and write instead $\L \pi^* \colon \D^{\perf}(\coh X) \to
\D^{\perf}(\coh \tilde X) = \D^b(\coh X)$.

The pullback and pushforward form an adjoint pair.  If $X$ is assumed
to have rational singularities, much more can be said.  For an object
$\cale$ in $\D^b(\coh \tilde X)$ and a perfect complex $\calp$ over
$X$, the derived projection formula gives
\[
\R\pi_*(\cale \lotimes_{\calo_{\tilde X}} \L\pi^* \calp) = \R\pi_*
\cale \lotimes_{\calox} \calp\,.
\]
In particular, setting $\cale = \calo_{\tilde X}$ and taking into
account $\R \pi_* \calo_{\tilde X} = \calox$, this yields
\[
\R\pi_* \L \pi^* \calp = \calp
\]
for every perfect complex $\calp$ in $\D^b(\coh X)$.  Otherwise said,
$\R \pi_* \colon \D^b(\coh \tilde X) \to \D^b(\coh X)$ identifies the
target with the quotient of the source by the kernel of $\R \pi_*$.
Bondal and Orlov propose to take this as a template:
\begin{definition}
  [Bondal--Orlov]\label{def:cat-desing}
  A \emph{categorical desingularization} of a
  triangulated category $\D$ is an abelian category $\catc$ of finite
  homological dimension and a triangulated subcategory $\catk$ of
  $\D^b(\catc)$, closed under direct summands, such that
  $\D^b(\catc)/\catk \simeq \D$.
\end{definition}

One problem with this definition is the assumption that $\catc$ have
finite homological dimension.  As observed in
\S\ref{sect:non-existence}, this is a very weak condition when $\catc$
is the category of modules over a non-commutative ring.  There are a
number of proposals for a better---that is, more restrictive---notion
of \emph{smoothness} for a (triangulated) category, but as far as I
can tell, no consensus on a best candidate~\cite{Kontsevich-Soibelman,
  Toen-Vaquie, Kuznetsov:2008, Lunts:2010}.

As an aside, I note here that the condition for $\pi\colon \tilde X
\to X$ to be crepant can be translated into categorical language as
``the right adjoint functor $\pi^{!}$, which is locally represented by
$\Hom_\calox(\calo_{\tilde X},-)$, is isomorphic to $\pi^*$.''  We
won't need this.

Let us reconsider Example~\ref{eg:Atiyah} from the point of view of
categorical geometry. This can be thought of as a warmup for
\S\ref{sect:detX}.

\begin{example}
  \label{eg:Atiyah2}
  Set $R = \CC[u,v,x,y]/(uv-xy)$, so that $X= \Spec R$ is the
  three-dimensional ordinary double point as in
  Example~\ref{eg:Atiyah}.  Let $I = (u,x)$ and $I' = (u,y)$.  Then in
  fact $I' = I^{-1} = I^* = \Hom_R(I,R)$ is the dual of $I$.  Notice
  too that $\End_R(I) = R$, either by direct computation or by
  Theorem~\ref{thm:Grauert-Remmert} below.

  Let $f \colon Y \to X$ and $f' \colon Y ' \to X$ be the blowups of
  $I$ and $I'$ as before.  On $Y$, consider the locally free sheaf
  $\cale = \caloy \oplus \caloy(1)$, which is the pullback of $\calox
  \oplus \cali$, where $\cali$ is the ideal sheaf of $I$.
  Straightforward calculations (or see \S\ref{sect:detX}) show that
  $\cale$ is a \emph{tilting bundle} on $Y$
  (Definition~\ref{def:tilt-obj}), and hence
  $\rHom_{\caloy}(\cale,-)\colon \D^b(\coh Y) \to \D^b(\lmod \Lam)$ is
  an equivalence, where $\Lam = \End_\caloy(\cale)$.  Furthermore, we
  have 
  \[
  \Lam \cong f_* \cEnd_\caloy(\cale) = \End_R(R \oplus I)\,,
  \]
  which can also be written as a block-matrix ring
  \[
  \Lam = 
  \begin{pmatrix}R & I \\ I^{-1} & \End_R(I)=R\end{pmatrix}\,.
  \]
  The induced functor $\D^b(\lmod \Lam) \to \D^b(\coh X)$ is then
  obviously a categorical desingularization.

  Repeating the construction above with $\cale' = \calo_{Y'} \oplus
  \calo_{Y'}(1)$ on $Y'$, one obtains $\Lam' = \End_R(R \oplus I')$.
  But since $I' = I^{-1}$, $\Lam'$ is isomorphic to $\Lam$.  This
  implies equivalences 
  \[
  \D^b(\coh Y) \simeq \D^b(\lmod \Lam) \simeq \D^b(\coh Y')\,.
  \]
\end{example}

Inspired by the example above and others from the minimal model
program, Bondal and Orlov expect that for a singular variety $X$, the
category $\D^b(\coh X)$ should have a \emph{minimal} categorical
desingularization, i.e.\ one embedding in any other.  Such a
category would be unique up to derived equivalence.  They propose in
particular the following conjecture.

\begin{conj}
  [Bondal--Orlov~\cite{Bondal-Orlov:2002}]\label{conj:BO2}
  Let $X$ be a complex algebraic variety with canonical singularities
  and let $f \colon Y \to X$ be a finite morphism with $Y$ smooth.
  Then $\cala = \cEnd_\calox(f_* \caloy)$ gives a minimal categorical
  desingularization, in the sense that $\lmod \cala$ has finite global
  dimension and if $\tilde X \to X$ is any other resolution of
  singularities of $X$, then there exists a fully faithful embedding
  $\D^b(\lmod \cala) \to \D^b(\coh \tilde X)$.  Moreover, if $\tilde X
  \to X$ is crepant, then the embedding is an equivalence. 
\end{conj}

In the next \S{} I will consider another family of examples providing
strong evidence for this conjecture.

\section{Example: the McKay correspondence}\label{sect:McKay}

In this \S{} I sketch a main motivating example, already
foreshadowed in Example~\ref{eg:ADE}.  The finite subgroups of
$\SL(2,\CC)$ were carefully studied by Klein in the 1880s, and the
resolutions of the corresponding singularities $\CC^2/G = \Spec
\CC[u,v]^G$ were understood by Du Val in the 1930s.  The structure of
the resolution faithfully reflects the representation theory of the
group $G$, as observed by McKay~\cite{McKay}, and the correspondence
naturally extends to the reflexive modules over the (completed)
coordinate ring $\CC[\![u,v]\!]^G$.  Even more, there is a natural
resolution of singularities of the quotient singularity, built from
the group $G$, which is derived equivalent to a certain
non-commutative ring built from these reflexive modules.  Thus the
group $G$ already knows the geometry of $\CC^2/G$ and its resolution of
singularities.

This \S{} is about this circle of ideas, which together go by
the name ``McKay correspondence.''  I consider first, more generally,
finite subgroups $G \subset \GL(n,k)$ with $n \geq 2$ and $k$ a field
of characteristic relatively prime to $|G|$.  Then I specialize to
$n=2$ and subgroups of $\SL$, where the strongest results hold.
See~\cite{Leuschke-Wiegand:BOOK} or~\cite{Yoshino:book} for proofs.

Let $S = k[\![x_1, \dots, x_n]\!]$ be a power series ring over an
algebraically closed field $k$ with $n \geq 2$.  Let $G \subset
\GL(n,k)$ be a finite subgroup with order invertible in $k$.  Make $G$
act on $S$ by linear changes of variables, and set $R = S^G$, the ring
of invariants. The ring $R$ is noetherian, local, and complete, of
dimension $n$.  It is even CM by the Hochster--Eagon
theorem~\cite{Hochster-Eagon:1971}.  Furthermore, $S$ is a
module-finite $R$-algebra, and is a maximal Cohen--Macaulay
$R$-module.

The central character in the story is the skew, or twisted, group
algebra $\SG$. As an $S$-module, $\SG$ is free on the elements of $G$,
and the product of two elements $s \cdot \sigma$ and $t \cdot \tau$,
with $s,t \in S$ and $\sigma, \tau \in G$, is defined by
\(
(s \cdot \sigma) (t \cdot \tau) = s \sigma(t) \cdot \sigma \tau\,.
\)
Thus moving $\sigma$ past $t$ ``twists'' the ring element.

Left modules over $\SG$ are precisely $S$-modules with a compatible
action of $G$, and one computes $\Hom_{\SG}(M,N) = \Hom_S(M,N)^G$ for
$\SG$-modules $M$ and $N$. Since the order of $G$ is invertible,
taking invariants is an exact functor, whence $\Ext_{\SG}^i(M,N) =
\Ext_S^i(M,N)^G$ for all $i>0$ as well.  It follows that an
$\SG$-module $P$ is projective if and only if it is free over
$S$. This, together with a moment's contemplation of the
($G$-equivariant) Koszul complex over $S$ on $x_1, \dots, x_n$, gives
the following observation.

\begin{prop}
  \label{prop:SGfingldim}
  The twisted group ring $\SG$, where $S = k[\![x_1,\dots, x_n]\!]$
  and $G$ is a finite group of linear automorphisms of $S$ with order
  invertible in $k$, has finite global dimension equal to $n$.  \qed
\end{prop}

The ``skew'' multiplication rule in $\SG$ is cooked up precisely so
that the homomorphism $\gamma \colon \SG \to \End_R(S)$, defined by
$\gamma(s\cdot \sigma)(t) = s \sigma(t)$, is a ring homomorphism
extending the group homomorphism $G \to \End_R(S)$ defining the action
of $G$ on $S$.  In general, $\gamma$ is neither injective nor
surjective, but under an additional assumption on $G$, it is both.
Recall that a \emph{pseudo-reflection} is an element $\sigma \in
\GL(n,k)$ of finite order which fixes a hyperplane. 

\begin{theorem}
  [Auslander~{\cite{Auslander:1962, Auslander:rationalsing}}]
  \label{thm:Aus-McKay}
  Let $S = k[\![x_1, \dots, x_n]\!]$, $n \geq 2$, let $G \subset
  \GL(n,k)$ be a finite group acting on $S$, and assume $|G|$ is
  invertible in $S$.  Set $R = S^G$.  If $G$ contains no non-trivial
  pseudo-reflections then the homomorphism $\gamma \colon
  \SG \to \End_R(S)$ is an isomorphism.  

  Consequently, in this case $\End_R(S)$ has finite global dimension
  and as an $R$-module is isomorphic to a direct sum of copies of $S$,
  so in particular is a MCM $R$-module.  \qed
\end{theorem}
The condition that $G$ contain no non-trivial pseudo-reflections is
equivalent to the extension $R \into S$ being unramified in
codimension one~\cite[Lemma 10.7]{Yoshino:book}. 

Let $\varrho \colon G \to \GL(W)$ be a representation of $G$ on
the finite-dimensional $k$-vector space $W$.  Then $S \otimes_k W$,
with the diagonal action of $G$, is a finitely generated $\SG$-module.
It is free over $S$, whence projective over $\SG$.  The submodule of
fixed points, $M_\varrho=(S \otimes_k W)^G$, is naturally an $R$-module.  If
$\varrho$ is irreducible, then one can show that $M_\varrho$ is a
direct summand of $S$ as an $R$-module.
Conversely, given any $R$-direct summand of $S$, the corresponding
idempotent in $\End_R(S)$ defines an $\SG$-direct summand $P$ of $\SG$,
whence a representation $P/(x_1, \dots, x_n) P$ of $G$. 

\begin{corollary}
  \label{cor:McKay-gen}
  These operations induce equivalences between the categories
  $\add_R(S)$ of $R$-direct summands of $S$, $\add \End_R(S)$ of
  finitely generated projective $\End_R(S)$-modules, $\add \SG$ of
  finitely generated projective $\SG$-modules, and $\rep_k G$ of
  finite-dimensional representations of $G$.\qed
\end{corollary}

As a final ingredient, define a quiver from the data of the 
representation theory of $G$, or equivalently---given the
correspondences above---of the $R$-module structure of $S$.
\begin{definition}[McKay~\cite{McKay}]
  \label{def:McKayquiver}
  The \emph{McKay quiver} of $G \subset \GL(n,k)$ has vertices
  $\varrho_0, \dots, \varrho_d$, a complete set of the non-isomorphic
  irreducible $k$-representations of $G$, with $\varrho_0$ the trivial
  irrep.  Denote by $\varpi$ the given $n$-dimensional representation
  of $G$ as a subgroup of $\GL(n,k)$.  Then draw $m_{ij}$ arrows
  $\varrho_i \to \varrho_j$ if the multiplicity of $\varrho_i$ in
  $\varpi \otimes_k \varrho_j$ is equal to $m_{ij}$.
\end{definition}

Now let us specialize to the case $n=2$.  Here the MCM $R$-modules are
precisely the reflexive ones.  This case is unique thanks to the
following result, which fails badly for $n\geq 3$.

\begin{lemma}
  [Herzog~{\cite{Herzog:1978}}] \label{lem:herzog} Let $S = k[\![u,
  v]\!]$, let $G\subset \GL(2,k)$ be a finite group of order
  invertible in $k$, and let $R = S^G$.  Then every finitely generated
  reflexive $R$-module is a direct summand of a direct sum of copies
  of $S$ as an $R$-module.  In particular, the MCM $R$-modules
  coincide with $\add_R(S)$, and there are only finitely many
  indecomposable ones.\qed
\end{lemma}
The one-one correspondences that hold for arbitrary $n$ can thus be
augmented in dimension two, giving a correspondence between the
irreducible representations of $G$ and the indecomposable MCM
$R$-modules.

Specialize one last time, to assume now that $G \subset \SL(2,k)$.
Note that then $G$ automatically contains no non-trivial
pseudo-reflections.  Furthermore, $R=S^G$ is Gorenstein by a result of
Watanabe~\cite{Watanabe:1974}; in fact, it is
classical~\cite{Klein:1884} that $\Spec R$ embeds as a hypersurface in
$k^3$, so $R \cong k[\![x,y,z]\!]/f(x,y,z)$ for some polynomial $f$.
As long as $k$ has characteristic not equal to $2$, $3$, or $5$, the
polynomials arising are precisely the ADE polynomials
of~\eqref{eq:ADE} defining the rational double points.  

The rational double points are distinguished among normal surface
singularities by the fact that their local rings have unique crepant
resolutions of singularities, which are the minimal resolutions of
singularities.  They are particularly easy to compute, being achieved
by a sequence of blowups of points (no normalization required).  The
preimage of the singular point is a bunch of rational curves $E_1,
\dots, E_n$ on the resolution.  These curves define the \emph{dual
  graph of the desingularization}: it has for vertices the irreducible
components $E_1, \dots, E_n$, with an edge joining $E_i$ to $E_j$ if
$E_i \cap E_j \neq 0$.  This graph is related to the other data as
follows.
\begin{theorem}
  [Classical McKay Correspondence]
  \label{thm:McKay2}
  Let $k$ be an algebraically closed field of characteristic not $2$,
  $3$, or $5$, and let $G \subset \SL(2,k)$ be a finite subgroup of
  order invertible in $k$.  Set $S = k[\![u,v]\!]$, with a natural
  linear action of $G$, set $R = S^G$, and let $\pi \colon \tilde X
  \to \Spec R$ be the minimal resolution of singularities with
  exceptional curves $E_1, \dots, E_n$.  Then
  \begin{enumerate}
  \item There is a one-one correspondence between
    \begin{enumerate}
    \item the exceptional curves $E_i$;
    \item the irreducible representations of $G$; and
    \item the indecomposable MCM $R$-modules.
    \end{enumerate}
  \item {}(McKay) The dual graph of the desingularization is
    isomorphic to the McKay quiver after deleting the trivial vertex
    and replacing pairs of opposed arrows by edges.  It is an ADE
    Coxeter-Dynkin diagram.
\[
\begin{array}{cc} 
{A}_n : & \xymatrix@R=0.2cm@C=0.3cm{ 
\circ \ar@{-}[r] & \circ \ar@{-}[r] & \cdots \ar@{-}[r] &
  \circ \ar@{-}[r] & \circ} \\ 
\raisebox{-0.7cm}{${D}_n :$} & \xymatrix@R=0.2cm@C=0.3cm{ 
 & & & & & \circ \\
 \circ \ar@{-}[r] & \circ \ar@{-}[r] & \cdots \ar@{-}[r] &
 \circ \ar@{-}[r] & \circ \ar@{-}[ru] \ar@{-}[rd] \\
 & & & & & \circ} \\
{E}_6 : & \xymatrix@R=0.2cm@C=0.3cm{
\circ \ar@{-}[r] & \circ \ar@{-}[r] & \circ \ar@{-}[r] \ar@{-}[d] &
  \circ \ar@{-}[r] & \circ \\ 
& & \circ & &  } \\
{E}_7 : & \xymatrix@R=0.2cm@C=0.3cm{
\circ \ar@{-}[r] & \circ \ar@{-}[r] & \circ \ar@{-}[r] \ar@{-}[d] &
  \circ \ar@{-}[r] & \circ \ar@{-}[r] & \circ \\ 
& & \circ & & &} \\
{E}_8 : & \xymatrix@R=0.2cm@C=0.3cm{
\circ \ar@{-}[r] & \circ \ar@{-}[r] & \circ \ar@{-}[r] \ar@{-}[d] &
  \circ \ar@{-}[r] & \circ \ar@{-}[r] & \circ \ar@{-}[r] & \circ  \\ 
 & & \circ }
\end{array}
\] \qed
  \end{enumerate}
\end{theorem}

Shortly after McKay's original observation~\cite{McKay} of the
isomorphism of graphs above, Gonzalez-Sprinberg and
Verdier~\cite{Gonzalez-Sprinberg-Verdier} gave, in characteristic
zero, a geometric construction linking the representation
theory of $G$ and the resolution of singularities $\tilde X$.  Later
constructions by Artin--Verdier~\cite{Artin-Verdier:1985},
Esnault~\cite{Esnault:1985}, and Kn\"orrer~\cite{Knorrer:1985} made
explicit the correspondences between the exceptional curves $E_i$, the
indecomposable reflexive $R$-modules, and the irreducible
representations of $G$.

The first intimation of a ``higher geometric McKay correspondence''
appeared in string theory in the mid-1980s.  Dixon, Harvey, Vafa, and
Witten~\cite{DHVW} observed that for certain $G \subseteq \SL(3,\CC)$,
and a certain crepant resolution $\tilde X \to \CC^3/G$, there is an
equality between the Euler characteristic $\chi(\tilde X)$ and the
number of conjugacy classes (= number of irreducible representations)
of $G$.  There followed a great deal of work on the existence of
crepant resolutions of singularities for quotient singularities of the
form $Y/G$, where $Y$ is an arbitrary smooth variety of dimension two
or more.  Specifically, one can ask for the existence of a crepant
resolution $\tilde X \to Y/G$ and a derived equivalence between
$\tilde X$ and the $G$-equivariant coherent sheaves on $Y$.  Let
$\D^b_G(Y)$ denote the bounded derived category of the latter.

Such an equivalence was first constructed by Kapranov and Vasserot in
the setting of Theorem~\ref{thm:McKay2}.  In this case, the minimal
resolution of singularities $\tilde X$ has an alternative
construction, as \emph{Nakamura's $G$-Hilbert scheme}
$\Hilb^G(\CC^2)$~\cite{Nakamura:2001, Ito-Nakamura:1999}.  This is an
irreducible component of the subspace of the Hilbert scheme of points
in $\CC^2$ given by the ideal sheaves $\cali \subseteq \calo_{\CC^2}$
such that the $\calo_{\CC^2}/\cali \cong \CC[G]$ as $G$-modules.  

\begin{theorem}
  [Kapranov--Vasserot~{\cite{Kapranov-Vasserot:2000}}]
  \label{thm:kap-vass}
  Let $G \subset \SL(2, \CC)$ be a finite group, $S = \CC[u,v]$, $R
  = S^G$, and $X = \Spec R$.  Set $H = \Hilb^G(\CC^2)$.  Then there is
  a commutative triangle 
  \[
  \xymatrix{
    **[l]\D^b(\lmod{\SG}) = \D^b_G(\coh \CC^2) \ar[rr]^-\Phi \ar[dr] &&
    \D^b(\coh H) \ar[dl]\\
    & \D^b(\coh X)
  }
  \]
  in which $\Phi$ is an equivalence of triangulated categories.  \qed
\end{theorem}
The equivalence $\Phi$ is given by an explicit ``equivariant''
Fourier-Mukai type functor, $\Phi(-) = (\R p_* \L q^*(-))^G$, where $Z
\subseteq X \times \CC^2$ is the incidence variety and $p,q$ are the
projections onto the factors.

In dimension greater than two, there is no minimal resolution of
singularities.  However, Nakamura's $G$-Hilbert scheme is still a
candidate for a crepant resolution of singularities in dimension
three.  Bridgeland, King, and Reid proved the following general
result about the $G$-Hilbert scheme.

\begin{theorem}
  [Bridgeland--King--Reid~{\cite{Bridgeland-King-Reid:2001}}]
  \label{thm:BKR}
  Suppose that $Y$ is a smooth and quasi-projective complex variety, and
  that $G \subseteq \Aut Y$ is a finite group of automorphisms such
  that the quotient $Y/G$ has Gorenstein singularities.  Let $H =
  \Hilb^G(Y)$.  If
  \[
  \dim \left(H \times_{Y/G} H\right) \leq \dim H +1\,,
  \]
  then $H$ is a crepant resolution of singularities of $Y/G$ and there
  is an equivalence (explicitly given by a Fourier-Mukai functor) of
  derived categories $\D^b(\Hilb^G(Y)) \to \D^b_G(Y)$, where
  $\D^b_G(Y)$ is the bounded derived category of $G$-equivariant
  coherent sheaves on $Y$.  \qed
\end{theorem}

The assumption on the fiber product $H \times_{Y/G} H$ is automatic if
$\dim Y \leq 3$, so this result implies a derived McKay correspondence
for three-dimensional quotient singularities $\CC^3/G$ with $G
\subset \SL(3,\CC)$.  In particular, such singularities have a
crepant resolution, which had been verified on a case-by-case basis
using the classification of finite subgroups of $\SL(3,\CC)$.  The
full details of the correspondences in dimension three are still being
worked out~\cite{Cautis-Logvinenko:2009}.

In dimension four, the hypothesis on $H \times_{Y/G} H$ need not hold
if $H \to Y/G$ contracts a divisor to a point.  Indeed, we have seen
in Example~\ref{eg:ADE} that some quotients $\CC^4/G$ have
\emph{no} crepant resolutions of singularities.  Furthermore, even
when a crepant resolution exists, the $G$-Hilbert scheme may be
singular, or non-crepant, or both~\cite[Example
5.4]{Reid:LaCorrMcKay}.  In general, the following conjecture is due
to Reid.
\begin{conj}[Derived McKay Correspondence Conjecture]
  \label{conj:derived-mckay}
  For a crepant resolution of singularities $\tilde X \to \CC^n/G$,
  should one exist, there is an equivalence between $\D^b(\coh
  \tilde X)$ and $\D^b_G(\CC^n)$.
\end{conj}
Compare with Conjecture~\ref{conj:BO} above.  The derived McKay
correspondence conjecture is known when $G$ preserves a complex
symplectic form on $\CC^n$ \cite{Bezrukavnikov-Kaledin:2004}, and when
$G$ is abelian \cite{Kawamata:2006}.

Notice, for a last comment, that the ``resolution'' $\SG
\cong \End_R(S)$ of Theorem~\ref{thm:Aus-McKay} exists in any
dimension for $G \subset \GL(n,k)$ having no non-trivial
pseudo-reflections, and delivers a derived equivalence $\D^b(\SG)
\simeq \D^b_G(\CC^n)$ by definition. In dimension two, it is even
derived equivalent to the ``preferred'' desingularization
$\Hilb^G(\CC^2)$.  As we shall see in the next \S{}, it is even in a
certain sense ``crepant,'' so represents a potential improvement on
the geometric situation.

\section{Non-commutative crepant resolutions}\label{sect:NCCRs}

Now I come to the title character of this article.  It is an attempt,
due to Van den Bergh, to define a concrete
algebraic object whose derived category will realize a categorical
desingularization in the sense of Definition~\ref{def:cat-desing}, and
which will also verify Conjecture~\ref{conj:BO2}.  The main
motivations are Examples~\ref{eg:Atiyah} and~\ref{eg:Atiyah2}, and
\S\ref{sect:McKay}.

Let $R$ be a commutative ring.  Recall from
Definitions~\ref{def:order} and~\ref{def:non-sing} that a ring $\Lam$
is an $R$-order if it is finitely generated and MCM as an $R$-module,
and is non-singular if $\gldim \Lam_\p = \dim R_\p$ for all $\p \in
\Spec R$.  Let us also agree that a module-finite algebra $\Lam$ over
a domain $R$ is \emph{birational} to $R$ if $\Lam \otimes_R K \cong
M_n(K)$ for some $n$, where $K$ is the quotient field.  If $\Lam$ is
torsion-free as an $R$-module, this is equivalent to asking that $\Lam
\subseteq M_n(K)$ and that $\Lam$ spans $M_n(K)$ when scalars are
extended to $K$.  The terminology is consistent with our determination
to identify objects that are Morita equivalent; the birationality of
$\Lam$ should mean that $\Lam\otimes_R K$ is Morita equivalent to $K$,
and the only candidates are the matrix rings $M_n(K)$.

Here is a provisional definition, to be improved shortly.

\begin{provdef}
  \label{provdef}
  Let $R$ be a CM normal domain with quotient field $K$.  A
  \emph{non-commutative desingularization} of $R$ is a non-singular
  birational $R$-order $\Lam$.
\end{provdef}

There is also a natural candidate for a ``crepancy'' condition.  

\begin{definition}
  \label{def:symmetric}
  Let $R$ be a local ring, and let $\Lam$ be a
  module-finite $R$-algebra.  Let us say that $\Lam$ is a
  \emph{symmetric} $R$-algebra if $\Hom_R(\Lam,R) \cong \Lam$ as a
  $(\Lam\text{-}\Lam)$-bimodule. 
\end{definition}


Notice immediately that if $\Lam$ is a symmetric $R$-algebra, then for
any left $\Lam$-module $M$, there are  natural isomorphisms
\[
\Hom_\Lam(M,\Lam) \cong \Hom_\Lam(M,\Hom_R(\Lam,R)) \cong
\Hom_R(\Lam\otimes_\Lam M ,R) \cong \Hom_R(M,R)\,.
\]
We also have the following direct consequence of
Proposition~\ref{prop:nonsingorder}:

\begin{corollary}
  \label{cor:symm-fingldim}
  Let $R$ be a Gorenstein local ring.  If $\Lam$ is a symmetric
  $R$-order of finite global dimension, then $\gldim \Lam = \dim R$.
  In particular, $\Lam$ is non-singular.
\end{corollary}
Notice that this Corollary fails badly for non-Gorenstein $R$; a
counterexample is Example~\ref{eg:scroll} below.

Here finally is the definition~\cite{VandenBergh:flops}.

\begin{definition}
  \label{def:NCCR}
  Let $(R, \m)$ be a CM local normal domain with quotient field $K$.
  A \emph{non-commutative crepant resolution} of $R$ (or of $\Spec R$)
  is a symmetric, birational, $R$-order $\Lam$ having finite global
  dimension.
\end{definition}

I first observe that the definition is Morita-invariant, i.e.\ if
$\Lam$ and $\Gam$ are Morita-equivalent $R$-algebras and $\Lam$ is a
symmetric birational $R$-order of finite global dimension, then so is
$\Gam$. Indeed, global dimension is known to pass across Morita
equivalence.  Suppose $\Lam$ is MCM over $R$ and $\Gam = \End_\Lam(P)$
for some $\Lam$-progenerator $P$.  Since $P$ is a progenerator, $P$ is
a direct summand of $\Lam^n$ for some $n$, and it follows that $\Gam$
is a direct summand of $\End_\Lam(\Lam^n) \cong M_n(\Lam)$ as an
$R$-module.  Thus $\Gam$ is a MCM $R$-module as well.  Symmetry is
similarly easy to verify.

Before considering other possible definitions and addressing the
examples from previous sections, I point out a connection with the
classical theory of orders~\cite{AG:maximalorders,
  Reiner:MaximalOrders}, following~\cite{Iyama-Reiten:2008}.  Let $R$
be a domain with quotient field $K$.  Recall that a module-finite
$R$-algebra $\Lam$, contained in a finite-dimensional division algebra
$D$ over $K$, is called a \emph{classical order} in $D$ if $\Lam$
spans $D$ over $K$, and is called \emph{maximal} in $D$ if it is
maximal among classical orders in $D$ with respect to containment.
Maximal orders over Dedekind domains have been completely understood
for many years; the following facts are well-known.
\begin{itemize}
\item Every finite-dimensional division $K$-algebra $D$ contains a
  unique maximal order $\Delta_D$.
\item A classical order is maximal if and only if it is Morita
  equivalent to $\Delta_{D_1} \times \dots \times \Delta_{D_k}$ for
  finite-dimensional division algebras $D_1, \dots, D_k$ over $K$.
\item A classical order is \emph{hereditary}, that is, has global
  dimension at most one, if and only if it is Morita equivalent to a
  ring of the form $T_{n_1}(\Delta_{D_1}) \times \dots \times
  T_{n_k}(\Delta_{D_k})$, where $T_n(\Delta)$ denotes the subring of
  $M_n(\Delta)$ containing matrices $(a_{ij})$ with $a_{ij} \in
  \rad(\Delta)$ for $i>j$.
\end{itemize}
With these facts in mind, let $R$ be a complete discrete valuation
ring and $\Lam$ a module-finite $R$-algebra.  If $\Lam$ is a symmetric
$R$-algebra of global dimension $1$, then it follows that $\Lam$ is a
maximal order.  Indeed, $\Lam$ is hereditary, so Morita equivalent to
$T_{n_1}(\Delta_{D_1}) \times \dots \times T_{n_k}(\Delta_{D_k})$ as
above.  One can check, however, that $T_n(\Delta)$ is symmetric only
for $n=1$.  Thus $\Lam$ is maximal.

Now, a classical order $\Lam$ over a normal domain $R$ is maximal if
and only if $\Lam$ is reflexive as an $R$-module and $\Lam_\p$ is a
maximal order for all primes $\p$ of height one in
$R$~\cite[1.5]{AG:maximalorders}, \cite[11.5]{Reiner:MaximalOrders}.
Combining this with the discussion above gives the following result.

\begin{prop}
  \label{prop:NCCR-maximal}
  Let $R$ be a normal domain with quotient field $K$, and $\Lam$ a
  symmetric birational $R$-order of finite global dimension. Then
  $\Lam$ is a maximal order. \qed
\end{prop}

The connection with the classical theory of (classical) orders gives a
structure theorem for symmetric non-singular orders, via the following
results of Auslander--Goldman~\cite[Lemma 4.2]{AG:maximalorders} and
Auslander~\cite[Lemma 5.4]{Auslander:rationalsing}.

\begin{theorem}
  \label{thm:AG-A}
  Let $R$ be a  normal domain with quotient field $K$. 
  \begin{enumerate}
  \item \label{item:AG} Let $\Lam$ be a classical order over $R$ in
    $M_n(K)$.  Then $\Lam$ is a maximal order if and only if there
    exists a finitely generated reflexive $R$-module $M$ such that
    $\Lam \cong \End_R(M)$.
  \item \label{item:endosymm} Let $M$ be a reflexive $R$-module, and
    set $\Lam = \End_R(M)$.  Then $\Lam$ is reflexive as an $R$-module
    and the map
    \(
    \alpha \colon \Lam\to \Hom_R(\Lam, R)
    \)
    defined by $\alpha(f)(g) = \tr(f g)$, where $\tr\colon \End_K(K
    \otimes_R M) \to \End_R(M)$ is the usual trace map, is an
    isomorphism of $(\Lam\text{-}\Lam)$-bimodules.  Hence $\Lam$ is a
    symmetric $R$-algebra. \qed
  \end{enumerate}
\end{theorem}

Here are a few definitions which, at least under certain hypotheses,
are equivalent to Definition~\ref{def:NCCR}.  Part~(\ref{item:homhom})
of the next Proposition is the original definition of a
non-commutative crepant resolution~\cite{VandenBergh:flops,
  VandenBergh:crepant}.  For that definition, say that $\Lam$ is
\emph{homologically homogeneous} over the central subring $R$ if it is
finitely generated as an $R$-module and every simple $\Lam$-module has
the same projective dimension, equal to $\dim
R$~\cite{Brown-Hajarnavis:1984,Brown-Hajarnavis-MacEacharn:1982}.
This condition seems first to have been introduced by
Vasconcelos~\cite{Vasconcelos:1973} under the name ``moderated
algebras.''  If $R$ is equidimensional, it is equivalent to asking
that for every $\p \in \Spec R$ the localization $\Lam_\p$ is MCM as
an $R_\p$-module and $\gldim \Lam_\p = \dim
R_\p$~\cite{Brown-Hajarnavis-MacEacharn:1983}.

\begin{prop}
  \label{prop:defequiv1}
  Let $R$ be a Gorenstein local normal domain and let $\Lam$ be a
  module-finite $R$-algebra.  Then the following sets of conditions on
  $\Lam$ are equivalent.
  \begin{enumerate}
  \item \label{item:NCCR} $\Lam$ is a symmetric birational $R$-order
    and has finite global dimension.
  \item\label{item:homhom} $\Lam \cong \End_R(M)$ for some reflexive
    $R$-module $M$, and $\Lam$ is homologically homogeneous.
  \item\label{item:endo} $\Lam \cong \End_R(M)$ for some reflexive
    $R$-module $M$, $\Lam$ is MCM as an $R$-module, and $\gldim \Lam <
    \infty$. \qed
  \end{enumerate}
\end{prop}

\begin{proof}
  Assume first that $\Lam$ satisfies (\ref{item:NCCR}).  Then $\Lam$
  is MCM over $R$ by definition, and this localizes well.  By
  Theorem~\ref{thm:AG-A}(\ref{item:AG}), $\Lam$ is an endomorphism
  ring of a reflexive module $M$, and by
  Corollary~\ref{cor:symm-fingldim} $\Lam$ is non-singular, giving
  (\ref{item:homhom}).  Clearly (\ref{item:homhom}) implies
  (\ref{item:endo}).  Finally, if $\Lam \cong \End_R(M)$ for a
  reflexive $R$-module $M$, then $\Lam$ is birational to $R$, and is
  symmetric by Theorem~\ref{thm:AG-A}(\ref{item:endosymm}).
\end{proof}
The implication (\ref{item:endo}) $\implies$ (\ref{item:homhom}) fails
if $R$ is not Gorenstein. Again see Example~\ref{eg:scroll} below.


Now it is clear that Auslander's Theorem~\ref{thm:Aus-McKay} proves
that, for any $n \geq 2$ and any finite group $G \subset \SL(n,k)$
with order invertible in $k$, the ring of invariants $R = k[\![x_1,
\dots, x_n]\!]^G$ has a non-commutative crepant resolution.  Namely,
with $S$ denoting the power series ring, the endomorphism ring
$\End_R(S)$ has finite global dimension and, since $S$ is a MCM
$R$-module, is an $R$-order.  Thus $\End_R(S) \cong \SG$ is a
non-commutative crepant resolution.  One can also prove directly that
the twisted group ring $\SG$ is symmetric over $R$.

Similarly, the three-dimensional ordinary double point in
Example~\ref{eg:Atiyah2} admits the non-commutative crepant resolution
$\Lam = \End_R(R \oplus I)$, which is derived equivalent to the
resolutions of singularities $Y$ and $Y'$.

\bigskip

For the next equivalent definition we need the notion of a
\emph{$d$-Calabi-Yau} algebra. There are a few approaches to topics
with this name. I follow~\cite{Iyama-Reiten:2008, Braun:2007}; see
also~\cite{Ginzburg:Calabi-Yau, Bocklandt:2009}.  I will always
assume that the base ring is local, which eases the exposition
considerably. 

\begin{definition}
  Let $R$ be a local ring and let $\Lam$ be a module-finite
  $R$-algebra. Write $D(-) = \Hom_R(-,E)$ for Matlis duality over $R$,
  where $E$ is the injective hull of the residue field.  Say that
  $\Lam$ is \emph{$d$-Calabi-Yau ($d$-CY)} if there is a functorial
  isomorphism
  \[
  \Hom_{\D(\Lmod \Lam)}(X,Y[d]) \cong D \Hom_{\D(\Lmod \Lam)}(Y,X)
  \]
  for all $X$ and $Y$ in $\D^b(\fl \Lam)$, the bounded derived
  category of finite-length $\Lam$-modules.  Similarly, $\Lam$ is
  \emph{$d$-CY$^-$} if an isomorphism as above holds for all $X$ in
  $\D^b(\fl \Lam)$ and all $Y$ in $\K^b(\add \Lam)$.
\end{definition}

These definitions are perhaps a bit much to swallow all at once.  Here
are some basic facts about the Calabi-Yau conditions. Let $R$ be a
local ring and $\Lam$ a module-finite $R$-algebra.  Then $\Lam$ is
$n$-CY for some integer $n$ if and only if $\Lam$ is $n$-CY$^-$ and
has finite global dimension.  Indeed, if $\gldim \Lam < \infty$ then
$\D^b(\fl \Lam) \subset \D^b(\lmod \Lam) = \K^b(\add \Lam)$.  The
``only if'' part is proved by completing and considering the
finite-length $\Lam$-module $Y/\rad(\Lam)^n Y$ for $Y$ in $\K^b(\add
\Lam)$. 

Calabi-Yau algebras are best-behaved when $R$ is Gorenstein.  In that
case~\cite[Theorem 3.2]{Iyama-Reiten:2008}, if $\Lam$ is $n$-CY or
$n$-CY$^-$ for some $n$, then $n = \dim R$.  Furthermore, $\Lam$ is
$d$-CY$^-$ if and only if $\Lam$ is a symmetric $R$-order. (This is
one point where life is easier because $R$ is local.  Iyama and Reiten
give an example, which they credit to J. Miyachi, of a $d$-CY$^-$
algebra over a non-local Gorenstein ring which is not symmetric, even
though $R \to \Lam$ is injective.  It is \emph{locally} symmetric.)
More precisely, the following equivalent conditions hold.

\begin{prop}
  [{~\cite{Iyama-Reiten:2008}}]
  \label{prop:defequiv2}
  Let $(R, \m, k)$ be a Gorenstein local ring with $\dim R = d$, and
  let $\Lam$ be a module-finite $R$-algebra.  The following are
  equivalent for any integer $n$.
  \begin{enumerate}
  \item $\Lam$ is $n$-CY$^-$.
  \item As functors on $\fl \Lam$, $\Ext_\Lam^n(-,\Lam)$ is
    isomorphic to the Matlis duality functor $D(-) = \Hom_R(-,E)$,
    and $\Ext_\Lam^i(-,\Lam)=0$ for $i \neq n$.
  \item $\rHom_R(\Lam,R) \cong \Lam[n-d]$ in the bounded derived
    category of $(\Lam\text{-}\Lam)$-bimodules.
  \item $\Lam$ is a CM $R$-module of dimension $n$ and
    $\Ext_R^{d-n}(\Lam,R) \cong \Lam$ as $(\Lam\text{-}\Lam)$-bimodules.
  \end{enumerate}
  In particular, a birational module-finite algebra $\Lam$ is $d$-CY
  if and only if it is symmetric and has finite global dimension. \qed
\end{prop}

Of course the value of a definition, even one as motivated as this one
has been, is in the theorems.  Here is the main result
of~\cite{VandenBergh:flops}. 

\begin{theorem}
  \label{thm:VdBmainflops}
  Let $R$ be a Gorenstein normal $\CC$-algebra, $X = \Spec R$, and
  $\pi \colon \tilde X \to X$ a crepant resolution of singularities.
  Assume that the fibers of $\pi$ have dimension at most one.  Then
  there exists a MCM $R$-module $M$ such that the endomorphism ring
  $\Lam = \End_R(M)$ is homologically homogeneous.  In particular,
  $\Lam$ is a non-commutative crepant resolution of $R$.  Furthermore,
  $\tilde X$ and $\Lam$ are derived equivalent: $\D^b(\coh \tilde X)
  \simeq \D^b(\lmod \Lam)$. \qed
\end{theorem}

Here is a sketch of the proof of Theorem~\ref{thm:VdBmainflops}.  We
know that existence of a crepant resolution implies that $X$ has
rational singularities.  Let $\call$ be an ample line bundle on the
smooth variety $\tilde X$ generated by global sections.  Then by the
hypothesis on the fibers of $\pi$ (\cite[4.2.4]{Bondal-VdB:2003}),
$\calo_{\tilde X} \oplus \call$  generates $\D(\Qcoh
\tilde X)$, that is, if $\caln$ in $\D(\Qcoh \tilde X)$ satisfies
$\Hom_{\D(\Qcoh \tilde X)}(\calo_{\tilde X}\oplus \call, \caln[i]) =
0$ for $i \neq 0$, then $\caln = 0$ (see the discussion after
Definition~\ref{def:tilt-obj}).  Take an extension $0 \to
\calo_{\tilde X}^r \to \calm' \to \call \to 0$ corresponding to a set
of $r$ generators for $\Ext_{\calo_{\tilde X}}^1(\call, \calo_{\tilde
  X})$ as an $R$-module.  Set $\calm = \calm' \oplus \calo_{\tilde
  X}$.  Then $\calm$ also generates $\D(\Qcoh \tilde
X)$. One can show that $\Ext_{\calo_{\tilde X}}^i(\calm,\calm)=0$ for
$i >0$ (this takes a good bit of work).  Thus $\calm$ is a tilting
bundle on $\tilde X$.  Set $\Lam = \End_{\calo_{\tilde X}}(\calm)$;
then the vanishing of the derived pushforwards $\R^i
\pi_*\cEnd_{\calo_{\tilde X}}(\calm) = \Ext_{\calo_{\tilde
    X}}^i(\calm, \calm)$ implies that $\Lam \cong \End_R(M)$, where
$M = \Gam(\tilde X,\calm)$.  The proofs that $\Lam$ and $M$ are both
MCM are more involved.

Van den Bergh also proves a result converse to
Theorem~\ref{thm:VdBmainflops}, constructing a geometric crepant
resolution $\pi \colon \tilde X \to \Spec R$ from a non-commutative
one under certain assumptions~\cite[\S6]{VandenBergh:crepant}.  The
method is roughly as follows: let $\Lam$ be a non-commutative crepant
resolution of $R$, and take for $\tilde X$ a moduli space of certain
stable representations of $\Lam$. Then he proves that if
$\dim(\tilde X \times_{\Spec R} \tilde X) \leq \dim R +1$, then
$\tilde X \to \Spec R$ is a crepant resolution and there is an
equivalence of derived categories $\D^b(\coh \tilde X) \simeq
\D^b(\lmod \Lam)$.  Observe that the hypothesis is exactly similar to
that of Theorem~\ref{thm:BKR}.  In particular, the hypothesis holds if $\dim R
\leq 3$, giving the following theorem.  

\begin{theorem}
  \label{thm:VdBmain}
  Let $R$ be a three-dimensional Gorenstein normal $\CC$-algebra with
  terminal singularities.
  \begin{enumerate}
  \item\label{item:CR-NCCR} There is a non-commutative crepant
    resolution of $R$ if and only if $X = \Spec R$ has a 
    crepant resolution of singularities.
  \item\label{item:allCRs} All crepant resolutions of
    $R$---geometric as well as non-commutative---are derived
    equivalent. \qed
  \end{enumerate}
\end{theorem}

The second statement verifies Conjecture~\ref{conj:BO} of Bondal and
Orlov in this case.  Iyama and Reiten~\cite{Iyama-Reiten:2008} have
recently shown that, even without the assumption on the singularities
of $R$ being terminal, all non-commutative crepant resolutions of $R$
are derived equivalent.  Even more recently, Iyama and
Wemyss~\cite{Iyama-Wemyss:NCBondalOrlov} have announced a sufficient
criterion for the existence of a derived equivalence between the
noncommutative crepant resolutions of $R$. When $d\leq 3$ this
criterion is always satisfied, recovering the Iyama--Reiten result.

Van den Bergh suggests the following extension of
Theorem~\ref{thm:VdBmain}(\ref{item:allCRs}).
\begin{conj}
  [Van den Bergh]\label{conj:VdB}
  Let $R$ be a Gorenstein normal $\CC$-algebra and $X = \Spec R$.
  Then all crepant resolutions of $R$---geometric as well as
  non-commutative---are derived equivalent.
\end{conj}

\section{Example: normalization}\label{sect:normalization}

The appearance of endomorphism rings as ersatz resolutions of
singularities may initially be unsettling.  It does, however, have a
precedent.  One can think of the \emph{normalization} $\bar R$ of an
integral domain $R$, i.e.\ the integral closure in its quotient field,
as a partial resolution of singularities, one that is especially
tractable since it does not leave the category of noetherian rings.
This result of Grauert and Remmert~\cite{Grauert-Remmert:1971,
  Grauert-Remmert:1984} interprets the normalization as an
endomorphism ring.

\begin{theorem}
  [Grauert--Remmert]
  \label{thm:Grauert-Remmert}
  Let $R$ be an integral domain and $I$ a non-zero integrally
  closed ideal of $R$ such that $R_\p$ is normal for every $\p
  \not\supset I$.  Then the following are equivalent.
  \begin{enumerate}
  \item $R$ is normal;
  \item For all non-zero fractional ideals $J$ of $R$, $\Hom_R(J,J) = R$;
  \item For all non-zero ideals $J$ of $R$, $\Hom_R(J,J)=R$;
  \item $\Hom_R(I,I)=R$. \qed
  \end{enumerate}
\end{theorem}

For any fractional ideal $J$, the containments $R \subseteq \Hom_R(J,J)
\subseteq \bar R$ always hold.  The latter inclusion sends $\phi \colon J \to
J$ to the fraction $\phi(r)/r$ for any fixed non-zerodivisor $r \in
J$; this is well-defined.  In particular, $\Hom_R(J,J)$ is a
commutative (!) ring.

Theorem~\ref{thm:Grauert-Remmert} was used by de Jong~\cite{deJong} to
give an algorithm for computing the normalization $\bar R$ of an
affine domain over a perfect field, or slightly more generally.  Let
$R$ be a local domain such that its normalization $\bar R$ is a
finitely generated $R$-module; equivalently, the completion $\widehat
R$ is reduced.  One needs to determine a non-zero integrally closed
ideal $I$ such that $V(I)$ contains the non-normal locus of $R$.  If
$R$ is affine over a perfect field, then the Jacobian criterion
implies that the radical of the Jacobian ideal will work; there are
other choices in other cases.  Set $R' = \Hom_R(I,I)$.  If $R' =R$,
then $R$ is normal, so stop.  Otherwise, replace $R$ by $R'$ and
repeat.  The algorithm has been refined and extended
since~\cite{Decker-deJong-Greuel-Pfister:1999,
  Greuel-Laplagne-Seelisch:2010}. 

It follows from Serre's criterion for normality that if $R$ is the
coordinate ring of an irreducible curve singularity, then the
normalization $\bar R$ is regular, whence is the coordinate ring of a
resolution of singularities of $\Spec R$.  Thus in this situation,
desingularization can be achieved as an iterated endomorphism ring.
In fact, as long as $R$ is affine over a perfect field, one actually
has $\bar R = \Hom_R(\bar R,\bar R)$, a single endomorphism ring of a
finitely generated module giving resolution of singularities.

\section{MCM endomorphism rings}\label{sect:MCMendo}

The requirement that a non-commutative crepant resolution of
singularities should be an order, i.e.\ a MCM module, raises a basic
question: Does the depth of $\Hom_R(M,M)$ depend in any predictable
way on the depth of $M$?  The short answer is No.  In this \S{} we
look at some examples.

First, observe that there is at least a lower bound on the depth of
a $\Hom$ module: If $R$ is any local ring and $M$, $N$ are finitely
generated modules with $\depth N \geq 2$, then $\Hom_R(M,N)$ has depth
$\geq 2$ as well.  Indeed, applying $\Hom_R(-,N)$ to a free
presentation of $M$ displays $\Hom_R(M,N)$ as the kernel of a map
between direct sums of copies of $N$, so the Depth Lemma gives the
conclusion. That's about the end of the good news.

Next notice that the depth of $\Hom_R(M,M)$ can be strictly greater
than that of $M$.  Indeed, let $R$ be a CM normal domain and let $J$
be any non-zero ideal of $R$.  Then $\Hom_R(J,J)=R$ by
Theorem~\ref{thm:Grauert-Remmert}, even though $\depth_R J$ can take
any value between $1$ (if, say, $J$ is a maximal ideal) and $d$ (if
for example $J$ is principal).  Furthermore, $R$ can be taken to be
Gorenstein, or even a hypersurface ring, so finding a class of rings
that avoids this problem seems hopeless.

One might hope at least that if $M$ is MCM then $\Hom_R(M,M)$ is MCM
as well.  This question was raised by
Vasconcelos~\cite{Vasconcelos:1968} for $R$ a Gorenstein local ring.
It also has a negative answer, though it is at least harder.  A
counterexample is given by Jinnah~\cite{Jinnah:1975}, based
on~\cite[Example 5.9]{Hochster:CMmodules}.

\begin{example}\label{eg:Jinnah}
  Let $k$ be a field and set $A = k[x,y,z]/(x^3+y^3+z^3)$, $B =
  k[u,v]$.  Let $R$ be the Segre product of $A$ and $B$, the graded
  ring defined by $R_n = A_n\otimes_k B_n$.  Then $R$ is the subring
  of $A[u,v]$ generated by $x u,xv,y u,y v,z u,z v$, a three-dimensional
  normal domain of depth $2$.  The ideal $I = v A[u,v] \cap R$ has
  depth $3$ over $R$.

  Write $R$ as a quotient of a graded complete intersection ring $S$
  of dimension $3$.  Then $I$ has depth $3$ over $S$ as well, but
  $\Hom_S(I,I) = \Hom_R(I,I) = R$ has depth two as an $S$-module.
  Localizing $S$ at its irrelevant ideal gives a local example.
\end{example}


Here is a useful characterization of the depth of
$\Hom_R(M,N)$.

\begin{lemma}
  [{\cite{Iyama:higher-dim,Dao:2010}}]
  \label{lem:Dao}
  Let $R$ be a CM local ring and let $M$ and $N$ be finitely generated
  $R$-modules.  Fix $n \geq 2$ and consider the following properties.
  \begin{enumerate}
  \item\label{item:hom} $\Hom_R(M,N)$ satisfies ($S_{n+1}$); and
  \item\label{item:exti} $\Ext_R^i(M,N)=0$ for $i = 1, \dots, n-1$.
  \end{enumerate}
  If $M$ is locally free in codimension $n$ and $N$ satisfies ($S_n$),
  then (\ref{item:hom}) $\implies$ (\ref{item:exti}).  If $N$ satisfies
  ($S_{n+1}$) then (\ref{item:exti}) $\implies$ (\ref{item:hom}).
\end{lemma}

Since this result is used in some later sections, I'll sketch the
proof.  First assume that $M$ is locally free in codimension $n$ and
that $N$ satisfies ($S_n$).  If $n \geq \dim R$ then $M$ is free and
there is nothing to prove, so one may localize at a prime ideal of
height $n+1$ to assume by induction that $M$ is locally free on the
punctured spectrum, and so $\Ext_R^i(M,N)$ has finite length for $i
\geq 1$.  Take a free resolution $P_\bullet$ of $M$ and consider the
first $n$ terms of the complex $\Hom_R(P_\bullet, N)$.  Since the
cohomologies of this complex, namely $\Ext_R^i(M,N)$ for $i = 1,
\dots, n-1$, all have finite length, $\Hom_R(P_\bullet,N)$ is exact by
the Acyclicity Lemma~\cite[Ex.\ 1.4.23]{BH}.  
For the
second statement, take once again the free resolution $P_\bullet$ of
$M$ and consider the first $n$ terms of $\Hom_R(P_\bullet,N)$, which
form an exact sequence by the assumption.  The Depth Lemma then
implies $\depth \Hom_R(M,N)_\p \geq \min\{ n+1 , \depth N_\p\}$ for
every $\p \in \Spec R$, which gives the conclusion.

The homological consequences of Lemma~\ref{lem:Dao} are even stronger
than is immediately apparent.  To describe these, recall that module
$N$ over a commutative ring $R$ is called \emph{$\Tor$-rigid} if
whenever $\Tor_i^R(M,N)=0$ for some $i \geq 0$ and some finitely
generated $R$-module $M$, necessarily $\Tor_j^R(M,N)=0$ for all $j
\geq i$.  Deciding whether a given module is $\Tor$-rigid is generally
a delicate problem, as Dao observes~\cite{Dao:2010}.  The following
result of Jothilingam~\cite{Jothilingam:1975} (see
also~\cite{Jorgensen:2008}) gives a very useful necessary condition.

\begin{prop}
  [Jothilingam]
  \label{prop:Joth}
  Let $R$ be a local ring and let $M$, $N$ be finitely generated
  $R$-modules.  Assume that $N$ is $\Tor$-rigid.  If $\Ext_R^1(M,N)=0$,
  then the natural map $\Phi_{M,N} \colon M^* \otimes_R N \to
  \Hom_R(M,N)$ is an isomorphism.  In particular, if $\Ext_R^1(N,N)
  =0$ then $N$ is free. \qed
\end{prop}

It follows immediately that if $R$ is a local ring satisfying ($R_2$)
and ($S_3$), and $M$ is a reflexive $R$-module with a non-free direct
summand which is $\Tor$-rigid, then $\Lam = \End_R(M)$ is \emph{not}
MCM, whence is not a non-commutative crepant resolution.  Indeed, let
$N$ be a $\Tor$-rigid summand of $M$ which is not free.  Then $N$ is
reflexive, so satisfies ($S_2$), and is free in codimension two as $R$
is regular on that locus.  Moreover, $\Hom_R(N,N)$ is a direct summand
of $\Hom_R(M,M)$.  If $\Hom_R(M,M)$ were MCM, then $\Hom_R(N,N)$ would
also be, so would satisfy ($S_3$).  But then $\Ext_R^1(N,N)=0$ by
Lemma~\ref{lem:Dao}, contradicting Proposition~\ref{prop:Joth}.

It is now easy to bolster Example~\ref{eg:Jinnah} by constructing,
over any CM local ring $(R, \m, k)$ of dimension 3 or more, a MCM module $M$
such that $\Hom_R(M,M)$ is not MCM\footnote{I'm grateful to Hailong
  Dao for pointing this out to me.}.  Indeed, take $M$ to be a high
enough syzygy of the residue field $k$; since $k$ is $\Tor$-rigid, the
same is true of $M$, and it is locally free on the punctured
spectrum.  By Lemma~\ref{lem:Dao} and Proposition~\ref{prop:Joth},
then, $\Hom_R(M,M)$ has depth at most $2$.

From Proposition~\ref{prop:Joth} and progress on understanding
$\Tor$-rigid modules over hypersurface rings, Dao derives the next
theorem, which identifies obstructions to the existence of
non-commutative crepant resolutions.

\begin{theorem}
  [{\cite{Dao:2008, Dao:2010,Dao:decency}}]
  \label{thm:Dao}
  Let $R = S/(f)$ be a local hypersurface ring with $S$ an
  equicharacteristic or unramified regular local ring and $f \in S$ a
  non-zero non-unit.  Assume that $R$ is regular in codimension two.
  \begin{enumerate}
  \item If $\dim R = 3$ and $R$ is $\QQ$-factorial, then every
    finitely generated $R$-module is $\Tor$-rigid, so $R$ admits no
    non-commutative crepant resolution.
  \item\label{item:evendim} If $R$ has an isolated singularity and
    $\dim R$ is an even number greater than $3$, then $\Hom_R(M,M)$
    satisfies ($S_3$) only if $M$ is free, so $R$ admits no
    non-commutative crepant resolution. \qed
  \end{enumerate}
\end{theorem}

Recall from Example~\ref{eg:Lin} that the isolated hypersurface
singularity defined by $x_0^r+x_1^d + \dots + x_d^d=0$
has a crepant resolution of singularities only if $r\equiv 0$ or $1$
modulo $d$.  Part (\ref{item:evendim}) of Dao's theorem thus implies
that the extension of Van den Bergh's
Theorem~\ref{thm:VdBmain}(\ref{item:CR-NCCR}) to higher dimensions has
a negative answer, at least without some further assumptions.

\begin{example}
  [{\cite[Example 3.6]{Dao:2010}, \cite[\S2]{BIKR:2008}}]
  \label{eg:Dao}
  Theorem~\ref{thm:Dao} allows some progress toward deciding which of
  the three-dimensional ADE singularities (see~\eqref{eq:ADE}) have
  non-commutative crepant resolutions.  Let $k$ be an algebraically
  closed field of characteristic zero.  The $3$-dimensional versions
  of ($A_{2\ell}$), ($E_6$), and ($E_8$) are factorial, so do not
  admit a non-commutative crepant resolution at all.
  
  Let $R = k[\![x,y,z,w]\!]/(xy+z^2-w^{2\ell+2})$, an ($A_{2\ell +1}$)
  singularity, with $\ell \geq 1$.  (Observe that the case $\ell =0$
  is the ordinary double point of Example~\ref{eg:Atiyah}.)  Then I
  claim that $R$ has a non-commutative crepant resolution $\Lam
  = \End_R(M)$ in which $M$ is MCM\@.  Indeed, the indecomposable MCM
  $R$-modules are completely known~\cite[Example 5.12]{Yoshino:book};
  they are the free module $R$, the ideal $I = (x,z+w^{\ell+1})$, the
  dual ideal $I^* = (y,z-w^{\ell+1})$, and $\ell$ indecomposables
  $M_1, \dots, M_\ell$ of rank two.

  Each $M_i$ is its own Auslander-Reiten translate, $\tau M_i \cong
  M_i$, so in particular $\Ext_R^1(M_i,M_i)\neq 0$ for $i =1, \dots,
  \ell$.  By Lemma~\ref{lem:Dao}, no $M_i$ can be a constituent in a
  non-commutative crepant resolution.  On the other hand, $I$ and
  $I^*$ satisfy $\Hom_R(I,I) \cong \Hom_R(I^*,I^*) \cong R$ by
  Theorem~\ref{thm:Grauert-Remmert}.  Thus at least $\End_R(R\oplus
  I)$ and $\End_R(R \oplus I^*)$ are symmetric $R$-orders; it will
  follow from the results in the next \S{} that since $R \oplus I$ and
  $R \oplus I^*$ are cluster tilting modules
  (Theorem~\ref{thm:Iyama-dim3}), the endomorphism rings have global
  dimension equal to $3$, so are non-commutative crepant resolutions.
\end{example}

\section{Global dimension of endomorphism rings}\label{sect:cluster}

The tendency for endomorphism rings to have finite global dimension
was first observed by Auslander~\cite[\S III.3]{Auslander:QueenMary}.
Recall that $\Lam$ is an \emph{artin algebra} if the center of $\Lam$
is a commutative artin ring and $\Lam$ is a finitely generated module
over its center.

\begin{theorem}
  [Auslander]
  \label{thm:Ausrepdim}
  Let $\Lam$ be an artin algebra with radical $\r$ and assume that
  $\r^n =0$, $\r^{n-1}\neq 0$.  Set $M = \bigoplus_{i=0}^n
  \Lam/\r^{i}$.  Then $\Gam = \End_\Lam(M)$ is a coherent artin
  algebra of global dimension at most $n+1$. \qed
\end{theorem}
Based on this result Auslander was led to define the
\emph{representation dimension} of an artin algebra $\Lam$ as the
least value of $\gldim \End_\Lam(M)$ as $M$ runs through all finitely
generated $\Lam$-modules which are generators-cogenerators for $\Lam$,
that is, $M$ contains as direct summands all indecomposable projective
and injective $\Lam$-modules.  Observe that
Theorem~\ref{thm:Ausrepdim} does not prove finiteness of the
representation dimension; while $M$ has a non-zero free direct
summand, it need not be a cogenerator unless $\Lam$ is self-injective.

Auslander proved in~\cite{Auslander:QueenMary} that $\repdim \Lam \leq
2$ if and only if $\Lam$ has finite representation type
(see~\S\ref{sect:FCMT}), but it was not until 2003 that Rouquier
constructed the first examples with representation dimension greater
than $3$~\cite{Rouquier:repdim}. 
Incidentally, Rouquier's proof uses the notion of the
\emph{dimension}~\cite{Bondal-VdB:2003, Rouquier:2008} of the derived category
$\D^b(\lmod \Lam)$.  The dimension of a triangulated category is a
measure of how many steps are required to obtain it starting from a
single object and inductively taking the closure under shifts, direct
sums and summands, and distinguished triangles.  Rouquier proved that
if $\Lam$ is a finite-dimensional algebra over a perfect field $k$,
then $\dim\D^b(\lmod \Lam) \leq \repdim \Lam$.

Iyama showed in 2003~\cite{Iyama:finrepdim} that the representation
dimension of a finite-dimensional algebra is always finite. He also
extended the definition of representation dimension to CM local rings
of positive Krull dimension.  

\begin{definition}
  \label{def:repdim}
  Let $R$ be a complete CM local ring with canonical module
  $\omega$. Set
  \[
  \repdim R = \inf_M\left\{ \gldim \End_R(R \oplus \omega \oplus
    M)\right\}\,,
  \]
  where the infimum is taken over all MCM $R$-modules $M$.
\end{definition}

Iyama's techniques involved \emph{maximal $n$-orthogonal modules,} now
called \emph{cluster tilting modules}~\cite{Iyama:higher-dim}.  Here I
will not say anything about cluster algebras or cluster categories;
see~\cite{Buan-Marsh} for an exposition.  Here is a direct definition
of cluster tilting modules~\cite{BIKR:2008}.

\begin{definition}
  \label{def:clustertilting}
  Let $R$ be a CM local ring and $M$ a MCM $R$-module.   Fix $n \geq 1$.
  \begin{enumerate}
  \item Set 
    \[
    M^{\perp_n} = \left\{ X \ \middle |\ X \text{ is MCM and }
      \Ext_R^i(M,X)=0 \text{ for } 1 \leq i \leq n\right\}
    \]
    and symmetrically
    \[
    {}^{\perp_n}M = \left\{ X \ \middle|\  X \text{ is MCM and }
      \Ext_R^i(X,M)=0 \text{ for } 1 \leq i \leq n\right\}\,.
    \]
  \item Say that $M$ is \emph{cluster tilting} if
    \[
    M^{\perp_1} = \add M = {}^{\perp_1}M\,.
    \]
  \end{enumerate}
\end{definition}

There are some isolated results about cluster tilting in small dimension.  For
one example,~\cite{BIKR:2008} constructs and classifies cluster
tilting modules for the one-dimensional ADE hypersurface
singularities.  When $R$ is two-dimensional and Gorenstein, the
Auslander--Reiten translate $\tau$ is the identity, so $\Ext_R^1(M,M)$
is never zero for MCM $M$; this rules out cluster tilting in this
case.

To describe the connection between cluster tilting modules and
non-commutative crepant resolutions, let's consider the following
theorem of Dao-Huneke~\cite[Theorem 3.2]{Dao-Huneke:2010}.
\begin{theorem}
  \label{thm:D-H}
  Let $R$ be a CM local ring of dimension $d \geq 3$.  Let $M$ be a
  MCM $R$-module with a non-zero free direct summand, and set $\Lam
  = \End_R(M)$.  Assume that $\Lam$ is MCM as an $R$-module.  Consider
  the following conditions.
  \begin{enumerate}
  \item \label{DH1} $M^{\perp_{d-2}} = \add M$.
  \item \label{DH2} There exists an integer $n$ with $1 \leq n \leq
    d-2$ such that $M^{\perp_n} = \add M$.
  \item \label{DH3} $\gldim \Lam \leq d$.
  \item \label{DH4} $\gldim \Lam = d$.
  \end{enumerate}
  Then (\ref{DH1}) $\implies$ (\ref{DH2}) $\implies$ (\ref{DH3})
  $\iff$ (\ref{DH4}).  If $R$ has an isolated singularity, then all
  four are equivalent. \qed
\end{theorem}
The main assertion here is (\ref{DH2}) $\implies$ (\ref{DH3}).
Everything else is relatively straightforward or follows from
Lemma~\ref{lem:Dao}.  To prove (\ref{DH2}) $\implies$ (\ref{DH3}),
Dao and Huneke use Proposition~\ref{prop:projectivization} to get, for
any $R$-module $N$ satisfying ($S_2$), a long exact sequence
\[
\sigma: \qquad \cdots \to M^{n_{j+1}} \to M^{n_j} \to \cdots \to
M^{n_0} \to N \to 0
\]
such that $\Hom_R(M,\sigma)$ is exact.  Let $N_j$ be the kernel at the
$j^\text{th}$ spot; then one shows by induction on $j$ that
$\Ext_R^1(M,N_j) \subseteq \Ext_R^1(M,M)^{n_{j+1}}$, so that $N_{d-2}
\in M^{\perp_{d-2}} = \add M$.  It follows that $\Hom_R(M,N_{d-2})$ is
$\Lam$-projective.  Thus every $\Lam$-module of the form $\Hom_R(M,N)$
has projective dimension at most $d-2$, so that $\gldim \Lam \leq d$.

As a corollary of Theorem~\ref{thm:D-H}, Dao and Huneke obtained
another proof of the following result of Iyama, which nicely
encapsulates the significance of cluster tilting modules to
non-commutative crepant resolutions.

\begin{theorem}[Iyama{~\cite[Theorem 5.2.1]{Iyama:Aus-corr}}]
  \label{thm:Iyama-dim3}  Let $R$ be a CM local ring of dimension $d
  \geq 3$ and with canonical module $\omega$.  Assume that $R$ has an
  isolated singularity.  Let $M$ be a MCM $R$-module and set $\Lam
  = \End_R(M)$.  The following conditions are equivalent.
  \begin{enumerate}
  \item $M$ contains $R$ and $\omega$ as direct summands, $\Lam$ is
    MCM, and $\gldim \Lam = d$.
  \item $M^{\perp_{d-2}} = \add M = {}^{\perp_{d-2}}M$.
  \end{enumerate}
  In particular, if $d=3$ and $R$ is a Gorenstein isolated
  singularity, then a MCM $R$-module $M$ gives a non-commutative
  crepant resolution if and only if it is a cluster tilting module.  \qed
\end{theorem}

For dimension $3$, this result gives a very clear picture of the
landscape of non-commutative crepant resolutions.  In higher
dimension, however, the assumption of isolated singularity becomes
more restrictive.  Moreover, as Dao and Huneke observe, for $d \geq 4$
the condition $\add M = M^{\perp_{d-2}}$ rules out a large class of
examples.  Specifically, if $\Ext_R^2(M,M)=0$ for a MCM module $M$
over a complete intersection ring $R$, then $M$ is necessarily free,
since one can complete and lift $M$ to a regular local
ring~\cite{AusDingSolberg}.

Back in dimension $3$, one can obtain even stronger results, and
address possible extensions of
Theorem~\ref{thm:VdBmain}(\ref{item:CR-NCCR}), by imposing geometric
hypotheses.  Recall that a cDV singularity (see \S\ref{sect:MMP}) is a
three-dimensional hypersurface singularity defined by a polynomial
$f(x,y,z) + t g(x,y,z,t)$, where $f$ is ADE and $g$ is arbitrary.  A
cDV singularity is called $c A_n$ if the generic hyperplane section is a
surface singularity of type ($A_n$).

\begin{theorem} 
  [{\cite[Theorem 5.5]{BIKR:2008}}] Let $(R, \m)$ be a local isolated
  cDV singularity.  Then $\Spec R$ has a crepant resolution of
  singularities if and only if $R$ has a non-commutative crepant
  resolution, and these both occur if and only if there is a cluster
  tilting module in the stable category $\underline{CM}(R)$.  If $R$
  is a $c A_n$ singularity defined by $g(x,y) + z t$, then these are
  equivalent to the number of irreducible power series in a prime
  decomposition of $g(x,y)$ being $n+1$. \qed
\end{theorem}

\section{Rational singularities}\label{sect:ratsings}

As we saw in Proposition~\ref{prop:crep-ratl}, GR
Vanishing implies that any complex algebraic variety
with a crepant resolution of singularities has rational singularities.
Furthermore, the idea of a categorical, or non-commutative,
desingularization is really only well-behaved for rational
singularities.  It would therefore be most satisfactory if existence
of a non-commutative crepant resolution---a symmetric birational order
of finite global dimension---implied rational singularities.  This is
true by work of Stafford and Van den Bergh~\cite{Stafford-VdB:2008}.
Their result is somewhat more general.  Recall from the discussion
preceding Proposition~\ref{prop:defequiv1} that $\Lam$ is
homologically homogeneous if every simple $\Lam$-module has the same
projective dimension.

\begin{theorem}
  [Stafford-Van den Bergh]
  \label{thm:ratsings}
  Let $k$ be an algebraically closed field of characteristic zero, and
  let $\Lam$ be a prime affine $k$-algebra which is finitely generated
  as a module over its center $R$.  If $\Lam$ is homologically
  homogeneous then the center $R$ has rational singularities.  In
  particular, if $R$ is a Gorenstein normal affine domain and has a
  non-commutative crepant resolution of singularities, then it has
  rational singularities. \qed
\end{theorem}
Van den Bergh gave a proof of the final sentence in case $R$ is graded
in~\cite[Proposition 3.3]{VandenBergh:crepant}.  (This argument in the
published version of~\cite{VandenBergh:crepant} is not quite correct;
see the updated version online for a corrected proof.)

Here are a few comments on the proof, only in the case where $\Lam$ is
a non-commutative crepant resolution of its Gorenstein center $R$, so
is symmetric, birational and of finite global dimension.  The first
step is a criterion for rational singularities, which is an
algebraicization of the criterion $\pi_* \omega_{\tilde X} = \omega_X$
of Proposition~\ref{prop:crep-ratl}.
\begin{lemma}
  \label{lem:ratsings}
  Let $R$ be a CM normal affine $k$-algebra, where $k$ is an
  algebraically closed field of characteristic zero.  Let $K$ be the
  quotient field of $R$ and let $\omega_R$ be the canonical module for
  $R$.  Then $R$ has rational singularities if and only if for every
  regular affine $S$ with $R \subseteq S \subseteq K$, one has
  $\omega_R \subseteq \Hom_R(S,\omega_R)$ inside
  $\Hom_R(K,\omega_R)$. 
\end{lemma}

Given the lemma, the derivation of the theorem is somewhat technical.
Here I simply note that one key idea is to show (\cite[Proposition
2.6]{Stafford-VdB:2008}) that if $\Lam$ is homologically homogeneous
of dimension $d$ then $\omega_\Lam = \Hom_R(\Lam,R)$ is an invertible
$\Lam$-module, and furthermore the shift $\omega_\Lam[d]$ is a
dualizing complex for $\Lam$ in the sense of
Yekutieli~\cite{Yekutieli:1992}.  This result has been
extended~\cite[Theorem 5.1.12]{Macleod:thesis} to remove the
hypothesis of finite global dimension (so $\Lam$ is assumed to be
``injectively homogeneous'') and the hypotheses on the field $k$.

The theorem of Stafford and Van den Bergh does require an assumption
on the characteristic of $k$, as they observe~\cite[page
671]{Stafford-VdB:2008}: there is a homologically homogeneous ring in
characteristic $2$ with CM center $R$ for which $R$ fails to have
rational singularities (in any reasonable sense).  The root cause of
this bad behavior seems to be the failure of a fixed ring $S^G$ to be
a direct summand of $S$ in bad characteristic.  It is reasonable to
ask, then, as Stafford and Van den Bergh do: Suppose $\Lam$ is a
homologically homogeneous ring whose center $R$ is an affine
$k$-algebra for a field $k$ of characteristic $p >0$, and assume that
$R$ is an $R$-module direct summand of $\Lam$.  Must $R$ have rational
singularities? 

One application of Theorem~\ref{thm:ratsings} is to rule out overly
optimistic thoughts on the existence of ``generalized''
non-commutative desingularizations.  For example, one might remove the
assumption that $\Lam$ be an $R$-order and simply say that a
\emph{weak non-commutative desingularization} is an $R$-algebra $\Lam
= \End_R(M)$, where $M$ is a reflexive $R$-module, such that $\gldim
\Lam < \infty$.  One might then hope that such things exist quite
generally, for, say, every Gorenstein normal
domain~\cite{MathOverflow:ExistenceNCDesings}. However, in dimension
two this definition would coincide with that of a non-commutative
crepant resolution since endomorphism rings of reflexive $R$-modules
have depth at least two, so would only exist for rational
singularities by Theorem~\ref{thm:ratsings}.  Therefore a
counterexample to the hope would be something like
$\CC[x,y,z]/(x^3+y^3+z^3)$, which is a Gorenstein normal domain but
does not have rational singularities.

\section{Examples: finite representation type}\label{sect:FCMT}

Let $\Lam$ be an artin algebra of \emph{finite representation type},
i.e.\ there are only a finite number of non-isomorphic indecomposable
finitely generated $\Lam$-modules.  Auslander defined what is now
called the \emph{Auslander algebra} of $\Lam$ to be $\Gam
= \End_\Lam(M_1\oplus \cdots \oplus M_t)$, where $M_1, \dots, M_t$ is
a complete set of non-isomorphic indecomposable finitely generated
$\Lam$-modules.  By Corollary~\ref{cor:intertwine}, $\Gam$ is Morita
equivalent to any other algebra of the form $\End_\Lam(N)$, where $N$
is a \emph{representation generator} for $\Lam$, that is, contains
every indecomposable finitely generated $\Lam$-module as a direct
summand.  These algebras are distinguished by the following
result.

\begin{theorem}
  [Auslander~{\cite{Auslander:QueenMary}}]
  \label{thm:Aus-alg}
  Let $\Lam$ be an artin algebra of finite representation type with
  representation generator $M$.  Assume that $\Lam$ is not semisimple.
  Set $\Gam = \End_\Lam(M)$.  Then $\gldim \Gam = 2$. \qed
\end{theorem}

The proof of this theorem is quite direct from
Proposition~\ref{prop:projectivization} and the left-exactness of
$\Hom_\Lam(M,-)$.  Indeed, assume that $\Lam$ is not semisimple and
let $X$ be a finitely generated $\Gam$-module, with projective
presentation $P_1 \xto{\ \phi\ } P_0 \to X \to 0$.  The projective modules
$P_i$ are each of the form $\Hom_\Lam(M,M_i)$ for $\Lam$-modules $M_1$
and $M_0$, both in $\add M$. Similarly, $\phi = \Hom_\Lam(M,f)$ for
some $f \colon M_1 \to M_0$.  Put $M_2 = \ker f$.  Then
\[
0 \to \Hom_\Lam(M,M_2) \to \Hom_\Lam(M,M_1) \xto{\Hom_\Lam(M,f)}
\Hom_\Lam(M,M_0) \to X \to 0
\]
is a projective resolution of $X$ of length two.

Auslander and Roggenkamp~\cite{Auslander-Roggenkamp} proved a version
of this theorem in Krull dimension one, specifically for (classical)
orders over complete discrete valuation rings.  For their result,
define an order $\Lam$ over a complete DVR $T$ to have finite
representation type if there are only a finite number of
non-isomorphic indecomposable finitely generated $\Lam$-modules which
are free over $T$; these are called \emph{$\Lam$-lattices}.  If $M$
contains all indecomposable $\Lam$-lattices as direct summands, then
$\Gam = \End_\Lam(M)$ is proven to have global dimension at most two;
the proof is nearly identical to the one sketched above.  One need
only observe that the kernel $M_2$ of a homomorphism between
$\Lam$-lattices $f \colon M_1 \to M_0$ is again a $\Lam$-lattice.

In general, say that a (commutative) local ring $R$ has
\emph{finite representation type}, or finite Cohen-Macaulay type, if
there are only a finite number of non-isomorphic indecomposable
maximal Cohen-Macaulay (MCM) $R$-modules.  Recall that when $R$ is
complete, a finitely generated $R$-module $M$ is MCM if and only if it
is free over a Noether normalization of $R$.

We have already met, in \S\ref{sect:McKay}, the two-dimensional
complete local rings of finite representation type, at least over
$\CC$.  By results of Auslander and
Esnault~\cite{Auslander:rationalsing,Esnault:1985}, they are precisely
the quotient singularities $R = \CC[\![u,v]\!]^G$, where $G \subset
\GL(2,\CC)$ is a finite group.  Moreover, Herzog's
Lemma~\ref{lem:herzog} implies that in that case the power series ring
$S = \CC[\![u,v]\!]$ is a representation generator for (the MCM
modules over) $R$.  Once again the proof above applies nearly verbatim
to show (redundantly, cf.\ Proposition~\ref{prop:SGfingldim}) that
$\End_R(S)$ has global dimension two.  

In dimension three or greater, the kernel $M_2 = \ker(M_1 \to M_0)$ is
no longer a MCM module.  When the ring $R$ is CM, however, one can
replace it by a high syzygy to obtain the following result.

\begin{theorem}
  [Iyama~{\cite{Iyama:Aus-corr}},
  Leuschke~{\cite{Leuschke:finrepdim}},
  Quarles~{\cite{Quarles:thesis}}] 
  \label{thm:FCMT-gldim}
  Let $R$ be a CM local ring of finite representation type and let $M$
  be a representation generator for $R$.  Set $\Lam = \End_R(M)$.
  Then $\Lam$ has global dimension at most $\max\left\{2,\dim
    R\right\}$, and equality holds if $\dim R \geq 2$.  More
  precisely, $\pd_\Lam S =2$ for every simple $\Lam$-module
  $S$ except the one corresponding to $R$, which has projective
  dimension equal to $\dim R$. \qed
\end{theorem}

Recall that the projective module corresponding to an indecomposable
direct summand $N$ of $M$ is $P_N = \Hom_R(M,N)$, and the
corresponding simple module is $S_N = P_N/\rad P_N$.

The proof of the assertion $\gldim \Lam \leq \max\left\{2,\dim
  R\right\}$ is exactly similar to the argument sketched
above.\footnote{In the published version of~\cite{Leuschke:finrepdim},
  I gave an incorrect argument for the equality $\gldim \Lam = \dim R$
  if $\dim R \geq 2$, pointed out to me by C.~Quarles and I.~Burban.
  I claimed that if $S$ is a simple $\Lam$-module, then a
  $\Lam$-projective resolution of $S$ consists of MCM $R$-modules, so
  has length at least $\dim R$ by the depth lemma.  That's not true,
  since $\Lam$ isn't MCM\@. The equality can be rescued by
  appealing to Proposition~\ref{prop:center-finite}(\ref{Ramras}).}
For the more precise statement about the projective dimensions of the
simple modules, recall that over a CM local ring of finite
representation type, every non-free indecomposable MCM module $X$ has
an \emph{AR (or almost split) sequence}.  This is a non-split short
exact sequence of MCM modules, $0 \to Y \to E \to X \to 0$, such that
every homomorphism $Z \to X$ from a MCM module $Z$ to $X$, which is
not a split surjection, factors through $E$.  In particular, one can
show that if $M$ is a representation generator, then applying
$\Hom_R(M,-)$ to the AR sequence ending in $X$ yields the exact
sequence
\[
0 \to \Hom_R(M,Y) \to \Hom_R(M,E) \to \Hom_R(M,X) \to S_X \to 0\,,
\]
where $S_X$ is the simple $\End_R(M)$-module corresponding to $X$.  In
particular, this displays a projective resolution of $S_X$ for every
non-free indecomposable MCM module $X$.  The simple $S_R$
corresponding to $R$ is thus very special, and has projective
dimension equal to $\dim R$ by
Proposition~\ref{prop:center-finite}(\ref{Ramras}).  Observe that this
argument relies essentially on the fact that $R$ has a representation
generator; below is an example where $\pd_\Lam S > \dim R$
for a simple $S$ even though $\Lam$ has finite global dimension.

Among other things, the statement about simple modules implies that
when $\dim R \geq 3$, the endomorphism ring of a representation
generator is never homologically homogeneous, so is never a
non-commutative crepant resolution.  A concrete example of this
failure has already appeared in Example~\ref{eg:Dao}.  Here is another
example in the non-Gorenstein case.

\begin{example}
  [{\cite[Example 12]{Leuschke:finrepdim}, \cite{Smith-Quarles:2005}}]
  \label{eg:scroll}
  Let $k$ be an infinite field and let $R$ be the complete scroll of
  type $(2,1)$, that is, $R = k[\![x,y,z,u,v]\!]/I$, where $I$ is
  generated by the $2 \times 2$ minors of the matrix
  $\left(\begin{smallmatrix} x& y & u \\ y & z &
      v\end{smallmatrix}\right)$.  Then $R$ is a three-dimensional CM
  normal domain which is not Gorenstein, and has finite representation
  type~\cite{Auslander-Reiten:1989}.  The only non-free indecomposable
  MCM modules are, up to isomorphism,
  \begin{itemize}
  \item the canonical module $\omega \cong (u,v)R$;
  \item the first syzygy of $\omega$, isomorphic to $\omega^* =
    \Hom_R(\omega,R)$ and to $(x,y,u)R$;
  \item the second syzygy $N$ of $\omega$, rank two and $6$-generated;
    and
  \item the dual $L = \Hom_R(\omega^*,\omega)$ of $\omega^*$, isomorphic to
    $(x,y,z)R$.
  \end{itemize}
  By Theorem~\ref{thm:FCMT-gldim}, $\Lam = \End_R( R \oplus \omega
  \oplus \omega^* \oplus N \oplus L)$ has global dimension $3$.  However,
  $\Lam$ is not MCM as an $R$-module, since none of $L^*$, $N^*$, and
  $\Hom_R(\omega,\omega^*)$ is MCM\@.  One can check with, say,
  \texttt{Macaulay2}~\cite{M2} that $\End_R(R\oplus \omega)$ and
  $\End_R(R\oplus \omega^*)$ are up to Morita equivalence the only
  endomorphism rings of the form $\End_R(D)$, with $D$ non-free MCM\@,
  that are themselves MCM\@. In fact  $\End_R(R \oplus \omega)
  \cong \End_R(R \oplus \omega^*)$ as rings.  

  Set $\Gam = \End_R(R \oplus \omega)$.  Then $\Gam$ has two simple
  modules $S_\omega$ and $S_R$.  Using Lemma~\ref{lem:Dao} and the
  known structure of the AR sequences over $R$, Smith and
  Quarles~\cite{Smith-Quarles:2005} show that $\pd_\Gam S_\omega =4$
  and $\pd_\Gam S_R = 3$.  Thus $\Gam$ has global dimension equal to
  $4$ by Proposition~\ref{prop:center-finite}(\ref{Ramras}), but is not a
  non-commutative crepant resolution of $R$.
\end{example}

\begin{example}
  \label{eg:Veronese}
  There is only one other known example of a non-Gorenstein CM
  complete local ring of finite representation type in dimension three
  or more.  It is the (completion of the) homogeneous coordinate ring
  of the cone over the Veronese embedding $\PP^2 \into \PP^5$.
  Explicitly, set $R = \CC[\![x^2,xy,xz, y^2,yz,z^2]\!] \subset
  \CC[\![x,y,z]\!] = S$.  Then the indecomposable non-free MCM
  $R$-modules are the canonical module $\omega = (x^2,xy,xz)R$ and its
  first syzygy $N$.  Observe that $S \cong R \oplus \omega$ as
  $R$-modules, so by Theorem~\ref{thm:Aus-McKay}, $\End_R(R\oplus
  \omega) \cong S\#(\ZZ_2)$ has finite global dimension.  Since
  $\End_R(S) \cong S \oplus S$, $\Lam = \End_R(S)$ is a
  non-commutative crepant resolution for $R$.

  By Theorem~\ref{thm:FCMT-gldim}, $\Gam = \End_R(R \oplus \omega
  \oplus N)$ has global dimension $3$.  But $\Hom_R(N,R)$ and
  $\Hom_R(N,N)$ have depth $2$, so $\Gam$ is not a non-commutative
  crepant resolution.
\end{example}

\section{Example: the generic determinant}\label{sect:detX}

The most common technique thus far for constructing non-commutative
crepant resolutions has been to exploit a known (generally crepant)
resolution of singularities and a tilting object on it.  In fact, the
basic technique is already present in Van den Bergh's proof of
Theorem~\ref{thm:VdBmain}.  This has been used in several other
families of examples.  This \S{} is
devoted to describing a particular example of this technique in
action, namely the generic determinantal hypersurface ring.

Let $k$ be a field and $X = (x_{ij})$ the generic square matrix of
size $n \geq 2$, whose entries $x_{ij}$ are thus a family of $n^2$
indeterminates over $k$.  Set $S = k[X] = k[\{x_{ij}\}]$ and let $R$
be the hypersurface ring $S/(\det X)$ defined by the determinant of
$X$.  Then $R$ is a normal Gorenstein domain of dimension $n^2-1$. 

Fix a free $S$-module $\F$ of rank $n$.  Left-multiplication with the
matrix $X$ naturally defines the generic $S$-linear map $\F \to \F$.
The exterior powers $\w^a X \colon \w^a \F \to \w^a \F$ define natural
$S$-modules 
\[
M_a = \cok \w^a X
\]
for $a = 1, \dots, n$.  In fact each $M_a$ is annihilated by $\det X$,
so is naturally an $R$-module.  The pair $(\w^a X, \w^{n-a}X^T)$
forming a \emph{matrix factorization} of $\det X$, the $M_a$ are even
MCM modules over $R$~\cite{Eisenbud:1980}.  They are in particular
reflexive, of rank $\binom{n-1}{a-1}$.

Set $M = \bigoplus_{a=1}^n M_a$ and $\Lam =\End_R(M)$.  The
crucial result of~\cite{Buchweitz-Leuschke-VandenBergh:2010}, in this
case, is then

\begin{theorem}
  \label{thm:BLV}
  The $R$-algebra $\Lam$ provides a non-commutative crepant resolution
  of $R$. \qed
\end{theorem}

The proof in~\cite{Buchweitz-Leuschke-VandenBergh:2010} proceeds by
identifying the $M_a$ as geometric objects with tilting in their
ancestries, as follows.  Let $F$ be a $k$-vector space of dimension
$n$, and set $\PP = \PP(F^\svee) \cong \PP_k^{n-1}$ be the projective
space over $R$, viewed as equivalence classes $[\lambda]$ of linear
forms $\lambda\colon F \to k$.  Put
\[
Y = \PP \times \Spec S\,,
\]
with canonical projections $\tilde p \colon Y \to \PP$ and $\tilde q
\colon Y \to \Spec S$.  Identify $\Spec S$ with the space of $(n
\times n)$ matrices $A$ over the field $k$, with coordinate functions
given by the indeterminates $x_{ij}$. Then the \emph{incidence variety}
\[
Z = \left\{ ([\lambda],A) \ \middle|\  \image A \subseteq \ker
  \lambda\right\}
\]
is a resolution of singularities of $\Spec R$.  (Compare with
Example~\ref{eg:Atiyah}, which is the case $n=2$.) Indeed, the image
of $Z$ under $\tilde q \colon Y \to \Spec S$ is precisely the locus of
matrices $A$ with $\rank A < n$, that is, $\Spec R$.  Furthermore, the
singular locus of $\Spec R$ consists of the matrices of rank $< n-1$,
and $q := \tilde q |_{Z} \colon Z \to \Spec R$ is an isomorphism away
from these points.  One can explicitly write down the equations
cutting $Z$ out of $Y$, and verify that $Z$ is smooth, and is a
complete intersection in $Y$; if in particular  $j \colon Z \to
Y$ is the inclusion, then this implies that $j_* \caloz$ is resolved
over $\caloy$ by a Koszul complex on the Euler form $F \otimes_k
\caloy(-1) \to \caloy$.

Here is a pictorial description of the situation.
\begin{equation}\begin{gathered}\label{eq:detX}
\xymatrix{%
Z \ar[dr]_j \ar@/^2ex/[drr]^p \ar[dd]_{q} \\
& Y \ar[r]^{\tilde p} \ar[d]_{\tilde q} & \PP \\
\Spec R \ar@{^{(}->}[r] & \Spec S 
}
\end{gathered}
\end{equation}
Recall from \S\ref{sect:beilinson} that $T = \bigoplus_{a=1}^n
\Omega^{a-1}(a)$, where $\Omega = \Omega_{\PP/k}$ is the sheaf of
differential forms on $\PP$ and $\Omega^j = \w^j \Omega$, is a tilting
bundle on $\PP$.  Set \( \calm_a = p^* \Omega^{a-1}(a) \) for $a=1,
\dots, n$, a locally free sheaf on the resolution $Z$.  As the
typography hints, $\calm_a$ is a geometric version of $M_a$, in the
following sense.

\begin{proposition}
  \label{prop:Ma}
  As $R$-modules, 
  \(
  \R q_* \calm_a = M_a\,.
  \)
  More precisely, $\R^j q_* \calm_a = 0$ for $j>0$ and $q_* \calm_a =
  M_a$ for all $a$. \qed
\end{proposition}

The proof of the Proposition involves juggling two Koszul complexes.
Tensoring~\eqref{eq:Omegares} with $\calop(a)$ and truncating gives an
exact sequence
\[
0 \to \Omega^{a-1}(a) \to \w^{a-1}F \otimes_k \calop(1) \to \cdots \to
F \otimes_k \calop(a-1) \to \calop(a) \to 0\,.
\]
The projection $p$ being flat, the pullback $p^*$ is exact, yielding
\[
0 \to \calm_a \to \w^{a-1}F \otimes_ \caloz(1) \to \cdots \to F
\otimes_k \caloz(a-1) \to \caloz(a) \to 0\,.
\]
Compute $\R q_*$ as $\R {\tilde q}_* j_*$.  As $j_* \caloz$ is
resolved over $\caloy$ by a Koszul complex, we may replace the former
with the latter and obtain a double complex in the fourth quadrant,
with $\w^{a-1}F \otimes_k \caloy(1)$ at the origin and $\w^{a-i+1}F
\otimes_k \w^{-j}F \otimes_k \caloy(i+j+1)$ in the $(i,j)$ position.
Now apply $\R q_*$.  By~\cite[Ex. III.8.4]{Hartshorne}, the higher
direct images of the projective bundle $q \colon Y \to \Spec S$ are
completely known,
\[
\R^j {\tilde q}_* \caloy(t) = 
\begin{cases}
0 & \text{if $t < 0$ or $1 < j < n-1$;} \\
\Sym_t(F) \otimes_k S = \Sym_t(\F) & \text{for $j=0$; and} \\
0 & \text{for $j=n-1$ if $t \geq -n$.}
\end{cases}
\]
This already proves $\R^j q_* \calm_a = 0$ for $j>0$, and allows one
to represent $q_* \calm_a$ by the homology of the total complex of the
following double complex of free $S$-modules.  (For notational
simplicity write $\w^i$ and $\Sym_j$ instead of $\w^i\F$ and
$\Sym_j\F$.)
\small
\[
\xymatrix{
 & 0  & & 0 & 0 \\
 0 \ar[r] & \w^{a-1} \otimes \Sym_1  \ar[r] \ar[u] & \cdots \ar[r] &
   \w^1 \otimes\Sym_{a-1} \ar[r] \ar[u] & \Sym_a  \ar[r] \ar[u] & 0 \\
 0 \ar[r] & \w^{a-1} \otimes \w^1 \ar[r] \ar[u] & \cdots \ar[r] & \w^1 \otimes \w^1 \otimes \Sym_{a-2} \ar[r] \ar[u] & \w^1 \otimes \Sym_{a-1}  \ar[r]
   \ar[u] & 0\\
 & 0 \ar[r] \ar[u] & \cdots \ar[r] & \w^1 \otimes \w^2  \otimes
 \Sym_{a-3}  \ar[r] \ar[u] & \w^2  \otimes
 \Sym_{a-2}  \ar[r] \ar[u] & 0 \\
 & & & \vdots \ar[u] & \vdots \ar[u] \\
 & & 0 \ar[r] & \w^1 \otimes \w^{a-1}  \ar[r] \ar[u] & \w^{a-1} \otimes \Sym_1 
 \ar[r] \ar[u] & 0 \\
 & & & 0 \ar[u] & \w^a  \ar[u] \\
 & & & & 0 \ar[u] 
}
\]
\normalsize
Here the $j^\text{th}$ column is obtained by tensoring the strand of
degree $j$ in the Koszul complex with $\w^{a-j-1}\F$, so is
acyclic~\cite[A2.10]{Eisenbud:book}.  Similarly, the $(-i)^\text{th}$ row
is the degree $a$ strand in a Koszul complex tensored with $\w^{i}\F$,
and so is exact with the exceptions of the top and bottom rows.  The
top row has homology equal to $\w^a \F$ at the leftmost end, while the
bottom row has homology $\w^a \F$ on the right.  One checks from the
explicit nature of the maps that the total complex is thus reducible
to $\w^a X \colon \w^a \F \to \w^a \F$, whence $q_* \calm_a =
M_a$, as claimed.

Now it is relatively easy to prove that 
\[
\R^j q_* \cHom_{\caloz}(\calm_b, \calm_a) = 
\begin{cases}
\Hom_R(M_b,M_a) & \text{if $j=0$, and}\\
0 & \text{otherwise,}
\end{cases}
\]
and to establish the rest of the assertions in the next theorem.

\begin{theorem}
  \label{thm:BLVtilt}
  The object $\R q_* \cHom_\caloz(\calm_b,\calm_a)$ is isomorphic in
  the bounded derived category $\D^b(\lmod S)$ to a single morphism
  between free $S$-modules situated in (cohomological) degrees $-1$
  and $0$.  Therefore the $R$-module $q_*
  \cHom_\caloz(\calm_b,\calm_a)= \Hom_R(M_b,M_a)$ is a MCM $R$-module
  and the higher direct images vanish, so that in particular
  \begin{equation*}
  \R^1q_* \cHom_{\caloz}(\calm_b,\calm_a) = \Ext_R^1(M_b,M_a) =0\,.
  \qed
  \end{equation*}
\end{theorem}

It remains to see that $\bigoplus_{a=1}^n \calm_a$ is a tilting object
on $Z$, so that $\Lam = \End_R(\bigoplus_a M_a) = q_*
\cEnd_\caloz(\bigoplus_a \calm_a)$ has finite global dimension, whence
is a non-commutative crepant resolution of $R$.  It suffices for this
to compute the cohomology
\[
H^i(\PP, \cHom_\calop(\Omega^{b-1}(b), \Omega^{a-1}(a))(c))\,;
\]
since $p$ is flat, this will compute $\Ext_\caloz^i(\calm_b,\calm_a)$
as well. In~\cite{Buchweitz-Leuschke-VandenBergh:2010} we gave a
characteristic-free proof of this vanishing, and another appears in
the appendix by Weyman to~\cite{Eisenbud-Schreyer-Weyman:2003}.  In
characteristic $0$, one can compute the cohomology with Bott
vanishing~\cite[Chapter 4]{Weyman:book}.  This allows the following
proposition and theorem.

\begin{proposition}
  \label{prop:BLVtilt}
  The $\caloz$-module $\calm = \bigoplus_a \calm_a = \bigoplus_{a=1}^n
  p^* \Omega^{a-1}(a)$ is a tilting bundle in $\D^b(\coh Z)$.  In
  detail, with $\cala = \End_{\D^b(\coh Z)}(\calm)$,
  \begin{enumerate}
  \item $\Ext_\caloz^i(\calm,\calm) := \Hom_{\D^b(\coh
      Z)}(\calm,\calm[i]) =0$ for $i >0$;
  \item $\rHom_\caloz(\calm,-) \colon \D^b(\coh Z) \to
    \D^b(\lmod {\cala})$ is an  equivalence of
    triangulated categories, with $-\lotimes_{\cala} \calm $ as inverse;
  \item $\cala$ has finite global dimension.
  \item $\cala \cong \Lam = \End_R(M)$. \qed
  \end{enumerate}
\end{proposition}

\begin{theorem}
  \label{thm:BLV-nccr}
  Let $k$ be a field, $X$ an $(n \times n)$ matrix of indeterminates,
  $n \geq 2$, and $R = k[X]/(\det X)$ the generic determinantal
  hypersurface ring.  Let $M_a = \cok \w^a X$ for $a=1, \dots, n$, and
  put $M = \bigoplus_a M_a$. Then the $R$-algebra $\Lam = \End_R(M)$
  has finite global dimension and is MCM as an $R$-module.  It is in
  particular a non-commutative crepant resolution of $R$. \qed
\end{theorem}

In~\cite{Buchweitz-Leuschke-VandenBergh:2010} we replace the square
matrix $X$ by an $(m \times n)$ matrix with $n \geq m$ and $R$ with
the quotient by the maximal minors $k[X]/I_m(X)$, which defines the
locus in $\Spec k[X]$ of matrices with non-maximal rank.  The same
construction $M_a = \cok \w^a X$ yields an algebra $\Lam
= \End_R(\bigoplus_{a=1}^m M_a)$ which is still MCM as an $R$-module
and still has finite global dimension.  In this case, however, $\Lam$
is not a non-singular  $R$-algebra, so not a non-commutative crepant
resolution according to our definition.  This is directly attributable
to the fact that quotients by minors are Gorenstein if and only if
$n=m$, so that Corollary~\ref{cor:symm-fingldim} fails for non-square
matrices.  In a forthcoming
paper~\cite{Buchweitz-Leuschke-VandenBergh:crepresdetX2}, we establish
the same result for the quotient by arbitrary minors $k[X]/I_t(X)$,
with $1 \leq t \leq m$, using a tilting bundle on the
Grassmannian~\cite{Buchweitz-Leuschke-VandenBergh:tiltGrass}.  In
particular we obtain non-commutative crepant resolutions when the
matrix is square.

Similar techniques, i.e.\ constructions using tilting objects on known
resolutions of singularities, are used by
Kuznetsov~\cite{Kuznetsov:2008} to give non-commutative
desingularizations for several more classes of examples, including
cones over Veronese/Segre embeddings and Grassmannians, as well as
Pfaffian varieties.  

\section{Non-commutative blowups}\label{sect:blowups}

It was clear from early on in the development of non-commutative
(projective) geometry that it would be most desirable to have a
non-commutative analogue of the most basic birational transformation,
the blowup.  This \S{} sketches a few approaches to the problem.  

First recall that if $R$ is a commutative ring and $I$ is an ideal of
$R$, the \emph{blowup} $\Bl_I(X)$ of $X = \Spec R$ at (the closed
subscheme defined by) $I$ is $\Proj R[It]$, where $R[It]$ is the Rees
algebra $R \oplus I \oplus I^2 \oplus \cdots$.  The exceptional locus
of the blowup is the \emph{fiber cone} $\Proj R[It]/IR[It] = R/I
\oplus I/I^2 \oplus \cdots$.

One might hope to mimic this definition for sufficiently nice
non-commutative rings.  This turns out to give unsatisfactory results.
For example~(\cite{Artin:1997}, \cite[page 2]{VdB:memoir}) set $\Lam =
k\langle x,y\rangle /(yx-xy-y)$, and consider the ideal $\m = (x,y)$
``corresponding'' to the origin of this non-commutative surface.  Then
$\m^n = (x^n,y)$ for all $y$, so the fiber cone $R[\m t]/\m R[\m t]$
is one-dimensional in each degree, and is isomorphic to $k[z]$.  This
means that the exceptional locus is in some sense zero-dimensional,
whereas one should expect the exceptional divisor of a blowup of a point
in a surface to have dimension $1$.

Van den Bergh~\cite{VdB:memoir} constructs an analogue of the Rees
algebra directly over projective quasi-schemes $\ncProj \Lam$ (see
\S\ref{sect:non-existence}).  Specifically, if $X = \ncProj \Lam$ is a
quasi-scheme, he gives a construction of the blowup of a smooth point
$p$ in a commutative curve $Y$ contained in $X$.  (This means that
$\Qcoh Y \simeq \ncProj(\Lam/x\Lam)$ for some $x \in \Lam$.)   Using
this construction, Van den Bergh considers blowups of quantum
projective planes at small numbers of points, in particular
non-commutative deformations of the del Pezzo surfaces obtained by
blowing up in $\leq 8$ points.  I won't go into the details of the
construction or the applications here.

There is a more recent proposal for a definition of the phrase
``non-commutative blowup,'' which is inspired by the classic flop of
Example~\ref{eg:Atiyah2} and by Theorem~\ref{thm:Grauert-Remmert}.  In
general, the idea is that for an ideal $I$ of a ring $\Lam$, the
\emph{non-commutative blowup} of $\Lam$ in $I$ is the ring
\[
\Bl^{\nc}_I(\Lam) = \End_\Lam(\Lam \oplus I)\,.
\]
In the situation of Example~\ref{eg:Atiyah2}, we saw that
$\Bl^{\nc}_I(R)$ was derived equivalent to the usual blowup $\Bl_I(\Spec
R)$.  Thus suggests the following question, a version of which I first
heard from R.-O.~Buchweitz.
\begin{question}\label{q:ROB}
  Can one generalize or imitate the normalization algorithm of
  \S\ref{sect:normalization} to show that there is a sequence of
  non-commutative blowups starting with $\Lam = \Lam_0$ and continuing
  with $\Lam_{i+1} = \Bl^{\nc}_{I_i}(\Lam_i) = \End_{\Lam_i}(\Lam_i
  \oplus I_i)$ for some ideals $I_i \subset \Lam_i$, such that $\Lam_i$
  eventually has finite global dimension?
\end{question}
One might try to follow Hironaka and blow up only in ``smooth
centers,'' i.e.\ assume that $\Lam_i/I_i$ is non-singular, as is the
case in Example~\ref{eg:Atiyah2}.

Very recent work of Burban--Drozd~\cite{Burban-Drozd:tilting} confirms
that the non-commutative blowup as above is a sort of categorical
desingularization whenever $X$ is a reduced algebraic curve
singularity having only nodes and cusps for singular points, and $I$
is the conductor ideal.  They observe that $\cala =
\cEnd_\calox(\calox \oplus I)$ has global dimension equal to $2$, that
$(\calox \oplus I)\otimes_\calox - \colon \coh X \to \lmod \cala$ is
fully faithful, and that $\Hom_\cala(\calox \oplus I,-) \colon \lmod
\cala \to \coh X$ is exact.

In his Master's thesis, Quarles constructs a direct connection between
blowups and non-commutative blowups, the only one I know of.  Let $(R,
\m, k)$ be a Henselian local $k$-algebra, with $k$ an algebraically
closed field.  Let $I$ be an ideal of $R$ which is MCM and reflexive
as an $R$-module, and set $\Lam = \Bl_{I^*}^{\nc}(R) = \End_R(R \oplus
I^*)$.  Then Quarles defines~\cite[Section 7]{Quarles:thesis} a
bijection between the closed points of $\Bl_I(\Spec R) = \Proj R[It]$
and the set of indecomposable $\Lam$-modules $X$ arising as extensions
$0 \to S_R \to X \to S_{I^*} \to 0$ of the two simple modules $S_R$
and $S_I$.  The bijection is just as sets, and carries no known
algebraic information; in particular, it is not known to be a moduli
space.

There are some immediate problems.  For example, in
Example~\ref{eg:scroll} we have $\End_R(R\oplus \omega) \cong \End_R(R \oplus
\omega^*)$, but $R[\omega t] \not \cong R[\omega^* t]$, since one is
regular and the other is not.  The associated projective schemes
are isomorphic, of course.  It is not clear how to reconcile this.

A similar approach has been suggested in prime
characteristic~\cite{Toda:2009, Toda:StabAndCY,
  Toda:BirationalCY3folds, Toda-Yasuda:2009, Yasuda:2009}.  For the
rest of this \S{} let $k$ be an algebraically closed field of
characteristic $p>0$.  Let $X$ and $Y$ be normal algebraic schemes
over $k$, and let $f \colon Y \to X$ be a finite dominant morphism.
  Then Yasuda~\cite{Yasuda:2009} proposes
to call the endomorphism ring $\cEnd_\calox(f_* \caloy)$ the
non-commutative blowup attached to $f$.

In particular, consider the non-commutative blowup of the
Frobenius. For every $e \geq 1$, set $X_e=X$ and let $F_X^e \colon X_e
\to X$ be the $e^\text{th}$ iterate of the Frobenius morphism.  Assume
that $F_X$ is finite. Then the non-commutative blowup of the
$e^\text{th}$ Frobenius, $\cEnd_\calox({F_X^e}_* \calo_{X_e})$ is
locally given by $\End_R(R^{1/p^e})$, where $R^{1/p^e}$ is the ring of
$(p^e)^\text{th}$ roots of elements of $R$.  It is isomorphic to
$\End_{R^{p^e}}(R)$, where now $R^{p^e}$ is the subring of
$(p^e)^\text{th}$ powers.  The ring $\End_{R^{p^e}}(R)$ consists of
differential operators on $R$~\cite{Smith-VandenBergh:1997} and is
sometimes a non-commutative crepant resolution of $R$.

\begin{theorem}
  [Toda--Yasuda~\cite{Toda-Yasuda:2009}] Let $R$ be a complete local
  ring of characteristic $p$ which is one of the following.
  \begin{enumerate}
  \item\label{item:dimone} a one-dimensional domain;
  \item\label{item:A1} the ADE hypersurface singularity of type
    ($A_1$) (and $p \neq 2$); or
  \item\label{item:tame} a ring of invariants $k[\![x_1, \dots,
    x_n]\!]^G$, where $G \subset \GL(n,k)$ is a finite subgroup with
    order invertible in $k$.
  \end{enumerate}
  Then for $e \gg 0$, $\End_{R}(R^{1/p^e})$ has finite global
  dimension.  However it is not generally MCM as an $R$-module, so is
  not a non-commutative crepant resolution.  \qed
\end{theorem}
Let me make a few comments on the proofs.  For (\ref{item:dimone}),
consider the integral closure $S \cong k[\![x]\!]$ of $R$.  Then for
any $e \geq 1$, one checks that $\End_{R^{p^e}}(S^{p^e})
= \End_{S^{p^e}}(S^{p^e}) = S^{p^e}$.  Take $e$ large enough that
$S^{p^e} \subseteq R$.  Then $R$ is free over $S^{p^e}$ of rank $p^e$,
so $\End_{R^{p^e}}(R) \cong M_{p^e}(\End_{R^{p^e}}(S^{p^e})) =
M_{p^e}(S^{p^e})$.  This is Morita equivalent to $S^{p^e} \cong S$, so has
global dimension equal to $1$.  It is also clearly MCM.

For (\ref{item:A1}), assume $p \neq 2$ and set $R = k[\![x_1, \dots,
x_d]\!]/(x_1^2 + \cdots + x_d^2)$.  Then one can show that for all $e
\geq 1$, $R$ is a representation generator for $R^{p^e}$.  (This
requires separate arguments for $d$ odd/even.)  By
Theorem~\ref{thm:FCMT-gldim}, $\End_{R^{p^e}}(R)$ has finite global
dimension.  It is not a non-commutative crepant resolution by
Theorem~\ref{thm:Dao}.

Finally, for (\ref{item:tame}), Toda and Yasuda use results of Smith
and Van den Bergh to show that if $S = k[\![x_1, \dots, x_d]\!]$ and
$R=S^G$ as in the statement, then for $e \gg 0$ every module of
covariants $(S \otimes_k W)^G$ appears as an $R^{p^e}$-direct summand
in $R$, in $S$, and in $S^{p^e}$.  Thus $\End_{R^{p^e}}(S^{p^e}) \cong
S^{p^e} \# G$ (Theorem~\ref{thm:Aus-McKay}) is Morita equivalent to
$\End_{R^{p^e}}(R)$ by Corollary~\ref{cor:intertwine}, and they
simultaneously have finite global dimension.

In general, there are non-trivial obstructions to $\End_R(R^{1/p^e})$
being a non-commutative crepant resolution.  For example, Dao points
out~\cite{Dao:2010} that when $R$ is a complete intersection ring,
$R^{1/p^e}$ is known to be $\Tor$-rigid~\cite{Avramov-Miller:2001}, so
if $R$ satisfies ($R_2$) then $\End_R(R^{1/p^e})$ is not MCM for any
$e\geq 1$ by the discussion following Proposition~\ref{prop:Joth}.

\section{Omissions and open questions}\label{sect:omissions}

In addition to the examples already mentioned in previous \S\S, there
is a large and growing array of examples of non-commutative crepant
resolutions and related constructions.  Lack of space and expertise
prevent me from describing them in full, but here are a few references
and comments.

Deformations of the Kleinian singularities $\CC^2/G$, with $G
\subset \SL(2,\CC)$, have non-commutative crepant
resolutions~\cite{Gordon-Smith:2004}, which are identified as deformed
preprojective algebras in the sense
of~\cite{CrawleyBoevey-Holland:1998}.   

In a different direction, Wemyss has considered the non-Gorenstein
case of the classical McKay correspondence, where $G \not \subset
\SL(2,\CC)$~\cite{Wemyss:An, Wemyss:Dn-I, Wemyss:Dn-II}; much of
Theorem~\ref{thm:McKay2} breaks down, but much can be recovered by
restricting to the so-called ``special'' representations.  This leads
to the \emph{reconstruction algebra,} which is the endomorphism ring
of the special MCM modules.  

Beil~\cite{Beil:2008} shows that \emph{square superpotential
  algebras,} which are certain quiver algebras with relations coming
from cyclic derivatives of a superpotential, are non-commutative
crepant resolutions of their centers (which are three-dimensional
toric Gorenstein normal domains). In fact,
Broomhead~\cite{Broomhead:2009} constructs a non-commutative crepant
resolution for every Gorenstein affine toric threefold, from
superpotential algebras called dimer models.  Similar algebras
associated to brane tilings have non-commutative crepant resolutions
as well~\cite{Mozgovoy:2009a,Mozgovoy:2009b}.

Finally, Bezrukavnikov~\cite{Bezrukavnikov:2006} constructs a
non-commutative version of the Springer resolution $Z$
from~\eqref{eq:detX}, which is different from that
in~\cite{Buchweitz-Leuschke-VandenBergh:2010}. 

\bigskip

Many other topics have been omitted that could have played a role.
For example, I have said nothing about (semi-)orthogonal
decompositions of triangulated categories and exceptional sequences.
These grew out of Be\u\i linson's result in \S\ref{sect:beilinson},
via the Rudakov seminar~\cite{Rudakov:1990}.  See~\cite[Chapter
1]{Huybrechts} or~\cite{Bondal-Orlov:1995}.

Connections of this material with string theory appear at every
turn~\cite{Seiberg-Witten:1999}. For example, the derived category
$\D(X)$ appears in string theory as the category of branes propagating
on the space $X$.  Non-commutativity arises naturally in this context
from the fact that \emph{open} strings can be glued together in two
different ways, unlike \emph{closed}
strings~\cite{Berenstein-Leigh:2001}.  Furthermore, the Calabi-Yau
condition of \S\ref{sect:NCCRs} is essential to the string-theoretic
description of spacetime~\cite{Broomhead:2009, Ooguri-Yamazaki:2009}.
Most obviously, high energy physics has been a driving force in
non-commutative desingularizations and the higher geometric McKay
correspondence.  I am not competent to do more than gesture at these
connections.

\bigskip

I end the article with a partial list of open problems.  Some of these
are mentioned in the text, while others are implicit.

\begin{enumerate}[(1)]\addtolength{\itemsep}{0.5\baselineskip}
\item Conjecture~\ref{conj:BO} of Bondal and Orlov, that a generalized
  flop between smooth varieties induces a derived equivalence, is
  still largely open outside of dimension three.  The related
  Conjecture~\ref{conj:VdB} of Van den Bergh, which asks for derived
  equivalence of both geometric and non-commutative crepant
  resolutions, is similarly
  open. See~\cite{Iyama-Wemyss:NCBondalOrlov} for some very recent
  progress on the non-commutative side.
\item Existence of a non-commutative crepant resolution is \emph{not}
  equivalent to existence of a crepant resolution of singularities in
  dimension four or higher.  See the end of Example~\ref{eg:ADE} for
  examples with non-commutative resolutions but no geometric ones, and
  Example~\ref{eg:Dao} for failure of the other direction.  However,
  it still may hold in general in dimension three.  One might also be
  optimistic and ask for additional hypotheses to rescue the case of
  dimension four.
\item Various results in the text fail for rings that are not
  Gorenstein, notably Corollary~\ref{cor:symm-fingldim} and
  Proposition~\ref{prop:defequiv1}.  Is there a better definition of
  non-commutative crepant resolutions which would satisfy these
  statements over non-Gorenstein Cohen-Macaulay rings?  Perhaps we
  should not expect one, since crepant resolutions of singularities
  exist only in the Gorenstein case.

  On a related note, is symmetry (Definition~\ref{def:symmetric}) too
  strong a condition?  The relevant property
  in~\cite{Stafford-VdB:2008} is that $\Hom_R(\Lam,R)$ is an
  \emph{invertible} $(\Lam\text{-}\Lam)$-bimodule, rather than insisting that
  $\Hom_R(\Lam,R) \cong \Lam$.  This would, unfortunately, rule out
  endomorphism rings $\End_R(M)$, since they are automatically
  symmetric by Theorem~\ref{thm:AG-A}(\ref{item:endosymm}).  Or
  perhaps the appropriate generalization to non-Gorenstein rings is
  that $\Hom_R(\Lam,\omega_R) = \Lam$.
\item Crepant resolutions of singularities are very special: they
  exist only for canonical singularities, not in general for terminal
  singularities.  The non-commutative version is  more
  general.  One might therefore hope that Theorem~\ref{thm:VdBmain} is
  true for canonical threefolds as well.  Van den Bergh's proof of
  Theorem~\ref{thm:VdBmain} applies verbatim for any canonical
  threefold admitting a crepant resolution of singularities with
  one-dimensional fibers.
\item In nearly all of the examples of non-commutative crepant
  resolutions, the module $M$ such that $\Lam = \End_R(M)$ can be
  taken maximal Cohen-Macaulay.  Lemma~\ref{lem:Dao} indicates one
  obstruction to $M$ having high depth.  Are there general situations
  where a non-commutative crepant resolution exists, but no MCM module
  will suffice?  Or situations (other than surfaces) where every
  non-commutative crepant resolution is given by a MCM module?
  See~\cite[5.12]{Iyama-Wemyss:ARduality} for one result in this
  direction.
\item Van den Bergh points out in~\cite{VandenBergh:crepant} that one
  might try to build a theory of rational singularities for
  non-commutative rings, extending the results of
  \S\ref{sect:ratsings}.  It would be essential to have a
  non-commutative analogue of the Grauert--Riemenschneider Vanishing
  theorem (\ref{thm:GRvan}), but none seems to be known.  There is an
  algebraic reformulation of GR Vanishing due to Sancho de
  Salas~\cite{SanchodeSalas}, cf.~\cite[Chapter 5]{Huneke:CBMS}: Let
  $R$ be a reduced CM local ring essentially of finite type over an
  algebraically closed field of characteristic zero, and let $I$ be an
  ideal of $R$ such that $\Proj R[It]$ is smooth; then the associated
  graded ring $\gr_{I^n}(R)$ is Cohen--Macaulay for $n \gg 0$. It
  would be very interesting to have a purely algebraic proof of this
  result, particularly if it encompassed some non-commutative rings.
  The proof of Sancho de Salas uses results
  from~\cite{Grauert-Riemenschneider}, so relies on complex analysis;
  see~\cite{Huckaba-Marley:1999} for some progress toward an algebraic
  proof in dimension two.
\item\label{item:ncblow} In a similar direction, Question~\ref{q:ROB}
  asks for an algorithm to resolve singularities via a sequence of
  ``non-commutative blowups.'' For a start, one needs any non-trivial
  connection between $\D^b(\coh \Proj R[It])$ and $\D^b(\lmod
  {\End_R(R \oplus I)})$; other than Quarles' bijection, none seems to
  be known.
\item Even given a very strong result along the lines of
  (\ref{item:ncblow}), an enormous amount of work would still be
  needed to obtain applications of non-commutative desingularizations
  analogous to those of resolutions of singularities.  For example,
  can one define an ``arithmetic genus'' in a non-commutative context,
  and show, as Hironaka does, that it is a ``birational'' invariant?  
\end{enumerate}


\nocite{Rouquier:2006}
\nocite{Orlov:QcohCommNoncomm,
  Orlov:DerCatsCohTriang,Orlov:DerCatsCohEquivs, Orlov:RemsGens}
\nocite{Holm-Jorgensen:2010}

\bibliographystyle{amsalpha}
\renewcommand{\baselinestretch}{1}

\newcommand{\arxiv}[2][AC]{\mbox{\href{http://arxiv.org/abs/#2}{\textsf{arXiv:#2[math.#1]}}}}
\newcommand{\oldarxiv}[2][AC]{\mbox{\href{http://arxiv.org/abs/math/#2}{\textsf{arXiv:math/#2[math.#1]}}}}
\renewcommand{\MR}[1]{%
  {\href{http://www.ams.org/mathscinet-getitem?mr=#1}{MR #1}}}
\providecommand{\bysame}{\leavevmode\hbox to3em{\hrulefill}\thinspace}
\newcommand{\arXiv}[1]{%
  \relax\ifhmode\unskip\space\fi\href{http://arxiv.org/abs/#1}{arXiv:#1}}


\def\cprime{$'$} \def\cprime{$'$} \def\cprime{$'$}
\providecommand{\bysame}{\leavevmode\hbox to3em{\hrulefill}\thinspace}
\providecommand{\MR}{\relax\ifhmode\unskip\space\fi MR }
\providecommand{\MRhref}[2]{%
  \href{http://www.ams.org/mathscinet-getitem?mr=#1}{#2}
}
\providecommand{\href}[2]{#2}

\end{document}